\documentclass[leqno,11pt]{amsart}
\usepackage{amssymb, amsmath,amsmath,latexsym, amssymb, amsfonts, amsbsy, amsthm, mathtools, graphicx, CJKutf8,CJKnumb, CJKulem, color,cite,ulem,bbm}
\usepackage{comment}
\usepackage{sim_setup}
\newcommand{\myh}[1]{}
\usepackage{tikz}
\tikzstyle{startstop} = [rectangle, rounded corners, minimum width=1cm, minimum height=1cm,text centered, draw=black]
\tikzstyle{io} = [rectangle, rounded corners, minimum width=1cm, minimum height=1cm,text centered, draw=black]
\tikzstyle{process} = [rectangle, rounded corners, minimum width=1cm, minimum height=1cm,text centered, draw=black]
\tikzstyle{decision} = [rectangle, rounded corners, minimum width=1cm, minimum height=1cm,text centered, draw=black]
\tikzstyle{explain} = [minimum width=1cm, minimum height=0.5cm,text centered, draw=black]
\tikzstyle{equations} = []
\tikzstyle{arrow} = [thick,->,>=stealth]
\tikzstyle{d-arrow} = [thick,->,dashed,>=stealth]
\usetikzlibrary{shapes.geometric, arrows,plotmarks,backgrounds}
\usepackage{pgfplots}
\usepackage{tikz}
\usepackage{hyperref}
\usepackage{cancel}
\numberwithin{equation}{section}

\allowdisplaybreaks

\let\al=\alpha
\let\b=\beta
\let\g=\gamma
\let\d=\delta
\newcommand{\e}{\mathrm{e}} 
\let\ep=\epsilon

\let\la=\lambda

\let\ka=\kappa
\let\s=\sigma
\let\f=\frac
\let\p=\psi
\let\om=\omega

\let\D=\Delta

\let\wt=\widetilde

\let\na=\nabla

\let\pa=\partial
\let\lf=\left
\let\rg=\right
\let\pa=\partial
\newcommand{\nq}{{\neq}}

\newcommand{\step}[1]{\noindent{\bf Step \# #1}}

\newcommand{\exb}[1]{{\exp\left\{#1\right\}}}
\newcommand{\br}[1]{{\left\langle #1 \right\rangle}}
\newcommand{\w}[2]{#1^{(#2)}}%

\usepackage{scalerel,stackengine}
\stackMath
\newcommand\reallywidehat[1]{%
\savestack{\tmpbox}{\stretchto{%
  \scaleto{%
    \scalerel*[\widthof{\ensuremath{#1}}]{\kern-.6pt\bigwedge\kern-.6pt}%
    {\rule[-\textheight/2]{1ex}{\textheight}}
  }{\textheight}%
}{0.5ex}}%
\stackon[1pt]{#1}{\tmpbox}%
}

\def\curl{\mathop{\rm curl}\nolimits}

\newcommand{\oml}{\omega ^{(\mathrm L)}}

\newcommand{\beq}{\begin{equation}}
\newcommand{\eeq}{\end{equation}}
\newcommand{\ben}{\begin{eqnarray}}
\newcommand{\een}{\end{eqnarray}}
\newcommand{\beno}{\begin{eqnarray*}}
\newcommand{\eeno}{\end{eqnarray*}}

\newcommand{\T}{\mathbf{T}^{(\kappa)}}

\newcommand{\msc}{\mathscr}

\renewcommand{\Re}{\mathrm{Re}}
\renewcommand{\Im}{\mathrm{Im}}

\newtheorem{theorem}{Theorem}[section]
\newtheorem{definition}[theorem]{Definition}
\newtheorem{lemma}[theorem]{Lemma}
\newtheorem{proposition}[theorem]{Proposition}
\newtheorem{corol}[theorem]{Corollary}
\newtheorem{remark}[theorem]{Remark}

\title[When Couette meets Robin]{Couette Flow with Robin Boundary Condition (I):\\
the viscosity-independent friction}

\author{Siming He}
\address[S. He]{Department of Mathematics, University of South Carolina, Columbia, USA}
\email{siming@mailbox.sc.edu}
\author{Binqian Niu}
\address[B. Niu]{
School of Mathematics and Statistics, Central China Normal University, Wuhan, 430079, P. R. China.}
\email{bqbling@163.com}
\author{Weiren Zhao}
\address[W. Zhao]{Department of Mathematics, New York University Abu Dhabi, Saadiyat Island, P.O. Box 129188, Abu Dhabi, United Arab Emirates.}
\email{zjzjzwr@126.com, wz19@nyu.edu}

\begin{document}
\begin{abstract}
    This article is the first paper in the series. In this series of articles, we will examine the influence of the friction factor $\alpha$ at the solid--fluid boundary on the stability of Couette flow. Specifically, we consider the stability of Couette flow in a bounded periodic channel \( \mathbb{T} \times [-1,1] \) under Robin-type boundary conditions (\( u^2|_{y=\pm 1} = 0 \), \( [\alpha \partial_n u^1 + u^1]|_{y=\pm1} = f \)), where \(\alpha\) is the friction factor and \(n\) is the unit outer normal vector. In this article, we prove that for a given friction factor \(\alpha\), as long as the fluid viscosity coefficient \(\nu\ll \alpha\) is sufficiently small, the system is asymptotically stable if the initial perturbation satisfies \(\|\omega_{\rm in}\| \leq \epsilon \nu^{1/3}\). Moreover, inviscid damping and enhanced dissipation hold. 
\end{abstract}
\maketitle
\section{Introduction}
Consider the two-dimensional incompressible Navier-Stokes equations near the Couette flow in the bounded periodic channel. Let the net velocity be $(y,0)+v$, then the perturbation $v=(v^1,v^2)$ solves: 
\begin{equation}\label{eq-initial}
    \begin{cases}
        \pa_t v -\nu\D v +y\pa_x v +(v^2,0)+ v \cdot\nabla v +\nabla p=0,\\
        \nabla\cdot v=0,\\
      v^2\big|_{y=\pm 1}=0,\quad  \big((1-\kappa)\pa_{{n}}v^1+\kappa v^1\big)\big|_{y=\pm 1}=0,\\
         v |_{t=0}= v _{\rm in}(x,y),\quad x\in \mathbb{T}=(\rr/2\pi\mathbb{Z}),\quad  y\in [-1,1].
    \end{cases}
    \end{equation} 
Here the parameter $\nu>0$ is the \textit{viscosity} (or the inverse Reynolds number ${\rm Re^{-1}}$), and the scalar function $p$ represents the pressure that ensures the divergence-free condition.  
The horizontal velocity satisfies Robin-type boundary conditions, while the vertical velocity is impermeable on the solid walls $y=\pm1$. Such a kind of boundary conditions can be either derived by phenomenological arguments, or deduced from a boundary condition for the Boltzmann equations, due to the atomic structure of the wall, see \cite{aoki1979slightly, bardos2012incompressible, masmoudi2003boltzmann}. Here, the symbol $\partial_{n}:= n\cdot \nabla$ denotes the derivative in the outer normal direction, and the parameter $\kappa$ takes values in $[0,1]$. If  $\kappa=1$, this condition reduces to the classical non-slip boundary condition; and if $\kappa=0$, it becomes the Navier/Lion-slip boundary condition. In this paper, we focus on the regime where the parameter $\kappa\in(0,1)$ is fixed and the \textit{friction factor} $\alpha:=\frac{1-\kappa}{\kappa}$ is strictly positive and finite. Finally, in this paper, the viscosity $\nu $ is small compared to the friction factor.

Since Reynolds' pioneering work, the stability of hydrodynamic flows at high Reynolds numbers has been a major focus of research. Experiments show that as the Reynolds number increases, steady laminar flow typically becomes unstable and transitions to turbulence. Mathematically, the linear stability phenomenon of the Couette flow has been studied in the pioneering works of Lord Kelvin \cite{Kelvin87}, Rayleigh \cite{Rayleigh80}, Orr \cite{Orr07}, and Sommerfeld \cite{Sommerfeld1908}. 
Kelvin 
\cite{Kelvin87} obtained the explicit solutions to the linearized Navier-Stokes equations around the Couette flow. He also suggested
that the stability/instability of the system is related to the size of the perturbation, and the threshold size is decreasing as $\nu\to 0$. 

Later on, Fourier analytic techniques were applied to identify two main stabilizing mechanisms that help characterize this threshold proposed by Kelvin: inviscid damping (I.D.) and enhanced dissipation (E.D.). To motivate these concepts, we consider the linearized dynamics of the vorticity perturbation $f$ associated with the full model  \eqref{eq-initial}. Let us simplify the setting further by considering the boundaryless domain $\mathbb{T}\times \mathbb{R}$, on which the evolution of the perturbation reads as follows
\begin{align}\label{PS_eq}
    \pa_t f+y\pa_x f=\nu \Delta f,\quad f(t=0)=f_{\rm in}.
\end{align}
After decomposing the solution into the streamline-average and the remainder \begin{align}\label{f_0,f_nq}
f_0(t,y):=\f{1}{|\mathbb T|}\int_{\mathbb T}f(t,x,y)dx,\quad f_\nq:=f-f_0,
\end{align}
one can observe two types of decay through the lens of Fourier analysis,
\begin{align}
\text{I.D.:    } \quad  \|\na\Delta^{-1}f_\nq(t)\|_{L^2} \leq& \frac{C}{1+t}\|f_{\nq;\rm in}\|_{H^1};\quad
\text{E.D.:    }  \quad   \|f_\nq(t)\|_{L^2}\leq C\|f_{\nq;\rm in}\|_{L^2}e^{-\delta\nu ^\f13t}. 
\end{align}
Through the standard Biot-Savart law, the first inequality implies that the $\neq$-component of velocity perturbation decays algebraically with a $\nu$-independent rate. This decay is called \textit{inviscid damping} and has close connection to the mixing phenomena in the literature \cite{MouhotVillani11,AlbertiCrippaMazzucato14,YaoZlatos17}. The second estimate highlights that the $\neq$-component of the vorticity perturbation decays on an $\mathcal O(\nu^{-\f13})$ time scale, which is much faster than the heat dissipation time scale $\mathcal O(\nu^{-1})$ in the vanishing viscosity regime. This is called \textit{enhanced dissipation}. Since we focus on the small $\nu$-regime, derivation of these two types of decay for the full Navier-Stokes equation \eqref{eq-initial} is essential. 

With these early developments, the transition threshold problem, initially proposed by Trefethen et al. \cite{trefethen1993}, was later mathematically formulated by Bedrossian-Germain-Masmoudi \cite{BGM2017}:

{\it Given a norm $\|\cdot\|_X$, identify the associated critical exponent $\beta=\beta(X)$ such that the following dichotomy holds
\begin{align}\label{threshold}
  \begin{aligned}    
  \|v_{\rm in}\|_{X}\leq \nu^{\beta}  &\;\Longrightarrow\;    \text{The solutions are \textbf{stable} and exhibit \textbf{I.D.} and \textbf{E.D.}};\\
\|v_{\rm in}\|_{X}\gg \nu^{\beta} &\;\Longrightarrow \;   \text{The solutions are \textbf{unstable}}.    
  \end{aligned}
\end{align}
}
Our goal is to understand the stable side of the transition threshold problem in the case where the domain is a finite channel and the solutions satisfy Robin boundary conditions. 

Before presenting the theorems, we survey the prior literature that motivates our exploration of this problem. As discussed above, the study of the transition threshold problem for the Couette flow in the absence of a boundary dates back to the 19th century. However, the main progress towards a detailed understanding of the long-time, nonlinear dynamics of the NS-solutions around the Couette flow was made only in the last two decades. This research trend is initiated by the works \cite{BM13,BMV14}. Since the literature is vast and the established results depend on the regularity level, the underlying domain, and the associated boundary conditions, we organize them into two categories: the boundaryless case and the bounded-domain case. 

In the boundaryless case, we fix the unbounded channel $\mathbb{T}\times \mathbb{R}$, and summarize the results based on the regularity-class $X$ \eqref{threshold}:
    \begin{itemize}
        \item Stability in Sobolev space:  Through sequences of works \cite{BVW18, MasmoudiZhao22,WZ23}, it is established that when $X=H^3$, the critical exponent $\beta\leq \f13$. Namely, if the $H^3$-norm of the initial perturbation $\|v_{\mathrm{in}}\|_{H^3}\leq \epsilon\nu^{\frac{1}{3}}$, then the solution $v$ will stay in an $\mathcal{O}(\nu^\f13)$-neighborhood of Couette and exhibit the inviscid damping and enhanced dissipation. 
        \item Stability in Gevrey class: In the Gevrey setting, the critical exponent $\beta$ is directly linked to the Gevrey index. Through the works  \cite{BMV14, LiMasmoudiZhao22}, it is established that if the Gevrey-$s$ ($s\in(0,1)$) norm of the initial data is small, i.e., 
            $\|v_{\mathrm{in}}\|_{\mathcal{G}^{s}}\leq \epsilon\nu^{\beta(s)}$, then stability holds, and one can observe inviscid damping and enhanced dissipation. Here $\beta(s)=\max\{0,\frac{1-2s}{3-3s}\}$. We can see that the critical exponent $\beta(s)$ continuously transits when one moves from the high Gevrey space to the Sobolev space. 
        \item Stability/Instability in lower regularity spaces: In the works \cite{MZ20, LiMasmoudiZhao24}, the authors showed that the critical exponent $\beta$ is indeed ${1/2}$ in some low regularity classes. Namely, instability is identified for the solutions initiated from rough data with $\|\nabla v_{\rm in}\|_{H^{\log}_xL^2_y}\gg\nu^\f12$.  
    \end{itemize}

Next, we consider domains with physical boundaries. Let us fix the finite periodic channel $\mathbb{T}\times [-1,1]$ and summarize the main results based on the boundary conditions that are imposed:
\begin{itemize}
    \item Non-slip boundary condition: In the work \cite{ChenLiWeiZhang18}, it is established that if the initial data is small, i.e., $\|v_{\mathrm{in}}\|_{H^3}\leq \epsilon \nu^{\frac12}$, then the solutions are stable in certain weighted Sobolev spaces and exhibit inviscid damping and enhanced dissipation in the time-integrated sense. In the presence of a non-slip boundary, the vorticity can quickly build up near the wall (even on the level of linearized equations, see, e.g., \cite{BedrossianHe19}). To overcome this difficulty, suitable boundary vanishing weighted norms of the vorticity are introduced in \cite{ChenLiWeiZhang18}, which also motivates our functional setup. In a recent work \cite{BianGrenierMasmoudiZhao25}, with computer assistance, the authors find the generation of sub-layers for some initial data belonging to any Sobolev spaces which exceeds $\mathcal{O}(\nu^\f12)$, which leads to some boundary instability. Hence, it is believable that in the non-slip setting, the critical exponent $\beta$ for the weighted Sobolev spaces is $\frac{1}{2}$.
\item Navier-slip boundary condition: Most of the existing work focuses on the free-slip case where $\kappa=0\leftrightarrow \al=\infty$. This physical boundary condition is more forgiving from an analysis standpoint, in that the vorticity is forced to vanish on the boundary. Hence, one does not expect strong boundary vorticity to build up. As a result, the results are aligned with the boundary-less case. 
\begin{itemize}
\item Stability in the Sobolev spaces: It is established in a sequence of works \cite{bedrossian2025stability,WeiZhangQuasiL} that as long as $\|v_{\mathrm{in}}\|_{ H^4}\leq \epsilon \nu^{\frac13}$, the solutions stay close to the Couette and one observes inviscid damping and enhanced dissipation.
\item Stability in Gevrey class: In the work \cite{bedrossian2024uniform}, the authors observe inviscid damping and enhanced dissipation in the situation where $\|v_{\mathrm{in}}\|_{\mathcal{G}^s}\leq \epsilon,\, s\in (\frac{1}{2},1)$. 
\end{itemize}
\end{itemize}
Now, one can see an unexplored realm that connects the non-slip and free-slip settings. By tuning the friction factor $\alpha$ from $0$ to $\infty$, the system \eqref{eq-initial} formally transits between the non-slip case and the free-slip case.  Our main goal is to characterize the transition behavior of the critical exponent $\beta$ in this process. Our main result proves that for fixed friction factor $\al$ and sufficiently small viscosity, initial perturbation of size $\mathcal{O}(\nu^\f13)$ in higher Sobolev space remains globally stable and exhibits inviscid damping, enhanced dissipation and weighted vorticity control.

\begin{theorem}\label{thm_main}
 Consider solution $v$ of the equation \eqref{eq-initial} with initial data $v_{\rm in}=(v_{\mathrm{in}}^1, v_{\mathrm{in}}^2)\in H^6(\Torus\times [-1,1])$.  For any fixed friction factor $\alpha=\frac{1-\kappa}{\kappa}>0$, there exist constants $\ep, \ep_0\in(0,1/4)$, and an $\alpha$-dependent threshold $\nu_0\in(0,1/4)$ such that for all $0<\nu\le \nu_0$, and all initial velocity $v_{\mathrm{in}}$ satisfying consistency constraint
    $\big(\alpha\pa_{{n}}v_{\mathrm{in}}^1+ v_{\mathrm{in}}^1, v_{\mathrm{in}}^2\big)|_{y=\pm1}\equiv(0,0)$ and smallness assumption
\begin{align}\label{smallness}
    \|v_{\mathrm{in}}\|_{H^6}\leq \epsilon \nu^{\frac13},
\end{align}
the following stability estimates hold for all time with a universal constant $C\geq 1$,
\begin{align*}
  &  \|e^{\ep_0\nu^{\f13}t}(1-|y|)\curl v_{\ne}\|_{L_t^{\infty}L_{x,y}^2}+\|e^{\ep_0\nu^{\f13}t}
  v_{\ne}\|_{L_t^2L_{x,y}^2}\\
  &\hspace{1.25cm}+\nu^{\f14}\|e^{\ep_0\nu^{\f13}t}
  \curl v_{\ne}\|_{L_t^2L_{x,y}^2}+\|\curl v_0\|_{L_t^{\infty}L_{x,y}^2}\le C\ep\nu^{\f13}.
\end{align*}
\end{theorem}
\begin{remark}
    In the forthcoming paper, we will study the stability threshold of the Navier-Stokes solutions near the Couette flow in the situation where the friction coefficient and the viscosity are coupled in the sense that $\alpha=\mathcal{O}(\nu^{b})$ with $b> 0$. 
\end{remark}

This problem connects different families of mathematical explorations. First of all, we would like to mention recent results concerning the stability of the Couette flow for viscous fluid in other physical domains \cite{BGM2017, BGM2020,BGM2015, WeiZhang21, ChenWeiZhang3D, chen2025optimal, chen2026quantitative, wang2025transition, li2025stability}. Furthermore, the dynamics of general two-dimensional Navier-Stokes flows subject to Navier boundary conditions have attracted significant attention from the mathematical community for more than six decades. We refer the interested readers to the works \cite{Lions1969quelques,Lions1996mathematical,Kelliher06,MasmoudiRousset12} for the well-posedness theory and the papers \cite{ClopeauMikelicRobert98,Kelliher06,LopesLopesPlanas05,IftimiePlanas06,MasmoudiRousset12,TaoWangZhang20
} and the references therein for the vanishing viscosity problems.  

In the next section, we will present our battle plan to prove Theorem \ref{thm_main}. 

\section{Outline}
\subsection{Strategy}
To present the roadmap for the analysis, we first highlight the key challenges in deriving the $\mathcal{O}(\nu^{1/3})$-stability threshold and introduce our strategy for overcoming them. We list the main challenges as follows:
\begin{itemize}
\item[a)] As being observed in the literature  \cite{BMV14, MasmoudiZhao22, LiMasmoudiZhao22,BM13}, an echo-chain (resonance) induced norm inflation can potentially drive the Navier-Stokes solutions to escape an $\mathcal{O}(\nu^\f13)$-neighborhood of the Couette flow. \item[b)] The new Robin/General Navier boundary condition ($\mathbb{T}_\pm[\om]:=\al \om(\pm1)\pm\pa_y\Delta^{-1}\om(\pm1)=0$)  forces us to consider new types of boundary correctors that are robust against time-varying background shears.
\item[c)] In the presence of the physical boundary, traditional Fourier transform techniques in \cite{BMV14} is less effective, and suitable functional framework on the physical side is preferable. 

\end{itemize}
Our main strategy is to combine the following ingredients in our analysis:
\begin{itemize} 
\item[a)] Apply a \textit{hierarchical amplitude-based decomposition} to the solution. On Level 1, the terms have size $\mathcal{O}(\nu^\f13)$, while Level 2 terms are of size $\mathcal{O}(\nu^{\f23-})$. As a consequence, we single out the main linear dynamics, and then the key resonant part of the remainder. By introducing a resonance-tracking functional framework, we record the norm inflation in the resonant component. Moreover, this decomposition naturally distinguishes between the short-time regime, where direct energy estimates play a major role, and the long-time regime, where resonant growth requires careful treatment.
\item[b)]  Design the boundary corrector to capture the boundary vorticity generation near the Robin boundary.
\item[c$_1$)]   In the short time regime, we introduce a \textit{good unknown} consisting of a mixture of components, each satisfying distinct boundary conditions. The advantage is that it respects the transport structure. See, \eqref{eq: good-unknown}.
\item[c$_2$)] In the long time regime, we introduce physical side singular integral operators (SIOs) $\{\mathfrak{J}_k,\, \w{\mathfrak{J}}\e_k\}_{k\neq 0}, \,\wt{\mathfrak{J}}_0$ to record the inviscid damping, the enhanced dissipation and resonance-induced norm inflation in the system. Here, $\wt{\mathfrak{J}}_0$ is a new SIO which is well-adapted to the Robin boundary conditions and captures the cascade growth from the quasi-linear interactions. 
\end{itemize}

In the forthcoming subsections, we will elaborate on each ingredients, then we will set up a complete bootstrap setup to conclude the proof. 

\subsection{Hierarchical Amplitude-based Decomposition}
The main idea of the hierarchical amplitude-based decomposition is to extract the large amplitude solution which exhibit detailed structure, and analyze the remainder with various energy functional. 

\noindent{\bf Level 1: }First of all, we introduce the first-level linear decomposition. 
We decompose the initial data as the $x$-average and the remainder,
\begin{align*}
     v_{\rm in}^1(x,y)= v^1_{0,\rm in}+ v^1_{\ne,\rm in},\quad  v^1_{0,\rm in}=\mathbb P_0 v^1_{\rm in},
\end{align*}
where $\mathbb P_0 f:=f_0:={|\mathbb T|}^{-1}\int_{\mathbb T}f(x,y)dx,\, \mathbb{P}_{\neq} f:=f_\nq:=f-f_0$. 
Consider the heat extension  $\wt U$ of the initial horizontal velocity,
\begin{align}
    \pa_t\wt U -\nu\pa_y^2\wt U =0,\quad \al\pa_y\wt U (\pm1)=\mp \wt U (\pm1),\quad \wt U |_{t=0}=v_{0,\rm in}^1(y),\label{eq-L-H-Z}
\end{align}
where $ v_{0,\rm in}^1$ satisfies the consistency condition
\begin{align}\label{eq-v_in}
     \al\pa_yv_{0,\rm in}^1(\pm1)=\mp v_{0,\rm in}^1(\pm1).
\end{align} The estimate of $\wt U$ is presented in Lemma \ref{Lem-U}. 
Next, we define the perturbation 
\begin{align*}
    u (t,x,y):= v(t,x,y)-(\wt U (t,y),0),
\end{align*}
which solves
\begin{equation}\label{eq-velocity}
    \begin{cases}
        \pa_t u +\mathcal{U}(t,y)\pa_x u -\nu\D u + u \cdot\nabla u +(\pa_y\mathcal{U}(t,y)u^2,0)+\na P=0,\\
        \na\cdot u =0,\quad \al\pa_yu^1(t,x,\pm1)=\mp u^1(t,x,\pm1),\\
         u (0, x,y)= u _{\rm in}(x,y)= u _{\ne,\rm in},
    \end{cases}
\end{equation}
where $\mathcal{U}(t,y)$ be the resulting shear flow given by the corresponding Biot-Savart law
\begin{align}\label{eq-U}
    \mathcal{U}(t,y)=y+\wt U (t,y).
\end{align}
Then, we compute the vorticity $\om=\curl  u =\pa_xu^2-\pa_yu^1$ of the solution to \eqref{eq-velocity} and obtain the following vorticity formulation of the perturbed Navier-Stokes equations:
\begin{equation}
\begin{cases}\label{eq-I}
    \pa_t \om+\mathcal{U}(t,y)\pa_x \om-\nu\D \om=\pa_y^2\mathcal{U}\pa_x\psi- u \cdot\nabla \om, \\
    \om=\D\psi, \quad  u =\nabla^{\perp}\psi=(-\pa_y\psi,\pa_x\psi),\\
- \al \om(t,x,\pm1)=\mp u^1(t,x,\pm1),   \quad \psi(y=\pm1)=0,\\
 \om(0,x,y)=\om_{\rm in}(x,y):=\curl  v_{\ne,\rm in}(x,y).
    \end{cases}
\end{equation}
To simplify the notations in the sequel, we define the operator $\mathbb{T}_{\pm}$:
\begin{align}\label{Trace_op}
\mathbb{T}_\pm [\omega_\nq]:=\al \omega_\nq(\pm1)\pm\pa_y\de^{-1}\omega_\nq(\pm1).
\end{align} Here and in the sequel, the $\de^{-1}$ always denotes the inverse of the Dirichlet Laplacian.  

Next, we consider the interior approximation $\w\omega a$ for the non-zero modes:
\begin{equation}\label{eq-a}    
   \pa_t\w\om{a}+\mathcal{U}(t,y)\pa_x\w\om{a}-\nu  t^2\pa_x^2\w\om{a}=0,\quad
   \w\om{a}(t=0)=\om_{\rm in}. 
\end{equation}
Here, we recall that $\omega_{\rm in}=\mathbb P_{\nq}\omega_{\rm in}$ because the zero-mode component $\mathbb P_0\omega_{\rm in}$ is already propagated through the heat equation. In the sequel, we use the notation $\w u \cdot$ to denote the velocity induced by $\w \omega \cdot$.
This interior approximation has the advantage of being explicitly computable and exhibiting sharp inviscid damping and enhanced dissipation. We further highlight that the approximation $\w \omega a$ does not satisfy the Robin boundary condition. This mismatch in boundary conditions necessitates the introduction of suitable boundary correctors. Without presenting the precise formula, which is delicate on its own, we define the first-level boundary corrector $\w\omega b_\nq$ that remedies the boundary condition of $\w\om{a}_\nq$, i.e.,
\begin{align}
\mathbb{T}_\pm[\w\omega a_\nq+\w\omega b_\nq]=0. \label{BC_rel_1}
\end{align} 
The choice of $\w\omega{b}$ is detailed in Section \ref{sec:omega bc}. 
Even though the first-level  approximations $\w\omega a$ and $\w \omega b$ capture the main bulk dynamics and the boundary behavior, error still remains and is recorded as follows
\begin{align}
\w\omega * :=\omega-\w\omega a-\w\omega b_\nq.\label{level_1_error}
\end{align}
This will be further resolved in the second-level analysis. 

\noindent{\bf Level 2:} 
The first-level approximation error $\w\omega*$ exhibits heterogenous behavior; for instance, inviscid damping and enhanced dissipation appear only in the $x$-dependent component of the solution. Consequently, we decompose $\w\omega*$ into the non-zero mode component and the streamline average, and treat them separately, 
\begin{align}
\w\omega*:=\w\omega*_\nq+\w\omega*_0.\label{level_1_error_nq_0}
\end{align}

\noindent{\bf Level 2: Nonzero mode.} First of all, we follow  the same strategy as in Level 1 to decompose the term into the interior bulk component and  the boundary corrector:
\begin{align}\label{omega*_nq_dcmp}
 \w\omega*_\nq=\w\omega{*,i}_\nq+ \w\omega{*,b}_\nq,\quad \w\omega{*,i}\big|_{y=\pm 1}=0,\quad \mathbb{T}_\pm[\w\omega{*,b}_\nq]=-\mathbb{T}_\pm[\w\omega{*,i}_\nq] .
\end{align}
We highlight that this decomposition is not unique and there are delicacies in defining the boundary corrector $\w\omega{*,b}_\nq$, which we will detail in Section \ref{sec:omega bc}. 
The challenge in analyzing the $\w\omega*_\nq$ resides in the fact that one needs to keep track of the quasilinear interactions (resonance) between the first level approximations $\w\omega a+\w\omega b$ and the error $\w\omega*_\nq$.  As it turns out,  this interaction is significant in the interior bulk region and for the long time $t\geq T_0$, so we introduce the decomposition
\begin{align}\label{eq-dcmp-e-re}
\w\omega{*,i}_\nq:=\w\omega \e_\nq+\w \omega {\rm re}_\nq,\quad \w\omega{\e}_\nq\big|_{y=\pm 1}=\w\omega{\rm re}_\nq\big|_{y=\pm 1}=0,\quad t\geq T_0 .
\end{align}
Here, $\w\omega {\rm re}_\nq$ is the ``resonant component" and $\w \omega{\e}_\nq$ is the second order ``error". For $t\ge T_0$, $\w\omega{\rm re}_\nq$ solves the following equation:
 \begin{equation}
        \begin{split}
            \partial_t \w{\om}{\rm re}_\nq+& \mathcal{U}(t,y)\pa_x\w{\om}{\rm re}_\nq-\nu\Delta \w{\om}{\rm re}_\nq=t\lf(\left(u^2-u^{(a, 2)}\right) \partial_x \w{\om}a\rg)_\nq ,\\ \w{\om}{\rm re}_\nq&(t=T_0)= 0, \quad \w{\om}{\rm re}_\nq(t,\pm 1)=0.
        \end{split}\label{eq-om_re}
  \end{equation}
For the resonant component $\w \omega {\rm re}_\nq$, we apply the resonant-tracking functional framework from \cite{WeiZhangQuasiL} to record norm inflation. For the second-level error $\w\omega \e_\nq$, we introduce the physical-side singular integral operator $\mathfrak{J}$ and its variants to track inviscid damping and enhanced dissipation. 

\noindent{\bf Level 2: Zero Mode.} For the zero-mode $\w\omega*_0$, since the average boundary vorticity is weaker, we choose not to introduce the boundary corrector. However, the resonance-induced norm inflation remains. As a result, we further decompose the zero-mode as follows:
\begin{align}
  \w\omega*_0=\w\omega{\rm re}_0+\w \omega{\e}_0,\quad \forall t\geq T_0.
\end{align}  
where $\w\om{\rm re}_0=-\pa_y\w u{\rm re}_0$, with $\w u{\rm re}_0$ satisfying
   \begin{equation}\label{eq-0re}
    \begin{cases}
        \pa_tu_0^{\rm (re)}-\nu\pa_y^2u_0^{\rm (re)}=\displaystyle\sum_{k\in\mathbb Z\backslash\{0\}}ik\big(\p_k-\p_k^{(a)}\big)\om_{-k}^{(a)},\\
        \al \pa_yu_0^{\rm (re)}(t,\pm1)=\mp u_0^{\rm (re)}(t,\pm1),\quad u_0^{\rm (re)}|_{t=T_0}=0.
        \end{cases}
    \end{equation}

To summarize, we have the decomposition:
\begin{align*}
&v=\underbrace{\underbrace{(\widetilde{U},0)}_{\text{zero mode \eqref{eq-L-H-Z}}}
+\underbrace{\underbrace{\nabla^{\bot}\Delta^{-1}\omega^{(a)}}_{\text{interior \eqref{eq-a}}}
+\underbrace{\nabla^{\bot}\Delta^{-1}\omega^{(b)}}_{\text{boundary \eqref{C-op} and \eqref{eq-b1}}}}_{\text{non-zero modes}}}_{\text{Level 1: $\nu^{1/3}$-based size}}
+\underbrace{\nabla^{\bot}\Delta^{-1}\omega^{(*)}}_{\text{$\nu^{2/3-}$ size}}\\
\end{align*}
The Level 2 decomposition is mainly based on the amplification caused by the nonlinear interactions, namely, the resonant contribution and the boundary contribution:
{\small\begin{align}\label{om*_large_dcmp}\\ \nonumber
\omega^{(*)}=
\underbrace{\overbrace{\omega^{(\mathrm{re})}_0}^{\text{Resonant contribution }\eqref{eq-0re}}+\omega_0^{(\mathrm{e})}}_{\text{Level 2: zero mode $\omega_{0}^{(*)}$}}+\underbrace{\underbrace{\overbrace{\omega^{(\mathrm{re})}_{\neq}}^{\text{Resonant contribution }\eqref{eq-om_re}}+\omega_{\neq}^{(\mathrm{e})}}_{\text{Interior $\omega^{(*,i)}_{\neq}$}}+\underbrace{\omega_{\neq}^{(*,b)}.}_{\text{Boundary \eqref{C-op} and \eqref{eq-b1}}}}_{\text{Level 2: non-zero modes $\omega_{\neq}^{(*)}$}}
\end{align}}
Here, we call $\w\omega *$ the first order remainder which is the error that we have made by approximating the solution by the main structure $\w\omega a_\nq+\w\omega b_\nq$. Within the remainder $\w\omega*_\nq$, one can further extract the resonant part $\w\om{\rm re}_\nq$ in the long time regime, leaving the second order remainder $\w\omega{\e}_\nq$ and the second order boundary corrector $\w\omega{*,b}_\nq$. Since the boundary effect for the zeroth mode of $\w\omega*$ is less pronounced, we only decompose it into the resonant part and error. In this way, we have to control the resonant contribution with the Robin boundary condition. This is one of the main difficulties overcame in the proof. 

We also remark that the key reason why the above decomposition works is that the following three main effects are decoupled: 
\begin{itemize}
    \item the I.D. and E.D. from the linearized system (see the estimates of $\omega^{(a)}$),
    \item the Orr mechanism caused by the quasi-linear interactions which gives the main growth (see the estimates of $(\omega^{(\rm re)}_{\neq}, \omega^{(\rm re)}_0)$), 
    \item and the boundary effect (see $(\omega^{(b)}, \omega^{(*,b)})$). 
\end{itemize}
It means that the effects of $(\omega^{(a)}, \omega^{(\rm re)}_{\neq}, \omega^{(\rm re)}_0, \omega^{(b)}, \omega^{(*,b)})$ on $(\omega^{(\rm e)}_{\neq}, \omega^{(\rm e)}_0)$ as well as their self-interactions are relatively weak. 

To conclude this section, we present the estimates of the first level approximations. Here, Proposition \ref{ol:estimate-u1} describes the bounds for the approximate solution $\w{\omega} a$ and the size of the boundary corrector $\w{\omega} b$ is characterized by Proposition \ref{pro:linear_boundary_corrector}, whose proof relies heavily on the machinery that we will discuss in next subsection. 
\begin{proposition}\label{ol:estimate-u1}Consider the first level interior approximate solutions $\w\omega a,$ \eqref{eq-a}, the corresponding stream function $\w\psi a_\nq=(\de)^{-1}\w\omega a_\nq$. The following estimates hold 
\begin{align*}
&\|\pa_x\w \omega a_\nq\|_{L^{\infty}}+\|(\pa_y+t\pa_x)\w\om{a}_\nq\|_{L^{\infty}}\lesssim e^{-2\nu^{\f13}t}\|\om_{\rm in}\|_{H^3},\\
&\|\pa_x\pa_y\w\psi{a}_{\ne}\|_{L^{\infty}}+\langle t\rangle \|\pa_x^2\w\psi{a}_{\ne}\|_{L^{\infty}}\lesssim \langle t\rangle^{-1}e^{-2\nu^{\f13}t}\|\om_{\rm in}\|_{H^3}.
 \end{align*}
\end{proposition}

\begin{proposition}\label{pro:linear_boundary_corrector} Consider the first level boundary approximate solution $\w\omega b$ \eqref{BC_rel_1}, the corresponding stream function $\w\psi b_\nq=(\de)^{-1}\w\omega b_\nq$. There exists a constant $\delta\in(0,1)$ such that the following estimate hold 
\begin{align*}
   &  \|e^{\delta\nu^{\f13}t}\pa_x\na\w\p{b}_{\ne}\|_{L^2L^2}+ \nu^{-\f16}\|e^{\delta\nu^{\f13}t}(1-|y|)\pa_x\w\om{b}_{\ne}\|_{L^{\infty}L^2}+\nu^{\f13}\|e^{\delta\nu^{\f13}t}\pa_x\w\om{b}_{\ne}\|_{L^2L^2}\lesssim \nu^{\f13}\|\om_{\rm in}\|_{H^2}.
\end{align*}
\end{proposition} 
Proposition \ref{ol:estimate-u1} is proven by an explicit representation formula of $\om^{(a)}$ and the proof can be found in Section \ref{sec: appro-sol}.  Proposition \ref{pro:linear_boundary_corrector} is a directly consequence of the discussion in the next subsection (Proposition \ref{prop-C}). Full details can be found in Section \ref{sec-boundary}.

\subsection{Boundary Correction Operators $\msc C$}\label{sec:omega bc}
In this section, we design a boundary correction operator $\msc C$ which serves the following  goal. Given a nonzero mode ``defect input" \(g_{\neq}(t,x,y)\) (In our case, $\w\omega a$ and $\w \omega{*,i}$), the operator $\msc C$ returns an output function \(\mathscr C[g_{\neq}]\) that restores the Robin boundary condition and produces a ``frozen-shear" error \(\mathscr E[g_{\neq}]\) that feeds back to the system. Moreover, the boundary corrector \(\mathscr C[g_{\neq}]\) exhibits bounds that are robust enough to treat time-varying background shear flows. To achieve this goal, we first review how the estimates for boundary correctors are derived. It is now classical to apply resolvent analysis, which starts by applying the Fourier transform of the equation with respect to the temporal variable. When dealing with a time-varying background shear, the Fourier transform generates convolution-type terms that are challenging to address. To overcome this, we use a technique known in the literature as the Frozen-time scheme. To implement this approach, we partition the time horizon and define
\begin{equation}\label{t_j}
    \begin{aligned}
       & t_j:=j(\nu^{-\f13}+1),\quad \mathcal{U}^{[j]}(y):=\mathcal{U}(t_j,y),\quad  \mathcal{I}_{[j]}:=[t_j,t_{j+1})\cap[0,T], \\
        &N:=\min_{1\leq j\in \mathbb{N}}\lf\{[0,T]\subset[0,j (\nu^{-\f13}+1)]\rg\}-1.
    \end{aligned}
\end{equation}
We further define the jump region:
\begin{align}\mathcal{J}_{[0]}:=\emptyset;\quad
\mathcal{J}_{[j]}:=[t_{j}-1, t_{j}),\quad j=1,2, \cdots, N-1. 
\end{align} 
Here $N$ is the smallest natural number such that $
    [0,T]=\cup_{j=0}^N\mathcal{I}_{[j]}$ and  $
    \mathcal{U}^{[j]}$ is the ``frozen shear" within the time interval $\mathcal{I}_{[j]}$. We further define the following smooth cutoff functions $\boldsymbol\chi_j$ such that 
    \begin{align}
    \boldsymbol{\chi}_j=\begin{cases} 1,\quad t\in [j(\nu^{-\f13}+1),(j+1)(\nu^{-\f13}+1)-1] \cap [0,T],\\
    0,\quad [0, j(\nu^{-\f13}+1)-1]\cup [(j+1)(\nu^{-\f13}+1), \infty).\end{cases}.
    \end{align}
We further pick the cutoff such that $\operatorname{support}\boldsymbol{\chi}_j=\overline{\mathcal{J}_{[j]}\cup \mathcal{I}_{[j]}}.$ 

Next, for function $g_\nq$ on the domain $\mathbb{R}^+\times\Torus\times[-1,1]$, we define the following boundary corrector $\mathscr{C}[g_\nq]$:
\begin{align}\begin{split}
& \mathscr{C}[g_\nq](t)=\sum_{j'=0}^j\mathscr{C}^{[j']}[g_\nq](t)\text{ for all }t\in \mathcal{I}_{[j]}\backslash\mathcal{J}_{[j+1]};\\
& \mathscr{C}[g_\nq](t)=\sum_{j'=0}^{j+1}\mathscr{C}^{[j']}[g_\nq](t)\text{ for all }t\in\mathcal{J}_{[j+1]},\\ &\lf(\pa_t+\mathcal{U}^{[j]}\pa_x-\nu\Delta\rg) \mathscr{C}^{[j]}[g_\nq]=0,  \quad \mathbb{ T}_\pm[\mathscr{C}^{[j]}[g_\nq]]=-\mathbb{T} _\pm[g_\nq]\boldsymbol{\chi}_{j},\quad \mathscr{C}[g_\nq](t=0)=0.\end{split}\label{C-op}
\end{align} 
We recall that $\mathbb{T}_\pm[g_\nq]:=\al g_\nq(\pm1)\pm\pa_y\de^{-1}g_\nq(\pm1)$. 
 Note that the bootstrap region is $[0,T_*)$, if $T_*\in{\rm supp}\boldsymbol\chi_{j}$, for $t\ge T_*$ and $t\in{\rm supp}\boldsymbol\chi_j$, we define $\mathbb T_{\pm}[g_{\ne}](t)=\mathbb T_{\pm}[g_{\ne}](T_*)e^{-(t-T_*)}.$
Finally, we define an error functional which records the ``Frozen shear error'' made by the boundary corrector $\msc C$:
\begin{equation}
\begin{aligned}\label{Frozen_err}
&\mathscr{E}[g_\nq](t,y):=\sum_{j=0}^N\mathscr{E}^{[j]}[g_\nq](t,y)\ {\bf 1}_{t\in \mathcal{I}_{[j]}};\\
&\mathscr{E}^{[j]}[g_\nq](t,y):=(\mathcal{U}^{[j]}(y)-\mathcal{U}(t,y))\sum_{j'=0}^j\mathscr{C}^{[j']}[g_\nq]-\sum_{j'=0}^{j-1}(\mathcal{U}^{[j]}(y)-\mathcal{U}^{[j']}(y))\mathscr C^{[j']}[g_\nq].
\end{aligned}
\end{equation}
Here, the $j$ ranges in $\{0,1,2,\cdots,N\}$. If $j=0$, the second sum is omitted. 

Now we recall the relations \eqref{BC_rel_1} and \eqref{omega*_nq_dcmp}, and choose the $\w \omega a_\nq$ and $\w\omega{*,i}_\nq$ as the input functions to define
\begin{align}\label{eq-b1}
\w\omega b_\nq:=&\mathscr{C}[\w \omega a_\nq],\quad \w\om{*,b}_{\ne}=\mathscr{C}_k[\w\om{*,i}_{\ne}].
\end{align} 
In the sequel, we will frequently use the following Fourier transform in $x\in\mathbb T$, 
\begin{align}\label{Fourier_x}
    f(t,x,y)=\sum_{k\in\mathbb Z} f_k(t,y)e^{ikx},\quad f_k(t,y)=\frac{1}{2\pi}\int_{\Torus}f(t,x,y)e^{-ikx}dx.
\end{align} Moreover, we will drop the classical Fourier $\widehat{(\cdot)}$-notation for the sake of simplicity. 
Then we define $\mathscr{C}_k[g_k]$ ($k\neq0$) as the $x$-Fourier transform of $\mathscr{C}[g_\nq]$.

Before dwelling into the concrete statement about the boundary correctors, we first present an essential parameter in the understanding of boundary layers. The general strategy of deriving estimates for the boundary layer is to take the Fourier-Laplace transform of a suitably normalized solution in time and study the its associated  resolvent equation in the spectral variable $\lambda\in \mathbb{C}$. Standard spectral estimates (carried out in Section \ref{sec:BC}) yields suitable estimates for the solutions on the half plane $k\Im \lambda\geq -\delta \nu^{\f13}|k|^\f23$ (see, e.g., \eqref{epsilon_0}). Throughout the paper, we define
\begin{align}
\ep_0:=\delta/2.     \label{defn_epsilon_0}
\end{align}
With the central parameter introduced, we collect the estimates for the boundary correctors in the following propositions.
\begin{proposition}\label{prop-C}
   Consider functions $\{g_k\}_{k\neq0}$ subject to the constraints
   \begin{align*}\begin{cases}\|e^{\ep_0\nu^{\f13}t}\al g_k(1)\|_{L_t^2[0,T_*)}^2+\|e^{\ep_0\nu^{\f13}t}\al g_k(-1)\|_{L_t^2[0,T_*)}^2+\|e^{\ep_0\nu^{\f13}t}\pa_y(\Delta_k)^{-1}g_k\|_{L_t^2([0,T_*);L_y^{\infty})}^2<\infty,\\ 
   \mathscr{C}_k[g_k](t=0,\cdot)\equiv 0,\end{cases}
   \end{align*} 
   the following estimate for the boundary corrector $\msc C_k[g_k]$ holds for all $k\neq 0$,
    \begin{align*}
        &\|e^{\ep_0\nu^{\f13}t}(ik,\pa_y)\Delta_k^{-1}\mathscr{C}_k[g_k]\|_{L_t^2([0,T_*);L_y^2)}^2+\nu^{-\f13}|k|^{\f13}\|e^{\ep_0\nu^{\f13}t}(1-|y|)\mathscr{C}_k[g_k]\|_{L_t^{\infty}([0,T_*);L_y^2)}^2\\
        &\quad +\nu^{\f23}|k|^{-\f23}\|e^{\ep_0\nu^{\f13}t}\mathscr{C}_k[g_k]\|_{L_t^2([0,T_*);L_y^2)}^2\\
    &\lesssim\nu|k|^{-1}\Big(\|e^{\ep_0\nu^{\f13}t}\al g_k(1)\|_{L_t^2[0,T_*)}^2+\|e^{\ep_0\nu^{\f13}t}\al g_k(-1)\|_{L_t^2[0,T_*)}^2+\|e^{\ep_0\nu^{\f13}t}\pa_y(\Delta_k)^{-1}g_k\|_{L_t^2([0,T_*);L_y^{\infty})}^2\Big).
    \end{align*}
\end{proposition}
\begin{lemma}\label{Lem-F-Y}
   Consider the function family $\{g_k\}_{k\neq 0}$ specified in Proposition \ref{prop-C}. It holds that
    \begin{align*}
  &\|e^{\ep_0\nu^{\f13}t}\mathscr{E}_k[g_k]\|_{L^2_t([0,T_*);L_y^2)}^2\\
    &   \lesssim    \nu^{\f73}|k|^{-\f13}\Big(\|e^{\ep_0\nu^{\f13}t}\al g_k(1)\|_{L_t^2[0,T_*)}^2+\|e^{\ep_0\nu^{\f13}t}\al g_k(-1)\|_{L_t^2[0,T_*)}^2+\|e^{\ep_0\nu^{\f13}t}\pa_y(\Delta_k)^{-1}g_k\|_{L_t^2([0,T_*);L_y^{\infty})}^2\Big).
    \end{align*}
\end{lemma}
As a corollary, we have the following:
\begin{corol}\label{prop-ba}
   The following estimates for $(\om^{(b)}_k,\psi^{(b)}_k),\, k\neq0$ hold
\begin{align*}
    & \|e^{\ep_0\nu^{\f13}t}(ik,\pa_y)\w\p{b}_k\|_{L_t^2([0,T_*),L_y^2)}^2+ \nu^{-\f13}|k|^{\f13}\|e^{\ep_0\nu^{\f13}t}(1-|y|)\w\om{b}_k\|_{L_t^{\infty}([0,T_*),L_y^2)}^2\\
    &\quad +\nu^{\f23}|k|^{-\f23}\|e^{\ep_0\nu^{\f13}t}\w\om{b}_k\|_{L_t^2([0,T_*),L_y^2)}^2\lesssim \nu^{\f23}|k|^{-3}\|\om_{k;\rm in}\|_{H_{k,y}^2}^2.
\end{align*}
\end{corol}
\begin{corol}\label{prop-eb}
    It holds that
    \begin{align*}
    &\|e^{\ep_0\nu^{\f13}t}(ik,\pa_y)\w\p{*,b}_k\|_{L_t^2([0,T_*),L_y^2)}^2+ \nu^{-\f13}|k|^{\f13}\|e^{\ep_0\nu^{\f13}t}(1-|y|)\w\om{*,b}_k\|_{L_t^{\infty}([0,T_*),L_y^2)}^2\\
    &\qquad+\nu^{\f23}|k|^{-\f23}\|e^{\ep_0\nu^{\f13}t}\w\om{*,b}_k\|_{L_t^2([0,T_*),L_y^2)}^2\\
    &\lesssim \nu|k|^{-1}\big\|e^{\ep_0\nu^{\f13}t}\pa_y\w\p{*,i}_k\big\|_{L^2_t([0,T_*);L_y^{2})}\big\|e^{\ep_0\nu^{\f13}t}\w\om{*,i}_k\big\|_{L^2_t([0,T_*);L_y^{2})}.
\end{align*}
\end{corol}

The proofs of the above results are presented in Section \ref{sec-boundary}. 

\subsection{Singular Integral Operators}
The functional analytic toolkit introduced in this subsection is applied to treat the nonzero mode error $\w\omega\e_\nq$ \eqref{eq-dcmp-e-re} and the zero mode resonant component $\w\omega{\rm re}_0$ \eqref{eq-0re}. 
The main delicacy in analyzing the error  $\w\omega {\rm e}_\nq$ is that unlike $\w\omega a_\nq,\, \w\omega {\rm re}_\nq$, we lack simple structure or direct formula to keep track of its size.
Motivated by the works \cite{li2024asymptotic, bedrossian2025stability}, we introduce various singular integral operators to keep track of the inviscid damping and enhanced dissipation. 
\begin{definition}\label{defn:SIO} The singular integral operators $\mathfrak{J}_k$ and $\mathfrak{J}_k^{\rm (e)}$ are given by
\begin{equation}\label{J_kJ_e}
\begin{aligned}\text{Inviscid Damping SIO: }\qquad
\mathfrak{J}_k[f_k](y)&:=kP.V.\int_{-1}^1\frac{G_k(y,y')}{2i\big(\mathcal{U}(t,y)-\mathcal{U}(t,y')\big)}f_k(y')dy',\\
   \text{Enhanced Dissipation SIO: }\qquad \mathfrak{J}_k^{\rm (e)}[f_k](y)&:=kP.V.\int_{-1}^1\frac{\nu^{\f13}|k|^{\f23}\w G{\rm e}_k(y,y')}{2i\big(\mathcal{U}(t,y)-\mathcal{U}(t,y')\big)}f_k(y')dy'.
\end{aligned}
\end{equation}
Here, the $G_k$ is the Green's function for $\D_k=\pa_y^2-|k|^2$ with homogeneous Dirichlet boundary condition on $[-1,1]:$
\begin{equation}\label{eq-Green}
    G_k(y,y')=-\frac{1}{|k|\sinh(2|k|)}\begin{cases}
        \sinh(|k|(1-y'))\sinh(|k|(1+y)),\quad y\le y';\\
        \sinh(|k|(1-y))\sinh(|k|(1+y')),\quad y\ge y'.
    \end{cases}
\end{equation}
The $\w G{\rm e}_k$ is the Green's function  for $\nu^{\f23}|k|^{\f43}\pa_y^2-|k|^2$ 
with Dirichlet boundary condition,
\begin{equation*}
     \w G{\rm e}_k(y,y')=\frac{\nu^{-\f13}|k|^{-\f23}}{|k|\sinh(2\nu^{-\f13}|k|^{\f13})}\begin{cases}
       \sinh\big(\nu^{-\f13}|k|^{\f13}(y+1)\big)\sinh \big(\nu^{-\f13}|k|^{\f13}(y'-1)\big),\ y\le y';\\
       \sinh\big(\nu^{-\f13}|k|^{\f13}(y-1)\big)\sinh \big(\nu^{-\f13}|k|^{\f13}(y'+1)\big),\ y\ge y'.
    \end{cases}
\end{equation*}
\end{definition}
The properties of these multipliers are detailed in Section \ref{sec-operator}. For presentation purpose, we take the $x$-Fourier transform of the linearized equation \eqref{PS_eq} with Dirichlet boundary condition
\begin{align}
\pa_t f_k+yik f_k=\nu \de_k f_k,\quad f_k\big|_{y=\pm 1}=0, \quad f_k(t=0)=f_{k;\rm in}, \quad k\neq0.
\end{align} Now one can consider the following energy associated with the above SIO's (see \eqref{eq-E} for further computational details)
\begin{align*}
\mathbb{E}_k(t)=\|f_k\|_{L^2}^2+\lf[c_\al\lan f_k, \mathfrak{J}_k [f_k]\ran_{L_y^2}+c_\beta\lan f_k, \mathfrak{J}_k^{(\e)} [f_k]\ran_{L_y^2}\rg]\approx \|f_k\|_{L^2}^2. 
\end{align*}
In the standard energetic arguments, the extra SIO-induced components support extra coercive damping terms (\textit{Cauchy-Kovalevskaya terms}) which correspond to inviscid damping and enhanced dissipation through Lemma \ref{Lem-SIO}. Further nonlinear estimates yield the bound for $\w\omega\e$. 

Next we move on to the analysis of the zero mode resonant component $\w\omega{\rm re}_0$, which suffers from ``resonant forcing'' from the nonzero modes. We introduce a new physical-side singular integral operator $\wt{\mathfrak{J}}_0$ to track its evolution. 
\begin{definition}\label{defn:SIO-0} The \emph{cascade} singular integral operator $\wt{\mathfrak{J}}_0$ is defined as  
    \begin{equation}\label{J0}
        \begin{aligned}
            \wt{\mathfrak{J}}_{0}[f_0](y)=\sum_{k\in\mathbb Z\backslash\{0\}}\f{{\rm sgn}(k)}{|k|^2}P.V.\int_{-1}^1\f{e^{ikt(y-y')}G_k(y,y')}{2i(y-y')}f_0(y')dy',
        \end{aligned}
    \end{equation}
    where $G_k(y,y')$ is defined in \eqref{eq-Green}.
\end{definition}

Now, we present the main idea to control the resonant component of the zero mode. After taking the $y$-derivative on \eqref{eq-0re}, we keep the main contributor in the vorticity dynamics and study the simplified model $\pa_t\w\om{\rm re}_0=\sum_{k\in\mathbb Z\backslash\{0\}}ik(\p_k-\p_k^{(a)})\pa_y\w\om{a}_{-k}${, in which the quadratic interaction between the nonlinear component and the first-level approximation $\w\om a_\nq$ drives the potential growth. These interactions are significant on a discrete set of times, and the \emph{cascade} SIO $\wt {\mathfrak{J}}_0$ (and the associated energy) serves to record the corresponding cascade growth.  We apply a simplification procedure as in the works \cite{BM13,LiMasmoudiZhao22} to extract the key model. Through standard change of coordinates $(z=x-yt,\, v=y)$, and application of Fourier transform $f(z,v)\leadsto \mathcal{F}[f](k,\eta)$, one can derive the following model}
\begin{align}
    \pa_tf_0(t)\approx -\sum_{k\in\mathbb Z\backslash\{0\}}\f{|k|f_k(t)}{k^2+(\eta-kt)^2}\f{1}{\langle k\rangle^2}{t\nu^{\f13}}e^{-\nu^{\f13}t}.
\end{align}
Here, $f_0(t),\ \{f_k(t)\}_{k\neq 0}$ play the roles of $\w\om{\rm re}_0(t,\eta)$ and $\{\w\om{\rm re}_k(t,\eta)\}_{k\neq 0}$ and we will only consider their values at a fixed $\eta$. Furthermore, classical $\Delta$ gets transformed to $\pa_z^2+(\pa_v-t\pa_z)^2$, so the Biot-Savart law yields the damping factor $\f{k}{k^2+(\eta-kt)^2}$.  The first level approximation $\pa_y\w\om{a}_{-k}$ is replaced by the  upper bound of its absolute size $\frac{1}{\langle k\rangle^{2}}\nu^{\f13}te^{-\nu^{\f13}t}\lesssim \frac{1}{\langle k\rangle^2}e^{-\f12\nu^{\f13}t}$. Since the regularity of $\w\om{a}$ is inherited from the initial data, we can introduce an extra $\langle k\rangle^{-2}$-weight to guarantee $k$-summability. 

Now we can do an energy estimate, seasoned with H\"older inequality, to determine the size of $|f_0|^2$ for $t\geq T_0$, 
\begin{align*}
|f_0(t)|^2-|f_0(T_0)|^2\leq& 2\Re\int_{T_0}^\infty \lf(\sum_{k\in\mathbb Z\backslash\{0\}}\f{f_k(t)}{k^2+(\eta-kt)^2}\f{t\nu^{\f13}e^{-\nu^{\f13}t}}{\langle k\rangle}\, f_0\rg)dt\\
\leq &2\lf(\int_{T_0}^\infty\sum_{k\in\mathbb Z\backslash\{0\}}\frac{|f_k(t)|^2}{\langle k\rangle^2(k^2+(\eta-kt)^2)} dt\rg)^{\f12}\lf(\int_{T_0}^\infty\sum_{k\in\mathbb Z\backslash\{0\}}\frac{|f_0|^2}{k^2+(\eta-kt)^2}dt\rg)^{\f12}.
\end{align*}
For the first factor on the right-hand side, we observe that the denominator has sufficient decay to guarantee time-integrability. However, one needs to control the second factor $\|(k^2+(\eta-kt)^2)^{-\frac{1}{2}}f_0\|_{L^2_tl^2_k}$. To this end, we can introduce the time-dependent Fourier multiplier on $\Torus\times\mathbb R$,
\begin{align}\label{eq: Fourier multiplier}
    \sum_{k\in\mathbb Z\backslash\{0\}}\frac{1}{|k|^3}\arctan(\eta/k-t). 
\end{align}
When one takes the time-derivative of its associated norm, extra damping term ($\mathcal{CK}$-term) appears to provide the necessary control over $\|(k^2+(\eta-kt)^2)^{-\frac{1}{2}}f_0\|_{L^2_tl^2_k}$. 
However, the Fourier analysis is not well-adapted to the case with boundary. To apply the above idea under our Robin boundary condition setting, on the physical side, we introduce the singular integral operator $\wt{\mathfrak{J}}_0$ in \eqref{J0}, which corresponds to the Fourier multiplier \eqref{eq: Fourier multiplier}. The inner product $\langle \pa_t\wt{\mathfrak{J}}_0[\w\om{\rm re}_0],\w\om{\rm re}_0\rangle$ can give us a good coercive term $-\sum_{k\in\mathbb Z\backslash\{0\}}|k|^{-1}\|\na^L_k(\D_k^L)^{-1}\w\om{\rm re}_0\|_{L^2}^2$ (see, e.g., \eqref{dfn_G_k0}), then we can close the proof. More details are provided in Section \ref{sec-longT_z}.

\subsection{Bootstrap}
In the sequel, we apply a bootstrap argument to develop the estimates for the other  components in this decomposition. We assume that $[0, T_\ast)\subset [0,\infty)$ is the maximal interval on which the following bootstrap assumptions hold
\begin{subequations}\label{Btstrp_Hyp}
\begin{enumerate}
\item The second level resonant component $\w\om{\rm re}_{\ne}$ estimates
\begin{equation}\label{hyp_h_om_re}
    \begin{aligned}
       & \|e^{\ep_0\nu^{\f13}t}\w\om{\rm re}_{\ne}\|_{L^{\infty}_t([T_0,T_*);L_{x,y}^2)}+ \nu^{\f12}\|e^{\ep_0\nu^{\f13}t}\na\w\om{\rm re}_{\ne}\|_{L^{2}_t([T_0,T_*);L_{x,y}^2)}\\
        &\quad+\|e^{\ep_0\nu^{\f13}t}\pa_x\na(\D^{-1})\w\om{\rm re}_{\ne}\|_{L^{2}_t([T_0,T_*);L_{x,y}^2)}+\nu^{\f14}\|e^{\ep_0\nu^{\f13}t}|D_x|^{\f12}\w\om{\rm re}_{\ne}\|_{L^{2}_t([T_0,T_*);L_{x,y}^2)}\\
        &\leq 4\ep\nu^{\f23}|\ln\nu|.
    \end{aligned}
\end{equation}
Since the $\w\omega{\rm re}$ is only nontrivial after the transition time $T_0:=\nu^{-\f16}$, we only measure it for $t\geq T_0$. Here, we define $(|D_x|^{\f12}f)^\wedge(k,y)=|k|^\f12\widehat{f}(k,y)$.
\item The second level error $\w\om{\rm e}_{\ne}$ estimate 
\begin{align}
 \text{Short Time:}\qquad &\|e^{\ep_0\nu^{\f13}t}\w\om{*,i}_{\ne}\|_{L_t^{\infty}([0,T_*];L^2_{x,y})}+\nu^{\f12}\|e^{\ep_0\nu^{\f13}t}\na\w\om{*,i}_{\ne}\|_{L^2_t([0,T_*] ;L_{x,y}^2)}\label{hyp_om_e_s}\\
 &+\nu^{\f{1}{12}}\|e^{\ep_0\nu^{\f13}t}\pa_x\na(\D^{-1})\w\om{*,i}_{\ne}\|_{L^2_t([0,T_*] ;L_{x,y}^2)}\le 4\ep\nu^{\f23}|\ln\nu|,\quad \text{if $T_*\le T_0$;}\notag
 \end{align}
 \begin{align}
 \text{Long Time:}\qquad
  & \f12\Big(\|e^{\ep_0\nu^{\f13}t}\w\om{\rm e}_{\ne}\|_{L^{\infty}_t([T_0,T_*);L_{x,y}^2)}+ \nu^{\f12}\|e^{\ep_0\nu^{\f13}t}\na\w\om{\rm e}_{\ne}\|_{L^{2}_t([T_0,T_*);L_{x,y}^2)}\label{hyp_om_e_l}\\
&\quad+\|e^{\ep_0\nu^{\f13}t}\pa_x\na(\D^{-1})\w\om{\rm e}_{\ne}\|_{L^{2}_t([T_0,T_*);L_{x,y}^2)}\notag\\    &\quad +\nu^{\f14}\|e^{\ep_0\nu^{\f13}t}|D_x|^{\f12}\w\om{\rm e}_{\ne}\|_{L^{2}_t([T_0,T_*);L_{x,y}^2)}\Big)\leq 4\ep\nu^{\f23}|\ln\nu|.\quad \text{if $T_*\ge T_0$;}\notag
\end{align}
\item The second level boundary corrector estimate
\begin{equation}
\begin{aligned}
&\nu^{-\f16}\|e^{\ep_0\nu^{\f13}t}(1-|y|)|D_x|^{\f{11}{30}}\pa_x\w\omega{*,b}_\nq\|_{L_t^{\infty}[0,T_*);L_{x,y}^2)}+\|e^{\ep_0\nu^{\f13}t}|D_x|^{\f15}\pa_x\na\w\p{*,b}_{\ne}\|_{L^2_t([0, T_\ast);L_{x,y}^2 )}\\
&\quad+\nu^{\f13}\|e^{\ep_0\nu^{\f13}t}|D_x|^{\f{13}{15}}\w\om{*,b}_{\ne}\|_{L_t^{2}([0,T_*);L_{x,y}^2)}+\nu^{\f{7}{18}}\|e^{\ep_0\nu^{\f13}t}\pa_x\w\om{*,b}_{\ne}\|_{L_t^{2}([0,T_*);L_{x,y}^2)}\leq 4\ep \nu|\ln\nu|.\label{hyp_om_*b}
\end{aligned}
\end{equation}
\item The second level zero mode estimate
\begin{align}
\text{Short Time:}\qquad&\|\w\om{*}_0\|_{L^\infty_t([0,T_0);L^2_y)}+\sup_{t\in[0,T_0]}|u_0^{(*)}(t,\pm1)|\label{hyp_om_0_s} \\
&\quad +\nu^{\f12}\|\pa_y\w\om{*}_0\|_{L_t^2L_y^2}\leq 4\ep\nu^{\f23}|\ln\nu|,\quad \text{if $T_*\le T_0$;}\notag
\\
\text{Long Time:}\qquad&\f12\Big(\|\w\omega{\rm re}_0\|_{L^\infty_t([T_0,T_*);L^2_y)}+\|\w\omega{\e}_0\|_{L^\infty_t([T_0,T_*);L^2_y)}+\sup_{t\in[T_0,T_*)}|u_0^{(\rm re)}(t,\pm1)|\label{hyp_om_0_l}\\
&\quad+\sup_{t\in[T_0,T_*)}|u_0^{(\rm e)}(t,\pm1)|\Big)\leq 4\ep\nu^{\f23}|\ln\nu|,\quad \text{if $T_*\ge T_0$.}\notag 
\end{align}
Here, we recall that we only decompose the $\w\omega*_0$ into the $\w\omega{\rm re}_0$ and $\w\omega{\e}_0$ for $t\geq T_0$.
\end{enumerate}
\end{subequations}

The goal is to develop the following proposition on the time interval $[0,T_\ast)$:
\begin{proposition}\label{pro:Boot}
   Assume the bootstrap hypotheses \eqref{Btstrp_Hyp}.  If the parameter $\ep$ and the viscosity threshold $\nu_0$ are chosen small enough, the stronger conclusions with all the 4's replaced by 2's hold on the same time interval $[0,T_\ast)$. 
\end{proposition}
\subsection{Three Main Families of Bounds}
In this subsection, we arrange the estimates to be proven in three groups and discuss how they are derived. These three families are short-time estimates (good unknown), long-time nonzero mode estimates (I.D. \& E.D. SIO) and long-time zero mode estimates (Cascade SIO).  
\subsubsection{Short-time Estimates}
For the whole bootstrapping horizon $[0, T_{*})$, we decompose the solution $\w\omega*$ as the $\w\omega {*}_\nq$ and $\w\omega *_0.$ The equation for $x$-average is straightforward:  
\begin{equation}\label{eq-om*0}
\begin{aligned}
    \begin{cases}
        \pa_t\w\om{*}_0-\nu\pa_y^2\w\om{*}_0=-\pa_y(u^2\om)_0,\\
        \mathbb T_{\pm}[\w\om{*}_0]=0,\quad \w\om{*}_0(t=0)=0.
    \end{cases}
\end{aligned}
\end{equation}
For the non-zero mode $\w\om*_\nq$, the effect of the boundary effect is significant. As a consequence, we reformulate the equation for $\w\om{*}_\nq$ as follows with the quantities $\w\om{*,i}_\nq$ and $\w\om{*,b}_\nq$ \eqref{omega*_nq_dcmp}:
\begin{align}
    \begin{cases}
\pa_t\w\om{*,i}_\nq+\mathcal{U}(t,y)\pa_x\w\om{*,i}_\nq-\nu\de\w\om{*,i}_\nq=\pa_x\msc{E}[\w\om{*,i}_\nq]+\mathbb F_\nq,\\
         \de\w\psi{*,i}_\nq=\w\om{*,i}_\nq,\quad \w\om{*,b}_\nq=\msc{C}[\w\om{*,i}_\nq],\\
        \w\om{*,i}_\nq|_{t=0}=0,\quad \w\om{*,i}_\nq(t,\pm1)=0,\quad (t,x,y)\in {[0,\min\{T_0, T_*\}]}\times \Torus\times[-1,1].
    \end{cases}\label{eq:om_*i_nq}
\end{align} 
We recall the boundary correction operator $\msc C$ and frozen error operator $\msc E$ defined in \eqref{C-op} and \eqref{Frozen_err}. The forcing $\mathbb{F}_\nq$ is as follows:
\begin{align}\label{F_nq}\begin{split}
   \mathbb F_\nq=&\pa_y^2\mathcal{U}\pa_x(\psi_\nq-\w\psi a_\nq)-\big({u}\cdot\na(\om-\w\om a)\big)_\nq-\big(({u}-{u}^{(a)})\cdot\nabla \w\om a\big)_\nq\\
    &-\Big(\nu t^2 \pa_x^2\w\om{a}_\nq-\nu\D\w\om{a}_\nq-\pa_y^2\mathcal{U}\pa_x\w\psi a_\nq+(\w{{u}}a\cdot\nabla\w\om a)_\nq\Big)
    +\pa_x\mathscr{E}[{\w\om{a}_{\ne}}].\end{split}
\end{align}
In the forcing expression, the $\om, u,\psi$ denote the solution to the full problem \eqref{eq-I} and we can view them as given inputs to the $\w\om{*,i}$-equation. 

Thanks to Proposition \ref{prop-C} (with $g_\nq$ being $\w\omega{*,i}_\nq$), suitable estimates for $\w\omega{*,i}_\nq$ yield the bounds for $\w\omega{*,b}_\nq$. Therefore, we mainly focus on the interior component $\w\omega{*,i}_\nq$ in the sequel. 
 
We introduce a good unknown 
\begin{align}\label{eq: good-unknown}
\sum_{k\in\mathbb Z\backslash\{0\}}e^{ikx}\w \omega{*,i}_k+\w\omega*_0    
\end{align} 
which does not satisfy the physical boundary condition (see \eqref{eq-om*0} and \eqref{eq:om_*i_nq}), but keeps the favorable transport structure. The corresponding modified energy on the initial time interval $t\in [0,T_0=\nu^{-\f16}]$ is considered:
\begin{align}\label{E_int}
E_{\rm initial}:=     & \sum_{k\in\mathbb Z\backslash\{0\}}{\|\w\om{*, i}_k\|_{L^2}^2}+\|\w\om*_0\|_{L^2}^2.
\end{align}
     The novelty of this energy functional is that it exploits the divergence-free structure of the fluid velocity to cancel key nonlinear contributions. Specifically, the main obstacles are the trilinear interactions involving the approximate vertical velocity $\w u{a,2}_\nq$, the zero-mode vorticity $ \w\om*_0$ and the nonzero mode interior error $\w\om{*,i}_\nq$. These interactions appear in pairs --- both in the $\w\om {*,i}_\nq$ evolution \eqref{inner-we} and in the $\w\om {*}_0$ evolution \eqref{eq-w0-inner}. By packaging these unknowns to form \eqref{eq: good-unknown}, we introduce a vital cancellation, which yields a favorable control. A detailed discussion is presented in Section \ref{sec-shortT}.
     
The following proposition characterizes the bounds of the energy functional.
\begin{proposition}\label{pro:shortT}
    For all $t\in [0,T_*)\subset [0,T_0=\nu^{-\f16}]$, under the bootstrap assumption \eqref{Btstrp_Hyp}, for $\ep$ chosen small enough and $k\in\mathbb Z\backslash\{0\}$, the following estimate holds
    \begin{align*}
      &E_{\rm initial}+\nu^{\f16}\||k|\na_k\w\p{*,i}_k\|_{L_t^2([0,T_*);l_k^2L^2_y)}^2\\
      &+\nu\|\na_k\w\om{*,i}_k\|_{L_t^2([0,T_*);l_k^2L^2_y)}^2+\sup_{t\in[0,T_*)}|u_0^{(*,i)}(t,\pm1)|^2\le\ep^2 \nu^{\f43}|\ln\nu|^2.
    \end{align*} 
\end{proposition}
The proof of the proposition is presented in Section \ref{sec-shortT}.
As a consequence of the proposition, we have obtained the improved estimate to \eqref{hyp_om_e_s} and \eqref{hyp_om_0_s}, since in this time region, the time weight $e^{\ep_0\nu^{\f13}t}\approx 1$.

\subsubsection{Long-time Nonzero Mode Estimates}
For the long time, we recall the resonant component $\w\omega{\rm re}_\nq$ \eqref{eq-om_re} (which is only nontrivial for $t\geq T_0$), and define the second order error 
\begin{align*}
\w\omega\e_\nq:= \w\omega{*,i}_\nq-\w\omega{\rm re}_\nq,\quad \forall t\geq 0. 
\end{align*}
The $\w\omega\e_\nq$ solves the following equation
\begin{align}\label{eq-we}\begin{cases}
\pa_t\w\om{\e}_\nq+\mathcal{U}(t,y)\pa_x\w\om{\e}_\nq-\nu\de\w\om{\e}_\nq=\pa_x\msc{E}[\w\om{\e}_\nq]+\w{\mathbb F}\e_{\ne},\quad
         \de\w\psi{\e}_\nq=\w\om{\e}_\nq,\\
        \w\om{\e}_\nq|_{t={T_0}}=\w\om{*,i}(T_0)
        ,\quad \w\om{\e}_\nq(t,\pm1)=0,\quad (t,x,y)\in [T_0,T_*)\times \Torus\times[-1,1],
\end{cases}\end{align}
where $\mathscr{E}[\w\omega\e_\nq]$ is the boundary corrector-induced error \eqref{Frozen_err}, and the force $\w{\mathbb{F}}\e_\nq$ is recorded below \begin{align}
\label{F_e}\begin{split}
    &\w{\mathbb F}{\e}_\nq:=\pa_y^2\mathcal{U}(t,y)\pa_x(\psi_\nq-\w\psi a_\nq)-\big({u}\cdot\na(\om-\w\om a)\big)_\nq-\big((u^1-u^{(a,1)})\pa_x\w\om a\big)_\nq\\
    &\hspace{1cm}+\big(-\nu t^2\pa_x^2\w\om{a}_\nq+\nu\D\w\om{a}_\nq+\pa_y^2\mathcal{U}\pa_x\w{\psi}{a}_\nq-(\w{{u}}a\cdot\nabla\w\om a)_\nq\big)\\ &\hspace{1cm}-\big((u^2-u^{(a,2)})(\pa_y+t\pa_x)\w\om a\big)_\nq+\pa_x\mathscr E[\w\omega a+\w\omega {\rm re}_\nq].\end{split}
\end{align}

\begin{proposition}\label{lem-re}
Under the Bootstrap assumptions \eqref{hyp_h_om_re} and \eqref{hyp_om_e_l} , the following estimate holds for all $t\in[T_0,T_*)$, 
\begin{align*}
     & \|e^{\ep_0\nu^{\f13}t}\w\om{\rm re}_{\ne}\|_{L^{\infty}_t([T_0,T_*);L_{x,y}^2)}+ \nu^{\f12}\|e^{\ep_0\nu^{\f13}t}\na\w\om{\rm re}_{\ne}\|_{L^{2}_t([T_0,T_*);L_{x,y}^2)}\\
        &+\|e^{\ep_0\nu^{\f13}t}\pa_x\na(\D^{-1})\w\om{\rm re}_{\ne}\|_{L^{2}_t([T_0,T_*);L_{x,y}^2)}+\nu^{\f14}\|e^{\ep_0\nu^{\f13}t}|D_x|^{\f12}\w\om{\rm re}_{\ne}\|_{L^{2}_t([T_0,T_*);L_{x,y}^2)}\le \ep \nu^{\f23}|\ln\nu|.
\end{align*}   
\end{proposition}

We have the following proposition.
\begin{proposition}\label{prop-e}
Under the bootstrap assumption \eqref{Btstrp_Hyp}, the following holds
    \begin{align*}
          &\|e^{\ep_0\nu^{\f13}t}\w\om{\rm e}_{\ne}\|_{L^{\infty}_t([T_0, T_\ast);L^2_{x,y})} +\|e^{\ep_0\nu^{\f13}t}\na\pa_x\w\p{\rm e}_{\ne}\|_{L^2_t([ T_0, T_\ast);L^2_{x,y})}\\
          &\quad+\nu^{\f14}\||D_x|^{\f12}e^{\ep_0\nu^{\f13}t}\w\om{\rm e}_{\ne}\|_{L^{2}_t([T_0, T_\ast);L^2_{x,y})}+\nu^{\f12}\|e^{\ep_0\nu^{\f13}t}\na\w\om{\rm e}_{\ne}\|_{L^2_t([T_0, T_\ast);L^2_{x,y})}\le 2\ep \nu^{\f23}|\ln\nu|.
    \end{align*}
\end{proposition}

The main tool to develop these estimates are the inviscid damping and enhanced dissipation SIO $\mf J_k, \mf J^{(\e)}_k$-induced norms. Detailed arguments involve technical energy type estimates, and will be presented in Section \ref{sec:Linearized System}.  

Since all the estimates for $\w\om{*,i}_\nq$ is now completed, we can state the estimate for the estimate for the boundary corrector $\w\om{*,b}_\nq$. 

\begin{proposition}\label{prop-om*b}
Under the bootstrap assumption \eqref{Btstrp_Hyp}, the following holds
    \begin{align*}
          &\sum_{k\in\mathbb Z\backslash\{0\}}\Big(\nu^{-\f13}\|e^{\ep_0\nu^{\f13}t}|k|^{\f{41}{30}}(1-|y|)\w\om{*,b}_k\|_{L^{\infty}_t([0, T_\ast);L^2_y)}^2+\|e^{\ep_0\nu^{\f13}t}|k|^{\f65}(k,\pa_y)\D_k^{-1}\w\om{*,b}_k\|_{L^2_t([0, T_\ast);L^2_y)}^2\\
          &\qquad\qquad+\nu^{\f23}\|e^{\ep_0\nu^{\f13}t}|k|^{\f{13}{15}}\w\om{*,b}_k\|_{L^{2}_t([0, T_\ast);L^2_y)}^2+\nu^{\f79}\|e^{\ep_0\nu^{\f13}t}|k|\w\om{*,b}_k\|_{L^{2}_t([0, T_\ast);L^2_y)}^2\Big)\le \ep^2\nu^{2}|\ln\nu|^2 .
    \end{align*}
\end{proposition}
\begin{proof}This proposition follows directly from  Corollary \ref{prop-eb} and Propositions \ref{pro:shortT}-\ref{prop-e}.
\end{proof}
\subsubsection{Long-time Zero Mode Estimates}
Instead of tracking the vorticity directly, we track the corresponding average velocity $\w u*_0=u_0^{\rm (re)}+u_0^{(\e)}$ in the long time regime. Recall that $\w u{\rm re}_0$ satisfies:
    \begin{equation}\label{eq-0re}
    \begin{cases}
        \pa_tu_0^{\rm (re)}-\nu\pa_y^2u_0^{\rm (re)}=\displaystyle\sum_{k\in\mathbb Z\backslash\{0\}}ik\big(\p_k-\p_k^{(a)}\big)\om_{-k}^{(a)}=:\mathbb F_0,\\
        \al \pa_yu_0^{\rm (re)}(t,\pm1)=\mp u_0^{\rm (re)}(t,\pm1),\quad u_0^{\rm (re)}|_{t=T_0}=0.
        \end{cases}
    \end{equation}
  Then, the error $\w u\e_0$ solves the equation,
    \begin{equation}\label{eq-0*}
        \begin{cases}
            \pa_tu_0^{(\e)}-\nu\pa_y^2u_0^{(\e)}=\displaystyle\sum_{k\in\mathbb Z\backslash\{0\}}ik\p_k\om_{-k}-\sum_{k\in\mathbb Z\backslash\{0\}}ik\big(\p_k-\p_k^{(a)}\big)\om_{-k}^{(a)},\\
            \al\pa_yu_0^{\rm (e)}(t,\pm1)=\mp u_0^{(\e)}(t,\pm1),\quad u_0^{(\e)}|_{t=T_0}=\w u{*}_0(T_0).
        \end{cases}
    \end{equation}
One can develop the following propositions concerning the $u_0^{\rm (re)},\, u_0^{\rm (e)}$:
\begin{proposition}\label{prop-wre0} Under the bootstrap assumptions \eqref{Btstrp_Hyp}, the following estimate holds
\begin{align*}
    \|\w\om{\rm re}_0\|_{L^{\infty}_t([T_0,T_*);L_y^2)}+\nu^{\f12}\|\w\om{\rm re}_0\|_{L^2_t([T_0,T_*);H^1_y)}+\sup_{t\in[T_0,T_*)}|u_0^{(\rm re)}(t,\pm1)|\le \ep \nu^{\f23}|\ln\nu|.
\end{align*}
    
\end{proposition}

\begin{proposition} \label{prop-w*0}Under the bootstrap assumptions \eqref{Btstrp_Hyp}, it holds that
\begin{align*}
\|\om^{(\e)}_0\|_{L^{\infty}_t([T_0,T_*);L_y^2)}+\sup_{t\in[T_0,T_*)}|u_0^{(\rm e)}(t,\pm1)|\le \ep \nu^{\f23}|\ln\nu|.
\end{align*} 
\end{proposition}
The proof of these two propositions make use of the cascade SIO \eqref{J0} and the details are in Section \ref{sec-longT_z}.

\subsection{Notations}
We have adopted several notation conventions in the paper and list them as follows:\begin{itemize} 
\item We use the notation $g_k(y)$ to denote the $x$-Fourier transform of the function $g(x,y)$. 
\item $\br{f}$ denotes the Japanese bracket $\br f:=(1+|f|^2)^\f12. $

\item For two nonnegative quantities $A,B$, we use the notation $A\lesssim B$ ($A\gtrsim B$) to denote the fact that there exists $C\geq 1$ such that $A\leq C B$ ($A\geq B/C$). Moreover, $A\approx B$ means that there exists $C\geq 1$ so that $ C^{-1}B\leq A\leq CB$.

\item Define the horizontal derivative by
$
    \widehat{|D_x|^{\g}f}(k,y)=|k|^{\g}\hat{f}(k,y).
$
\item We define the $H_{k,y}^M$-Sobolev norm as follows
\begin{align}
    \label{defn_H_k}\|f_k\|_{H_{k,y}^M}^2:=\sum_{m+n\leq M}\|k^{m}\pa_y^nf_k \|_{L_{y}^2}^2.
\end{align}

\item The notations $T_1, \ T_2, \ T_3,...$, and $J_1,\ J_2,\ J_3,...$ represent various terms that appear in the proof of theorems or lemmas. They are ``local variables" and will be redefined when one moves to the proof of a different  theorem/lemma. 

\end{itemize}

\section{Interior Approximate Solutions}\label{sec: appro-sol}
In this section, we consider the first-level approximate solution $\w\omega a_k$ to the equation \eqref{eq-a}. We will develop the estimates stated in Proposition \ref{ol:estimate-u1}. 

First of all, we develop the following $k$-by-$k$ estimates. The proof of Proposition \ref{ol:estimate-u1} is a direct consequence of Lemma \ref{lem:estimate-u1_k} and summation in $k$. 
\begin{lemma}\label{lem:estimate-u1_k}
For $t\ge 0$ and all $k\in \mathbb{Z}\backslash\{0\}$, the following estimates hold
\begin{align*}
   &  \|k\w\om{a}_k\|_{L^{\infty}}+\|k\w\om{a}_k\|_{L^{2}}+\|(\pa_y+ikt)\w\om{a}_k\|_{L^{\infty}}\lesssim  e^{-2\nu^{\f13}|k|^{\f23}t}\|\om_{k;\rm in}\|_{H_{k,y}^2},\\
 &\|\pa_y\w\p a _k\|_{L^2\cap L^{\infty}}+\langle t\rangle\|k\w\p a _k\|_{L^2\cap L^{\infty}}\lesssim  |k|^{-\f32}\langle t\rangle^{-1} e^{-2\nu^{\f13}|k|^{\f23}t}\|\om_{k;\rm in}\|_{H_{k,y}^2},\\
&\| \nu k^2t^2\w\om{a}_k+\nu\D_k\w\om{a}_k\|_{L^2}\lesssim \nu^{\f23}e^{-2\nu^{\f13}|k|^{\f23}t}\|\om_{k;\rm in}\|_{H_{k,y}^2},\\
&\|(\w{ u }a\cdot\nabla\w\om a)_k\|_{L^2}\lesssim e^{-4\nu^{\f13}|k|^{\f23}t}\langle t\rangle ^{-1}\sum_{l\in\mathbb Z\backslash\{0,k\}}|l|^{-2}\|\om_{l;\rm in}\|_{H_{l,y}^2}\|\om_{k-l;\rm in}\|_{H_{k-l,y}^2}.
\end{align*}
\end{lemma}
\begin{proof}
    First of all, we apply the integration factor method to derive the close-form formula for the solution to \eqref{eq-a}:
\begin{align}\label{expre-wa}
    \w\om a_k(t,y) =e^{-ik\mathcal{U}_1(t,y)-\f13\nu k^2t^3}\om_{k;\rm in}(y),\quad \mathcal{U}_1(t,y):=\int_0^t\mathcal{U}(\tau,y)d\tau.
\end{align} 
By the regularity estimate \eqref{ineq-U} and Gagliardo-Nirenberg inequality, we get
\begin{equation*}
\begin{aligned}
  \big\| (\pa_y+ikt)\w\om{a}_k\big\|_{L_y^\infty}&=\Big\|ik e^{-ik\mathcal{U}_1(t,y)-\f13\nu k^2 t^3}\om_{k;\rm in}\int_0^t\big(1-\pa_y\mathcal{U}(\tau,y)\big)d\tau\Big\|_{L_y^\infty}\\
   &\lesssim \nu^{\f13}t{|k|^{\f12}}e^{-3\nu^{\f13}|k|^{\f23}t}\|\om_{k;\rm in}\|_{H_{k,y}^1}\lesssim {|k|^{-\f16}}e^{-2\nu^{\f13}|k|^{\f23}t}\|\om_{k;\rm in}\|_{H_{k,y}^1}.
\end{aligned}
\end{equation*}
Therefore, it is straightforward to check the relation:
\begin{align*}
\|k\w\om{a}_k\|_{L^2}+\|k\w\om{a}_k\|_{L^{\infty}}+\|(\pa_y+ikt)\w\om{a}_k\|_{L^{\infty}}\lesssim e^{-2\nu^{\f13}|k|^{\f23}t}\|\om_{k;\rm in}\|_{H_{k,y}^2}.
\end{align*} 
By \eqref{ineq-U}, we obtain
\begin{equation*}
    \|\mathcal{U}_1(t)-ty\|_{C^3}\lesssim t\|W_{0,\rm in}\|_{H^6}\lesssim \nu^{\f13}t.
\end{equation*}
Next, through integration by parts and the fact that $G_k(y,y'=\pm1)=\pa_yG_k(y,y'=\pm1)=0$, we obtain 
\begin{align*}
    \pa_y \w\p a_k 
    &=\int_{-1}^1\pa_{y'}\Big(\pa_y G_k(y,y')\om_{k;\rm in}(y')e^{-\f13\nu k^2 t^3}\big(ik\pa_{y'}\mathcal{U}_1(t,y')\big)^{-1}\Big)e^{-ik\mathcal{U}_1(t,y')}dy'\\
    & \qquad+\om_{k;\rm in}(y)e^{-ik\mathcal{U}_1(t,y)-\f13\nu k^2t^3}\big(-ik\pa_y\mathcal{U}_1(t,y)\big)^{-1}.
\end{align*}
We note that the derivative of the shear is strictly positive $ |\pa_y\mathcal{U}|\geq \f12$ by \eqref{ineq-U}. As a consequence, 
\begin{align*}
    | \pa_y \w\p a_k |&\lesssim e^{-2\nu^{\f13}|k|^{\f23}t}|k|^{-1}\langle t\rangle^{-1}\big(|k|^{\f12}\|\om_{k;\rm in}\|_{L^2}+|k|^{-\f12}\|\om_{k;\rm in}\|_{H_{k,y}^1}+\nu^{\f13}|k|^{-\f12}\|\om_{k;\rm in}\|_{L^2}\big)\\
    &\lesssim e^{-2\nu^{\f13}|k|^{\f23}t}\langle t\rangle^{-1}|k|^{-\f32}\|\om_{k;\rm in}\|_{H_{k,y}^1}.
\end{align*}
Then it immediately follows that
\begin{align*}
\|\pa_y\w\p{a}_k\|_{L^{\infty}}+\|\pa_y \w\p a _k\|_{L^2}\lesssim|k|^{-\f32}\langle t\rangle^{-1} e^{-2\nu^{\f13}|k|^{\f23}t}\|\om_{k;\rm in}\|_{H_{k,y}^1}.
\end{align*}
Moreover, for the $ik\w\p{a}_k$ part, we can do integration by parts one more time, and it holds
\begin{align*}
    \w\p{a}_k
    &=\pa_{y'}\Big(G_k(y,y')\om_{k;\rm in}(y')e^{-\f13\nu k^2 t^3}\big(ik{\pa_{y'}\mathcal{U}_1(t,y')}\big)^{-1}\Big)\f{e^{-ik\mathcal{U}_1(t,y')}}{ik{\pa_{y'}\mathcal{U}_1(t,y')}}\Big|_{y'=-1}^{y'=y}\\
    &\quad +\pa_{y'}\Big(G_k(y,y')\om_{k;\rm in}(y')e^{-\f13\nu k^2 t^3}\big(ik{\pa_{y'}\mathcal{U}_1(t,y')}\big)^{-1}\Big)\f{e^{-ik\mathcal{U}_1(t,y')}}{ik{\pa_{y'}\mathcal{U}_1(t,y')}}\Big|_{y'=y}^{y'=1}\\
    &\quad+\int_{-1}^1\pa_{y'}^2\Big(G_k(y,y')\om_{k;\rm in}(y')e^{-\f13\nu k^2 t^3}\big(ik{\pa_{y'}\mathcal{U}_1(t,y')}\big)^{-1}\Big)
\big(ik{\pa_{y'}\mathcal{U}_1(t,y')}\big)^{-1}e^{-ik\mathcal{U}_1(t,y')}dy'\\
   &\quad +\int_{-1}^1\pa_{y'}\Big(G_k(y,y')\om_{k;\rm in}(y')e^{-\f13\nu k^2 t^3}\big(ik{\pa_{y'}\mathcal{U}_1(t,y')}\big)^{-1}\Big)\pa_{y'}\Big(\big(ik{\pa_{y'}\mathcal{U}_1(t,y')}\big)^{-1}\Big)e^{-ik\mathcal{U}_1(t,y')}dy',
\end{align*}
combining with the expression of $G_k(y,y')$ and the properties of $\mathcal{U}(t,y)$, we obtain
\begin{equation}\label{est-wa-k}
\begin{aligned}
    |\w\p{a}_k|&\lesssim |k|^{-2}\langle t\rangle^{-2}e^{-2\nu^{\f13}|k|^{\f23}t}\big(|k|^{-\f12}\|\om_{k;\rm in}\|_{H_{k,y}^1}+|k|^{\f12}\|\om_{k;\rm in}\|_{L^2}+|k|^{-\f12}\|\pa_y\om_{k;\rm in}\|_{L^2}\\
    &\quad +\nu^{\f13}|k|^{-\f12}\|\om_{k;\rm in}\|_{L^2}+|k|^{-\f32}\|\pa_y^2\om_{k;\rm in}\|_{L^2}+\nu^{\f13}|k|^{-\f32}\|\om_{k;\rm in}\|_{L^2}\\
    &\lesssim |k|^{-2}\langle t\rangle^{-2}e^{-2\nu^{\f13}|k|^{\f23}t}|k|^{-\f32}\|\om_{k;\rm in}\|_{H_{k,y}^2},
\end{aligned}
\end{equation}
therefore,
\begin{align*}
 |k|  \|k\w\p{a}_k\|_{L^{\infty}}+ |k|\|k\w\p{a}_k\|_{L^2}\lesssim \langle t\rangle^{-2}e^{-2\nu^{\f13}|k|^{\f23}t}|k|^{-\f32}\|\om_{k;\rm in}\|_{H_{k,y}^2}.
\end{align*}
Direct calculation shows
\begin{align*}
    \nu t^2k^2\w\om{a}_k+\nu\D_k\w\om{a}_k=&\nu e^{-ik\mathcal{U}_1(t,y)-\f13\nu k^2t^3}\Big(\pa_y^2\om_{k;\rm in}-2ik\pa_y\mathcal{U}_1(t,y)\om_{k;\rm in}\\
    &-k^2(\pa_y\mathcal{U}_1(t,y))^2\om_{k;\rm in}+k^2(t^2+1)\om_{k;\rm in}-ik\pa_y^2\mathcal{U}_1(t,y)\om_{k;\rm in}\Big),
\end{align*}
which gives that 
\begin{equation}\label{est-k-wa1}
\begin{aligned}
   \|\nu t^2k^2\w\om{a}_k+\nu\D_k\w\om{a}_k\|_{L^2}    \lesssim \nu^{\f23}e^{-2\nu^{\f13}|k|^{\f23}t}\|\om_{k;\rm in}\|_{H_{k,y}^2},
\end{aligned}
\end{equation}
and
\begin{equation}\label{est-wa-k2}
    \begin{aligned}
      \|  (\w{ u }{a}\cdot\na\w\om{a})_k\|_{L^2}
      &\le \sum_{l\in\mathbb Z\backslash\{0,k\}}\Big(\|\pa_y\p^{(a),1}_l\|_{L^2}\|(k-l)\w\om{a}_{k-l}\|_{L^{\infty}}+\|l\p_l^{(a),2}\|_{L^2}\|\pa_y\w\om{a}_{k-l}\|_{L^{\infty}}\Big)\\
      &\lesssim e^{-4\nu^{\f13}|k|^{\f23}t}\langle t\rangle ^{-1}\sum_{l\in\mathbb Z\backslash\{0,k\}}|l|^{-2}\|\om_{l;\rm in}\|_{H_{l,y}^2}\|\om_{k-l;\rm in}\|_{H_{k-l,y}^2}.
    \end{aligned}
\end{equation}
This finishes the proof.
\end{proof}

\section{Boundary Correctors}\label{sec-boundary}In this section, we first prove Proposition \ref{prop-C} and Lemma \ref{Lem-F-Y}. Then we use Proposition \ref{prop-C} to prove Corollary \ref{prop-ba} and \ref{prop-eb}. 

\subsection{Proof of Proposition \ref{prop-C}}
\begin{proof}
\step{1. Estimates of $\mathscr{C}_k^{[j]}[g_k]$}
 Consider \eqref{C-op}, after $x$-Fourier transform, the boundary condition of $\mathscr{C}^{[j]}_k[g_{k}]$ is 
\begin{align*}
    \al\mathscr{C}^{[j]}_k[g_{k}](\pm1)\pm\pa_y(\Delta_k)^{-1}\mathscr{C}^{[j]}_k[g_{k}](\pm1)=\big(-\al g_k(\pm1)\mp\pa_y(\Delta_k)^{-1}g_k(\pm1)\big)\boldsymbol\chi_{j}(t).
\end{align*}
Introduce
\begin{equation}\label{defn-apm}
\begin{aligned}
    &a_+^{[j]} (t)=\big(-\al g_k(1)-\pa_y(\Delta_k)^{-1}g_k(1)\big)\boldsymbol\chi_j,\\
    &a_-^{[j]} (t)=\big(-\al g_k(-1)+\pa_y(\Delta_k)^{-1}g_k(-1)\big)\boldsymbol\chi_j.
\end{aligned}
\end{equation}
Here, we note that since the initial data $\mathscr{C}_k[g_k](t=0)$ vanishes on the whole domain, all the expressions $\lf\{\mathscr{C}^{[j]}_k[g_k]),(\Delta_k)^{-1}\mathscr{C}^{[j]}_k[g_k],\pb a j_\pm\rg\}$ can be  continuously extended to the real line $t\in \rr$ by zero extension.  
We consider the time-weighted expressions $e^{\delta \nu^\f13  t}\lf\{\mathscr{C}^{[j]}_k[g_k]),(\Delta_k)^{-1}\mathscr{C}^{[j]}_k[g_k],\pb a j_\pm\rg\}$ and consider their Fourier-Laplace transform in $t$ with Fourier variable $-k\lambda_r\in\rr$ (we extend the function by zero on the negative half real line),
\begin{align*}
   &C^{[j]}_k(\lambda_r,y):=\int_0^{+\infty}\mathscr{C}^{[j]}_k[g_{k}](t,y)e^{-it(-k\lambda_r)+\delta\nu^{\f13}  t}dt,\\
  &  D^{[j]}_k(\lambda_r,y):=\int_0^{+\infty} \Delta_k ^{-1}\mathscr{C}^{[j]}_k[g_k](t,y)e^{-it(-k\lambda_r)+\delta\nu^{\f13} t}dt,\  \pb\gamma j_{\pm}(\lambda_r):=\int_0^{+\infty}\pb aj_{\pm}(t)e^{-it(-k\lambda_r)+\delta\nu^{\f13}  t}dt.
\end{align*}
We highlight that due to the Plancherel equality for Fourier transform and Minkowski integral inequality, one has the following relations for $ p\in [1,\infty]$:
\begin{equation}\begin{aligned}
&|k|^\f12\|\pb Cj_k\|_{L^p_y L_{\lambda_r}^2}=\|e^{\delta\nu^\f13 t}\pb{\mathscr{C}}j_k[g_k]\|_{L^p_y L_t^2},&& \|e^{\delta\nu^\f13 t}\pb{\mathscr{C}}j_k[g_k]\|_{L_t^\infty L^p_y }\leq C|k|\|\pb Cj_k\|_{L_{\lambda_r}^1 L^p_y }, \\
&|k|^\f12\|\pb Dj_k\|_{L^p_y L_{\lambda_r}^2}=\|e^{\delta\nu^\f13 t}\de_k^{-1}\pb{\mathscr{C}}j_k[g_k]\|_{L^p_y L_{t}^2},&&\|e^{\delta\nu^\f13 t}\de_k^{-1}\pb{\mathscr{C}}j_k[g_k]\|_{L_t^\infty L^p_y }\leq C|k|\|\pb Dj_k\|_{L_{\lambda_r}^1 L^p_y },\\
&|k|^\f12\|\pb \gamma j_\pm\|_{L_{\lambda_r}^2}=\|e^{\delta\nu^\f13 t}\pb{a}j_\pm\|_{L_{t}^2},&&  \|e^{\delta\nu^\f13 t}\pb{a}j_\pm\|_{L_{t}^\infty}\leq C|k|\|\pb \gamma j_\pm\|_{L_{\lambda_r}^1}. \end{aligned}\label{Lp_subisometry}
\end{equation}
We further note that when $p=2$, the order of $L_y^2$ and $L_t^2/L_{\lambda_r}^2$ in the first column can be swapped. Recall the relation about the trace operators \eqref{eq-relation}, the transformed variables satisfy the following equation
\begin{align*}
    &\lf(-ik\lambda_r-\delta\nu^{\f13}+ik\mathcal{U}^{[j]}\rg)C^{[j]}_k-\nu(\pa_y^2-k^2)C^{[j]}_k=0,\quad \pb Dj_k=\Delta_k^{-1}\pb Cj_k,\\
    &\mathbb {T}_+[C^{[j]}_k]=\al C^{[j]}_k(1)+\pa_y D _k^{[j]}(1)=\pb\gamma j_+ ,\quad\mathbb T_-[C_k^{[j]}]=\al C^{[j]}_k(-1)-\pa_y D_k^{[j]}(-1)=\pb \gamma j_- .
\end{align*}

To compactify the notation, we denote $\lambda_i=-  \delta\nu^{\f13}|k|^{-1}$ and $\lambda:=\lambda_r+i\lambda_i$ ($k\geq 0$).
By adapting the boundary conditions, we obtain
\begin{align*}
&\pb Cj_k=(1+\al)^{-1}\gamma_{+}^{[j]}w_++(1+\al)^{-1}\gamma_{-}^{[j]}w_-,
\end{align*}
where $w_{\pm}$ are defined in \eqref{phi-OS} with $k\lambda_i= -\delta\nu^{\f13}$. 
We fix $\delta$ such that the estimates for the $w_\pm$ in Corollary \ref{w1w2} are satisfied and recall \eqref{defn_epsilon_0}. Direct calculations yields that
\begin{equation}\label{eq-a-12}
\begin{aligned}
  & \|e^{\ep_{0}\nu^{\f13}(t-t_j)}e^{\ep_0\nu^{\f13}t}\pb aj_+ (t)\|_{L^2(\mathcal{J}_{[j]}\cup\mathcal{I}_{[j]})}\\
   &\lesssim \al\|e^{\ep_0\nu^{\f13}(t-t_j)}e^{\ep_0\nu^{\f13}t} g_k(1)\|_{L_t^2(\mathcal{J}_{[j]}\cup\mathcal{I}_{[j]})}+\|e^{\ep_{0}\nu^{\f13}(t-t_j)}e^{\ep_0\nu^{\f13}t}\pa_y\Delta_k^{-1}g_k(1)\|_{L_t^2(\mathcal{J}_{[j]}\cup\mathcal{I}_{[j]})}.
\end{aligned}
\end{equation}
We recall the definition of $t_j$ \eqref{t_j}, and define 
\begin{align*}
 \pb{\wt{a}}j_\pm (t):=e^{\ep_{0}\nu^{\f13}(t-t_j)}e^{\ep_0\nu^{\f13}t}\pb aj_\pm (t)=e^{-\ep_0j-\ep_0j\nu^{\f13}}e^{\delta\nu^\f13 t}\pb aj_\pm (t).
 \end{align*}
 We make the observation that $\pb{\wt{a}}j_\pm (t)\in L^2(\mathbb R)$ and 
\begin{align*}\|\pb{\wt{a}}j_\pm& (t)\|_{L^2(\mathbb R)}^2=\|\pb{\wt{a}}j_\pm (t)\|_{L^2(\mathcal{J}_{[j]}\cup\mathcal{I}_{[j]})}^2=\|e^{-\ep_0j-\ep_0j\nu^{\f13}}e^{\delta\nu^\f13 t}\pb aj_\pm (t)\|_{L^2(\mathcal{J}_{[j]}\cup\mathcal{I}_{[j]})}^2\\
\leq& Ce^{-2\ep_0j-2\ep_0j\nu^{\f13}}\lf(\al^2\|e^{\delta\nu^\f13 t}g_k(t,y=\pm1)\|_{L_t^2(\mathcal{J}_{[j]}\cup\mathcal{I}_{[j]})}^2+\|e^{\delta\nu^\f13 t}\pa_y \de_k^{-1}g_k(t,y=\pm1)\|_{L_t^2(\mathcal{J}_{[j]}\cup\mathcal{I}_{[j]})}^2\rg). 
\end{align*}
Here, the last line is due to \eqref{defn-apm}. Hence, by \eqref{Lp_subisometry}, 
\begin{equation}\label{estimate-L2}
\begin{aligned}
   & \|\pb\gamma j_+ \|_{L^2_{\lambda_r}(\mathbb R)}^2+\|\pb\gamma j_- \|_{L^2_{\lambda_r}(\mathbb R)}^2
   \approx e^{2\ep_0j+2\ep_0j\nu^\f13 }|k|^{-1}\|\pb{\wt{a}}j_+ \|_{L_t^2(\mathbb R)}^2+e^{2\ep_0j+2\ep_0j\nu^\f13 }|k|^{-1}\|\pb{\wt{a}}j_- \|_{L_t^2(\mathbb R)}^2\\
    &\lesssim \al^2|k|^{-1}\|e^{\delta\nu^{\f13}t} g_k(1)\|_{L_t^2(\mathcal{J}_{[j]}\cup\mathcal{I}_{[j]})}^{2}+ \al^2|k|^{-1}\|e^{\delta\nu^{\f13}t} g_k(-1)\|_{L_t^2(\mathcal{J}_{[j]}\cup\mathcal{I}_{[j]})}^{2}\\
    &\quad+ |k|^{-1}\|e^{\delta\nu^{\f13}t}\pa_y(\Delta_k)^{-1}g_k\|_{L_t^2(\mathcal{J}_{[j]}\cup\mathcal{I}_{[j]};L_y^{\infty})}^{2}.
    \end{aligned}
\end{equation} 
Since $e^{\delta \nu^\f13 t}\approx e^{\ep_0\nu^\f13 t} e^{\ep_0 j+\ep_0j\nu^\f13 }$ on the interval $\mathcal{J}_{[j]}\cup \mathcal{I}_{[j]}$, we obtain the following
\begin{align}\label{gamma_j_est}\begin{split}
&e^{-2\ep_0j {-2\ep_0j\nu^\f13 }}\Big( \|\pb\gamma j_+ \|_{L^2_{\lambda_r}(\mathbb R)}^2+\|\pb\gamma j_- \|_{L^2_{\lambda_r}(\mathbb R)}^2\Big)\lesssim  \al^2|k|^{-1}\|e^{ \ep_0 \nu^{\f13}t} g_k(1)\|_{L_t^2(\mathcal{J}_{[j]}\cup\mathcal{I}_{[j]})}^2\\
&\qquad+ \al^2|k|^{-1}\|e^{\ep_0\nu^{\f13}t} g_k(-1)\|_{L_t^2(\mathcal{J}_{[j]}\cup\mathcal{I}_{[j]})}^2+ |k|^{-1}\|e^{\ep_0\nu^{\f13}t}\pa_y(\Delta_k)^{-1}g_k\|_{L_t^2(\mathcal{J}_{[j]}\cup\mathcal{I}_{[j]};L_y^{\infty})}^2.
    \end{split}
\end{align} 
Combining the relations \eqref{Lp_subisometry} with  the estimates of $w_{\pm}$ in Corollary \ref{w1w2} yields the bound
\begin{equation*}
\begin{aligned}
    &\|e^{\ep_0\nu^{\f13}(t-t_j)}e^{\ep_{0}\nu^{\f13}t}\mathscr{C}^{[j]}_k[g_k]\|_{L^2_t([t_j-1,T_*);L^2_y)}^2\lesssim e^{-2\ep_0j-2\ep_0j\nu^\f13}|k|\|C^{[j]}_k\|_{L^2_{\lambda_r}L^2_y}^2 \\
    &\lesssim e^{-2\ep_0j-2\ep_0j\nu^\f13}|k|\|\gamma_+^{[j]} w_+\|_{L_{\lambda_r}^2L^2}^2+e^{-2\ep_0j-2\ep_0j\nu^\f13}|k|\|\gamma_-^{[j]} w_-\|_{L^2_{\lambda_r}L^2_y}^2\\
    &\lesssim e^{-2\ep_0j-2\ep_0j\nu^\f13}\nu^{\frac{1}{3}}|k|^{-\f13}|k|(\|\gamma_+^{[j]} \|_{L_{\lambda_r}^2}^2+\|\gamma_-^{[j]} \|_{L_{\lambda_r}^2}^2)\\
    &\lesssim \nu^{\f13}|k|^{-\f13}\Big(\al^2\|e^{\ep_0\nu^{\f13}t} g_k(1)\|_{L_t^2(\mathcal{J}_{[j]}\cup\mathcal{I}_{[j]})}^2+\al^2\|e^{\ep_0\nu^{\f13}t} g_k(-1)\|_{L_t^2(\mathcal{J}_{[j]}\cup\mathcal{I}_{[j]})}^2\\
    &\qquad\qquad\qquad+\|e^{\ep_0\nu^{\f13}t}\pa_y(\Delta_k)^{-1}g_k\|_{L_t^2(\mathcal{J}_{[j]}\cup\mathcal{I}_{[j]};L_y^{\infty})}^2\Big).
\end{aligned}
\end{equation*}
For any $j_1$ satisfies $0\le j\le j_1\le N$, it holds 
\begin{equation}\label{eq-gk-L2L2}
    \begin{aligned}
        &\|e^{\ep_{0}\nu^{\f13}t}\mathscr{C}^{[j]}_k[g_k]\|_{L^2_t(\mathcal{J}_{[j_1]}\cup\mathcal{I}_{[j_1]};L_y^2)}^2\lesssim e^{-2\ep_0|j_1-j|(1+\nu^{\f13})}\nu^{\f13}|k|^{-\f13}\Big(\al^2\|e^{\ep_0\nu^{\f13}t}g_k(1)\|_{L_t^2(\mathcal{J}_{[j]}\cup\mathcal{I}_{[j]})}^2\\
        &\qquad+\al^2\|e^{\ep_0\nu^{\f13}t}g_k(-1)\|_{L_t^2(\mathcal{J}_{[j]}\cup\mathcal{I}_{[j]})}^2+\|e^{\ep_0\nu^{\f13}t}\pa_y(\Delta_k)^{-1}g_k\|_{L_t^2(\mathcal{J}_{[j]}\cup\mathcal{I}_{[j]};L_y^{\infty})}^2\Big).
    \end{aligned}
\end{equation}
By Corollary \ref{w1w2}, there holds
\begin{align*}
    \|(ik,\pa_y)D^{[j]}_k\|_{L^2}\lesssim |\gamma_+^{[j]} |\|(ik,\pa_y)\phi_+\|_{L^2}+|\gamma_-^{[j]} \|(ik,\pa_y)\phi_-\|_{L^2}\lesssim \nu^{\f12}|k|^{-\f12}(|\gamma_+^{[j]} |+|\gamma_-^{[j]} |).
\end{align*}
Using the relation in \eqref{Lp_subisometry} and \eqref{gamma_j_est}, for $0\le j\le j_1\le N$, we can conclude that
\begin{equation}\label{est-phi-L2L2}
\begin{aligned}
   &e^{2\ep_0\nu^{\f13}(t_{j_1}-t_j)}\|e^{\ep_{0}\nu^{\f13}t}(ik,\pa_y)(\Delta_k)^{-1}\mathscr{C}_k^{[j]}[g_k]\|_{L^2_t(\mathcal{J}_{j_1}\cup\mathcal{I}_{[j_1]};L_y^2)}^2\\
   &\le  e^{-2\ep_0\nu^{\f13}t_j}\|e^{\delta\nu^{\f13}t}(ik,\pa_y)(\Delta_k)^{-1}\mathscr{C}_k^{[j]}[g_k]\|_{L^2_t([t_j-1,T_*);L_y^2)}^2\\
   &\approx e^{-2\ep_0\nu^{\f13}t_j}|k|\|(ik,\pa_y)D_k^{[j]}\|_{L^2_{\lambda_r}L^2}^2\lesssim  e^{-2\ep_0\nu^{\f13}t_j}|k|\nu|k|^{-1}(\|\gamma_+^{[j]} \|_{L^2_{\lambda_r}}^2+\|\gamma_-^{[j]} \|_{L^2_{\lambda_r}}^2)\\
   &\lesssim  \nu|k|^{-1}\big(\al^2\|e^{\ep_0\nu^{\f13}t} g_k(1)\|_{L_t^2(\mathcal{J}_{[j]}\cup\mathcal{I}_{[j]})}^2+\al^2\|e^{\ep_0\nu^{\f13}t} g_k(-1)\|_{L_t^2(\mathcal{J}_{[j]}\cup\mathcal{I}_{[j]})}^2\\
   &\qquad\qquad+\|e^{\ep_0\nu^{\f13}t}\pa_y(\Delta_k)^{-1}g_k\|_{L_t^2(\mathcal{J}_{[j]}\cup\mathcal{I}_{[j]};L_y^{\infty})}^2\big).
\end{aligned}
\end{equation}
Next, we prove the weighted $L_t^{\infty}L_y^2$-estimate. For $0\le j\le j_1\le N$, combining the relation \eqref{Lp_subisometry}  and the estimates of $w_{\pm}$ in Corollary \ref{w1w2} yields that 
\begin{equation}\label{est-infty-weight-1}
\begin{aligned}
    &e^{2\ep_0\nu^{\f13}(t_{j_1}-t_j)}\|e^{\ep_{0}\nu^{\f13}t}(1-|y|)\mathscr{C}^{[j]}_k[g_k](t,y)\|_{L^{\infty}_t(\mathcal{J}_{[j_1]}\cup\mathcal{I}_{[j_1]};L^2_y)}^2\\
    &\le e^{-2\ep_0\nu^{\f13}t_j}\|e^{\delta\nu^{\f13}t}(1-|y|)\mathscr{C}^{[j]}_k[g_k](t,y)\|_{L^{\infty}_t([t_j-1,T_*);L^2_y)}^2\\
    &\lesssim e^{-2\ep_0\nu^{\f13}t_j}|k|\|\gamma_+^{[j]} (1-|y|)w_+\|_{L_{\lambda_r}^1L_y^2}^2+e^{-2\ep_0\nu^{\f13}t_j}|k|\|(1-|y|)\gamma_-^{[j]} w_-\|_{L_{\lambda_r}^1L_y^2}^2\\
    &\lesssim e^{-2\ep_0j-2\ep_0\nu^{\f13}j}|k|\nu |k|^{-1}\big(\|(1+|Z_+
    |)^{-\f34}\gamma_+^{[j]} \|_{L^1_{\lambda_r}}^2+\|(1+|Z_-
    |^{-\f34})\gamma_-^{[j]} \|_{L_{\lambda_r}^1}^2\big).
\end{aligned}
\end{equation}
Here, we introduce the definition \eqref{eq:def-zpm} that $-Z_+=L_+(1+\overline{d}_+)=\nu^{-\f13}|k|^{\f13}(\mathcal{V}'(1))^{-\f23}(\mathcal{V}(1)-\lambda_r-i\lambda_i-i\nu k)$. As a consequence, 
\begin{align*}\|(1+|Z_+|)^{-\f34}\|_{L^2_{\lambda_r}}\lesssim \lf(\int_{\rr} (1+|Z_+|^2)^{-\f34}d\lambda_r \rg)^\f12\lesssim\lf(\int_{\rr} (1+|L_+|^2|\mathcal{V}(1)-\lambda_r|^2)^{-\f34}d\lambda_r \rg)^\f12 \le  CL_+^{-\f12}.
\end{align*} 
A similar argument yields the bound for the $Z_-$-case. An application of H\"older's inequality, together with the bound \eqref{gamma_j_est}, yields that
\begin{equation}\label{est-infty-weight}
\begin{aligned}
  \eqref{est-infty-weight-1}_{\rm R.H.S.} 
    &\lesssim \sum_{\sigma\in\{\pm\}}\nu \|(1+|Z_\sigma
    |)^{-\f34}\|_{L^2_{\lambda_r}}^2e^{ {-2\ep_0j-2\ep_0\nu^{\f13}j}}\|\gamma_\sigma^{[j]}\|_{L_{\lambda_r}^2}^2\\
    &\lesssim \nu|k|^{-1} L^{-1}\Big(\al^2\|e^{ \ep_0 \nu^{\f13}t} g_k(1)\|_{L_t^2(\mathcal{J}_{[j]}\cup\mathcal{I}_{[j]})}^2+ \al^2\|e^{\ep_0\nu^{\f13}t} g_k(-1)\|_{L_t^2(\mathcal{J}_{[j]}\cup\mathcal{I}_{[j]})}^2\\
    &\hspace{3cm}+ \|e^{\ep_0\nu^{\f13}t}\pa_y(\Delta_k)^{-1}g_k\|_{L_t^2(\mathcal{J}_{[j]}\cup\mathcal{I}_{[j]};L_y^{\infty})}^2\Big).
\end{aligned}
\end{equation}
This concludes the step.

\step{2. Estimates of $\mathscr{C}_k[g_k]$}
Since on each time interval $\mathcal{I}_{[j]}\backslash\mathcal{J}_{[j+1]}$, $(0\le j\le N)$, $\mathscr{C}_k[g_k]=\sum_{j'=0}^{j}\mathscr{C}^{[j']}_k[g_k]$, and when $t\in\mathcal{J}_{[j+1]}$, $(0\le j\le N)$, $\mathscr{C}_k[g_k]=\sum_{j'=0}^{j+1}\mathscr{C}_k^{[j']}[g_k]$, then the Cauchy-Schwarz inequality and \eqref{eq-gk-L2L2} give that
\begin{align*}
    &\|e^{\ep_0\nu^{\f13}t}\mathscr{C}_k[g_k]\|_{L_t^2([0,T_*);L_y^2)}^2\\
    &\lesssim \sum_{j=0}^N\int_{\mathcal{I}_{[j]}\backslash\mathcal{J}_{[j+1]}}\Big(\sum_{j'=0}^j\Big\|e^{\ep_0\nu^{\f13}t}\mathscr{C}_k^{[j']}[g_k]\Big\|_{L_y^2}\Big)^2dt+\sum_{j=0}^N\int_{\mathcal{J}_{[j+1]}}\Big(\sum_{j'=0}^{j+1}\Big\|e^{\ep_0\nu^{\f13}t}\mathscr{C}_k^{[j']}[g_k]\Big\|_{L_y^2}\Big)^2dt\\
    &\lesssim \sum_{j=0}^N\sum_{j'=0}^j(j-j'+1)^2\|e^{\ep_0\nu^{\f13}t}\mathscr{C}_k^{[j ']}[g_k]\|_{L^2_t(\mathcal{I}_{[j]\backslash\mathcal{J}_{[j+1]}};L_y^2)}^2
    \\
    &\quad+ \sum_{j=0}^N\sum_{j'=0}^{j+1}(j-j'+1)^2\|e^{\ep_0\nu^{\f13}t}\mathscr{C}_k^{[j ']}[g_k]\|_{L^2_t(\mathcal{J}_{[j+1]};L_y^2)}^2\\
    &\lesssim \sum_{j=0}^N\sum_{j'=0}^{j+1}(j-j'+1)^2e^{-2\ep_0|j-j'|}\nu^{\f13}|k|^{-\f13}\Big(\al^2\|e^{\ep_0\nu^{\f13}t} g_k(1)\|_{L_t^2(\mathcal{I}_{[j']}\cup\mathcal{J}_{[j']})}^2\\
    &\qquad\qquad+\al^2\|e^{\ep_0\nu^{\f13}t} g_k(-1)\|_{L_t^2(\mathcal{I}_{[j']}\cup\mathcal{J}_{[j']})}^2+\|e^{\ep_0\nu^{\f13}t}\pa_y(\Delta_k)^{-1}g_k\|_{L_t^2(\mathcal{I}_{[j']}\cup\mathcal{J}_{[j']};L_y^{\infty})}^2\Big)\\
     &\lesssim \nu^{\f13}|k|^{-\f13}\Big(\al^2\|e^{\ep_0\nu^{\f13}t} g_k(1)\|_{L_t^2[0,T_*)}^2+\al^2\|e^{\ep_0\nu^{\f13}t}g_k(-1)\|_{L_t^2[0,T_*)}^2+\|e^{\ep_0\nu^{\f13}t}\pa_y(\Delta_k)^{-1}g_k\|_{L_t^2([0,T_*);L_y^{\infty})}^2\Big).
\end{align*}
Here, we use the disjointness of $\mathcal{J}_{[j']}$ and $\mathcal{I}_{[j']}$, moreover, for $1\le j'\le N$, there holds
\begin{align*}
    \|f\|_{L^2(\mathcal{I}_{[j']}\cup\mathcal{J}_{[j']})}^2\le \|f\|_{L^2(\mathcal{J}_{[j']})}^2+\|f\|_{L^2(\mathcal{I}_{[j']})}^2\le \|f\|_{L^2(\mathcal{I}_{[j']})}^2+\|f\|_{L^2(\mathcal{I}_{[j'-1]})}^2.
\end{align*}
Next, through a similar argument, we can derive from \eqref{est-phi-L2L2} and \eqref{est-infty-weight} that
\begin{align*}
    &\|e^{\ep_0\nu^{\f13}t}(ik,\pa_y)(\Delta_k)^{-1}\mathscr{C}_k[g_k]\|_{L_t^2([0,T_*);L_y^2)}^2+\nu^{-\f13}|k|^{\f13}\|e^{\ep_0\nu^{\f13}t}(1-|y|)\mathscr{C}_k[g_k]\|_{L_t^{\infty}([0,T_*);L_y^2)}^2\\
    &\lesssim \nu|k|^{-1}\Big(\al^2\|e^{\ep_0\nu^{\f13}t} g_k(1)\|_{L_t^2[0,T_*)}^2+\al^2\|e^{\ep_0\nu^{\f13}t} g_k(-1)\|_{L_t^2[0,T_*)}^2+\|e^{\ep_0\nu^{\f13}t}\pa_y(\Delta_k)^{-1}g_k\|_{L_t^2([0,T_*);L_y^{\infty})}^2\Big),
\end{align*}
which finished the proof.
\end{proof}

\subsection{Proof of Lemma \ref{Lem-F-Y}, Corollary \ref{prop-ba}, and Corollary \ref{prop-eb}}
\begin{proof}[Proof of Lemma \ref{Lem-F-Y}]
   Recalling $\mathscr{E}[g_{\ne}]$ defined in \eqref{Frozen_err}, after $x$-Fourier transform, we have
   \begin{align*}
       \mathscr{E}_k^{[j]}[g_k](t,y)=(\mathcal{U}^{[j]}(y)-\mathcal{U}(t,y))\sum_{j'=0}^j\mathscr{C}^{[j']}[g_k](t,y)-\sum_{j'=0}^{j-1}(\mathcal{U}^{[j]}(y)-\mathcal{U}^{[j']}(y))\mathscr C^{[j']}[g_k](t,y).
   \end{align*}
  The definition of $\wt U $ and \eqref{eq-U} show
   \begin{equation*}
       \begin{cases}
           \pa_t(\mathcal{U}-y)-\nu\pa_y^2(\mathcal{U}-y)=0,\\
           \al\pa_y(\mathcal{U}-y)|_{y=\pm1}=\mp(\mathcal{U}(\pm1)\mp1),\quad (\mathcal{U}-y)|_{t=0}=v_{0,\rm in}^1.
       \end{cases}
   \end{equation*}
   Let $t\ge s\ge0$. From the fundamental theorem of calculus it follows that
   \begin{equation}\label{est-U}
   \begin{aligned}
       &\mathcal{U}(t,y)-\mathcal{U}(s,y)=\int_s^t\pa_{\tau}\mathcal{U}(\tau,y)dy=\nu\int_s^t\pa_y^2\mathcal{U}(\tau,y)d\tau, \quad y\in[-1,1],\\
       &|\mathcal{U}(t,y)-\mathcal{U}(s,y)|\lesssim \nu^{\f43}|t-s|.
       \end{aligned}
   \end{equation}
Then, by Cauchy-Schwarz inequality, \eqref{ineq-U}, \eqref{eq-gk-L2L2}, and the fact that $0\le j'\le j\le N$, we have the following estimate
\begin{align*}
    &\|e^{\ep_0\nu^{\f13}t}\mathscr{E}_k[g_k]\|_{L^2_t([0,T_*);L_y^2)}^2=\sum_{j=0}^N\Big\| e^{\ep_0\nu^{\f13}t }\mathscr{E}^{[j]}_k[g_k]\Big\|_{L_t^2(\mathcal{I}_{[j]};L^2_y)}^2\\
     &  \lesssim \nu^2\sum_{j=0}^N\sum_{j'=0}^{j}|j-j'+1|^4\Big\|e^{\ep_0\nu^{\f13}t}\mathscr{C}^{[j']}[g_k]\Big\|_{L_t^2([\mathcal{I}_{[j]};L^2_y)}^2\\
    &   \lesssim    \nu^2\sum_{j=0}^N\sum_{j'=0}^{j}|j-j'+1|^4e^{-2\ep_0|j-j'|}\nu^{\f13}|k|^{-\f13}\Big(\al^2\|e^{\ep_0\nu^{\f13}t}g_k(1)\|_{L_t^2(\mathcal{I}_{[j]}\cup\mathcal{J}_{[j]})}^2\\
    &\qquad\qquad+\al^2\|e^{\ep_0\nu^{\f13}t} g_k(-1)\|_{L_t^2(\mathcal{I}_{[j]}\cup\mathcal{J}_{[j]})}^2+\|e^{\ep_0\nu^{\f13}t}\pa_y(\Delta_k)^{-1}g_k\|_{L_t^2(\mathcal{I}_{[j]}\cup\mathcal{J}_{[j]};L_y^{\infty})}^2\Big)\\
    &   \lesssim    \nu^{\f73}|k|^{-\f13}\Big(\al^2\|e^{\ep_0\nu^{\f13}t} g_k(1)\|_{L_t^2[0,T_*)}^2+\al^2\|e^{\ep_0\nu^{\f13}t} g_k(-1)\|_{L_t^2[0,T_*)}^2+\|e^{\ep_0\nu^{\f13}t}\pa_y(\Delta_k)^{-1}g_k\|_{L_t^2([0,T_*);L_y^{\infty})}^2\Big).
\end{align*}
\end{proof}

\begin{proof}[Proof of Corollary \ref{prop-ba}]
According to our construction \eqref{eq-b1}, we can see that $\w\om{b}_k=\mathscr{C}_k[\w\om{a}_k]$, then by Lemma \ref{lem:estimate-u1_k} and Proposition \ref{prop-C}, it immediately follows
\begin{align*}
   & \|e^{\ep_0\nu^{\f13}t}(ik,\pa_y)\w\p{b}_k\|_{L_t^2([0,T_*),L_y^2)}^2+\nu^{-\f13}|k|^{\f13} \|e^{\ep_0\nu^{\f13}t}(1-|y|)\w\om{b}_k\|_{L_t^{\infty}([0,T_*),L_y^2)}^2\\
   &\qquad+\nu^{\f23}|k|^{-\f23}\|e^{\ep_0\nu^{\f13}t}\w\om{b}_k\|_{L_t^2([0,T_*),L_y^2)}^2\\
   &\lesssim \nu|k|^{-1}\Big(\|e^{\ep_0\nu^{\f13}t}\al\w\om{a}_k\|_{L_t^2([0,T_*);L_y^{\infty})}^2+\|e^{\ep_0\nu^{\f13}t}\pa_y\w\p{a}_k\|_{L^2_t([0,T_*);L_y^{\infty})}^2\Big)\lesssim \nu^{\f23}|k|^{-3}\|\om_{k;\rm in}\|_{H_{k,y}^2}^2.
\end{align*}
\end{proof}

\begin{proof}[Proof of Corollary \ref{prop-eb}] The proof is similar to the proof of Corollary \ref{prop-ba} and we omit it for the sake of brevity. 
\end{proof}

\section{Semi-group Estimates}\label{sec:Linearized System}
In this section, we derive the estimates of the error part $\w\om{\rm e}_k$ \eqref{eq-we}.
We consider the general version of \eqref{eq-we}: 
\begin{align}\label{eq-wL}\begin{cases}
\pa_t\oml_k+ik\mathcal{U}(t,y)\oml_k-\nu\de\oml_k=\mathbb F_{k;1}+\mathbb F_{k;2}+\pa_y\mathbb F_{k;3},\\
         \de_k\w\psi{\rm L}_k=\oml_k,\\
        \oml_k|_{t=T_0}=g_k,\quad \oml_k(t,\pm1)=0,\quad (t,y)\in [T_0,T_*)\times[-1,1],
\end{cases}
\end{align}Here, the forcing $\mathbb{F}_{k;1}, \mathbb{F}_{k;2}, \mathbb{F}_{k;3} $ satisfy appropriate regularity constraints specified in the theorems below. We observe that the equation \eqref{eq-we} is the special case of \eqref{eq-wL} with the right-hand side being chosen as $ik\mathscr{E}[\w\om{\rm e}_k]+\w{\mathbb F}{\rm e}_k$. 
Now, we apply the singular integral operators to obtain the optimal inviscid damping estimate and the enhanced dissipation estimate.
\begin{proposition}\label{prop-e-E}
Consider the equation \eqref{eq-wL} and assume that $\mathbb F_{k;1}\in L_t^1L_{y}^2,\,\ \mathbb F_{k;2}\in L_t^2L_{y}^2,\  \mathbb F_{k;3}\in L_t^2 H^{1}_{y}$. There is a constant $c_0>0$ independent of $\nu$ and $k$, such that for all $0<c<c_0$, the following estimate holds
    \begin{align*}
          &\|e^{c\nu^{\f13}t}\oml_k\|_{L^{\infty}_tL_y^2}^2+k^2\|e^{c\nu^{\f13}t}(k,\pa_y)\w\p{\rm L}_k\|_{L_t^2L_y^2}^2+\nu^{\f13}|k|^{\f23}\|e^{c\nu^{\f13}t}\oml_k\|_{L_t^{2}L_y^2}^2+\nu\|e^{c\nu^{\f13}t}\na_k\oml_k \|_{L_t^2L_y^2}^2\\
        &\lesssim  \|g_k\|_{L^2_y}^2+\|e^{c\nu^{\f13}t}\mathbb F_{k;1}\|_{L_t^1L_y^2}^2+\min\{\nu^{-\f13}|k|^{-\f23},\nu^{-1}|k|^{-2}\}\|e^{c\nu^{\f13}t}\mathbb F_{k;2}\|_{L_t^2L_y^2}^2+\nu^{-1}\|e^{c\nu^{\f13}t}\mathbb F_{k;3}\|_{L_t^2L_y^2}^2.
    \end{align*}
\end{proposition}
As a corollary of this proposition, we obtain that 
\begin{corol} \label{cor-e-E}Consider the equation \eqref{eq-wL} and assume that $\mathbb F_{1}\in L_t^1L_{x,y}^2,\,\ \mathbb F_{2}\in L_t^2L_{x,y}^2,\  \mathbb F_{3}\in L_t^2 H^{1}_{x,y}$. 
There is a constant $c_0>0$ independent of $\nu$ and $k$, such that for all $0<c<c_0$, the following estimate holds
    \begin{align*}
          &\|e^{c\nu^{\f13}t}\oml_\nq\|_{{L_t^{\infty}L_{x,y}^2}}^2+\|e^{c\nu^{\f13}t}\pa_x\na\w\p{\rm L}_\nq\|_{L_t^2L_{x,y}^2}^2+\nu^{\f13}\|e^{c\nu^{\f13}t}|\pa_x|^\f13\oml_\nq\|_{L_t^{2}L_{x,y}^2}^2+\nu\|e^{c\nu^{\f13}t}\na\oml_\nq \|_{L_t^2L_{x,y}^2}^2\\
        &\lesssim \|g_\nq\|_{L_{x,y}^2}^2+\|e^{c\nu^{\f13}t}\mathbb F_{\nq;1}\|_{L_t^1L_{x,y}^2}^2+\min\{\|e^{c\nu^{\f13}t}\nu^{-\f16}|\pa_x|^{-\f13}\mathbb F_{\nq;2}\|_{L^2L^2}^2,\|e^{c\nu^{\f13}t}\nu^{-\f12}|\pa_x|^{-1}\mathbb F_{\nq;2}\|_{L^2L^2}^2\}\\
        &\quad+\nu^{-1}\|e^{c\nu^{\f13}t}\mathbb F_{\nq;3}\|_{L_t^2L_{x,y}^2}^2.
    \end{align*} 
    Here, $g_\nq=\sum_{k\in\mathbb Z\backslash\{0\}}g_k e^{ikx},\, \mathbb{F}_{\nq;j}=\sum_{k\in\mathbb Z\backslash\{ 0\}}\mathbb{F}_{k;j} e^{ikx}.$
\end{corol}
\begin{remark}
    Here, by introducing the SIO-induced norm and applying energy-type estimates, we obtained both inviscid damping and enhanced dissipation while reducing the regularity requirement on the initial data to the $L^2$-level. Whereas previous arguments require higher regularity \cite{ChenLiWeiZhang18, bedrossian2025stability}.  
\end{remark}
\begin{proof}The corollary follows directly from the differential inequality \eqref{prop-e-E_pf2} and Gr\"onwall's inequality. 
\end{proof}
\begin{proof}[Proof of Proposition \ref{prop-e-E}]
We decompose the proof into steps. 

\step{1: Energy setup.} 
For any $k\ne0$, we consider the following energy functional
\begin{equation}\label{eq-E}
     \mathbb E_k(t)=\|\oml_k\|_{L^2}^2+c_{\al}{\rm Re}\langle \oml_k,\mathfrak{J}_k[\oml_k]\rangle+c_{\b}{\rm Re}\langle \oml_k,\mathfrak{J}_k^{\rm (e)}[\oml_k]\rangle,\quad \forall t\in [ T_0,+\infty).
\end{equation}
Here, the operators $\mathfrak{J}_k,\, \mathfrak{J}_k^{\rm(e)}$ are defined in Definition \ref{defn:SIO} and the positive parameters $c_\al,\ c_\beta$ will be chosen during the proof in \eqref{chc_cal_beta}. The time evolution of $\mathbb{E}_k(t)$ involves three components
\begin{align}
\frac{d}{dt} \mathbb E_k(t)=\frac{d}{dt}\|\oml_k\|_{L^2}^2+c_{\al}\frac{d}{dt}{\rm Re}\langle \oml_k,\mathfrak{J}_k[\oml_k]\rangle+c_{\b}\frac{d}{dt}{\rm Re}\langle \oml_k,\mathfrak{J}_k^{\rm (e)}[\oml_k]\rangle.\label{prop-e-E_pf1}
\end{align}
We will estimate each term in the sequel. 

\step{2: Estimate of the $\frac{d}{dt}\|\oml_k\|^2_2$.}
According to \eqref{eq-wL}, since $\oml_k(t,\pm1)=0$, through integration by parts, we get
\begin{equation}\label{est-L2}
    \begin{aligned}
        \f{d}{dt}\|\oml_k\|^2_2&=2{\rm Re}\langle \oml_k,\pa_t\oml_k\rangle=2{\rm Re}\big\langle \oml_k,\mathbb F_{k;1}+\mathbb F_{k;2}+\pa_y\mathbb F_{k;3}+\nu\D_k\oml_k-ik\mathcal{U}\oml_k\big\rangle\\
        &\le -\f32\nu\|\na_k\oml_k\|_{L^2}^2+2\|\oml_k\|_{L^2}\big(\|\mathbb F_{k;1}\|_{L^2}+\|\mathbb F_{k;2}\|_{L^2}\big)+2\nu^{-1}\|\mathbb F_{k;3}\|_{L^2}^2.
    \end{aligned}
\end{equation}
 
\step{3: Estimate of the $\frac{d}{dt}{\rm Re}\langle \oml_k,\mathfrak{J}_k[\oml_k]\rangle$.}
According to the fact that $\pa_t\mathcal{U}(t,y)-\nu \pa_y^2\mathcal{U}=0$, and the symmetry of $\mathfrak{J}_k$ \eqref{eq-symm}, we get
\begin{align*}
    &\frac{d}{dt}{\rm Re}\langle\oml_k,\mathfrak{J}_k[\oml_k]\rangle={\rm Re}\big\langle\pa_t\oml_k,\mathfrak{J}_k[\oml_k]\big\rangle+{\rm Re}\big\langle \oml_k,\pa_t\big(\mathfrak{J}_k[\oml_k]\big)\big\rangle\\
    &=2{\rm Re}\Big\langle \oml_k,\mathfrak{J}_k\big[\pa_t\oml_k\big]\Big\rangle+{\rm Re}\big\langle \oml_k,[\pa_t,\mathfrak{J}_k][\oml_k]\big\rangle\\
     &=2{\rm Re}\big\langle \oml_k,\mathfrak{J}_k[-ik\mathcal{U}\oml_k]\big\rangle+2\nu{\rm Re}\big\langle \mathfrak{J}_k[\D_k\oml_k],\oml_k\big\rangle+2{\rm Re}\big\langle \mathbb F_{k;1}+\mathbb F_{k;2}+\pa_y\mathbb F_{k;3}
     ,\mathfrak{J}_k[\oml_k]\big\rangle\\
     &\quad+{\rm Re}\big\langle \oml_k,[\pa_t,\mathfrak{J}_k][\oml_k]\big\rangle=:J_1+J_2+J_3+J_4.
\end{align*}
Next, we estimate $J_j$ $(j=1,2,3,4)$ term by term. Taking $f_k=\oml_k$ in Lemma \ref{Lem-SIO}, it yields
\begin{align*}
    J_1&=2{\rm Re}\big\langle \oml_k,\mathfrak{J}_k[-ik\mathcal{U}\oml_k]\big\rangle=-\f{k^2}2\|\na_k\w\p{\rm L}_k\|_{L^2}^2.
    \end{align*}
Next, since $\oml_k(\pm1)=0$, using integration by parts, Corollary \ref{corol-bound} and Lemma \ref{comm-bound}, we obtain   
    \begin{align*}
    J_2
    &=2\nu{\rm Re}\int_{-1}^1[\mathfrak{J}_k,\pa_y][\pa_y\oml_k]\overline{\oml_k}dy-2\nu{\rm Re}\int_{-1}^1\mathfrak{J}_k[\na_k\oml_k]\overline{\na_k\oml_k}dy\\
    &\le C\nu\big(|k|^{-1}\|[\mathfrak{J}_k,\pa_y]\|_{L^2\to L^2}+\|\mathfrak{J}_k\|_{L^2\to L^2}\big)\|\na_k\oml_k\|_{L^2}^2\le \nu C\|\na_k\oml_k\|_{L^2}^2.
    \end{align*}
   We can see from the definition of $G_k(y,y')$ \eqref{eq-Green} that $G_k(y=\pm1,y')\equiv 0$ for all $y'\in[-1,1]$, thus $\mathfrak{J}_k[\oml_k](y=\pm1)=0$. Through integration by parts, together with Corollary \ref{corol-bound} and Lemma \ref{comm-bound}, we get
    \begin{align*}
    J_3
    &=2{\rm Re}\langle \mathbb F_{k;1}+\mathbb F_{k;2},\mathfrak{J}_k[\oml_k]\rangle-2{\rm Re}\big\langle\mathbb F_{k;3},[\pa_y,\mathfrak{J}_k][\oml_k]+\mathfrak{J}_k[\pa_y\oml_k]\big\rangle\\
    &\le C \big(\|\mathbb F_{k;1}\|_{L^2}+\|\mathbb F_{k;2}\|_{L^2}\big)\|\mathfrak{J}_k\|_{L^2\to L^2}\|\oml_k\|_{L^2}\\
    &\qquad+ C\|\mathbb F_{k;3}\|_{L^2}\big(\|[\pa_y,\mathfrak{J}_k]\|_{L^2\to L^2}\|\oml_k\|_{L^2}+\|\mathfrak{J}_k\|_{L^2\to L^2}\|\pa_y\oml_k\|_{L^2}\big)\\
    &\le C \big(\|\mathbb F_{k;1}\|_{L^2}+\|\mathbb F_{k;2}\|_{L^2}\big)\|\oml_k\|_{L^2}+C\nu^{-1}\|\mathbb F_{k;3}\|_{L^2}^2+\nu\|\na_k\oml_k\|_{L^2}^2.
    \end{align*}
    Next, for $J_4$, Lemma \ref{lem-s-t} implies
    \begin{align*}
    J_4&= {\rm Re}\big\langle \oml_k,[\pa_t,\mathfrak{J}_k][\oml_k]\big\rangle\le C\|[\pa_t,\mathfrak{J}_k]\|_{L^2\to L^2}\|\oml_k\|_{L^2}^2\le C\nu^{\f43}\|\oml_k\|_{L^2}^2.
    \end{align*} 
Now, we can conclude that
\begin{equation}\label{de-t-Jw}
    \begin{aligned}
        &\frac{d}{dt}{\rm Re}\big\langle\oml_k,\mathfrak{J}_k[\oml_k]\big\rangle\le -\f{k^2}{2}\|\na_k\w\p{\rm L}_k\|_{L^2}^2+\nu C\|\na_k\oml_k\|_{L^2}^2\\
        &\qquad  +C \|\mathbb F_{k;1}\|_{L^2}\|\oml_k\|_{L^2}+C\|\mathbb F_{k;2}\|_{L^2}\|\oml_k\|_{L^2}+\nu^{-1}C\|\mathbb F_{k;3}\|_{L^2}^2+C\nu^{\f43}\|\oml_k\|_{L^2}^2.
    \end{aligned}
\end{equation}

\step{4: Estimate of the $ \frac{d}{dt}{\rm Re}\langle \oml_k,\mathfrak{J}_k^{\rm (e)}[\oml]\rangle$.} 
using the symmetry property of $\mathfrak{J}_k^{(\e)}$ in Remark \ref{eq-symm}, we decompose the term as follows,
\begin{align*}
   & \frac{d}{dt}{\rm Re}\langle \oml_k,\mathfrak{J}_k^{\rm (e)}[\oml_k]\rangle
    =2{\rm Re}\big\langle \oml_k,\mathfrak{J}_k^{(\e)}[\pa_t\oml_k]\big\rangle+{\rm Re}\big\langle \oml_k,[\pa_t,\mathfrak{J}_k^{(\e)}][\oml_k]\big\rangle\\
    &=2{\rm Re}\big\langle \oml_k,\mathfrak{J}_k^{(\e)}[-ik\mathcal{U}\oml_k]\big\rangle+2\nu{\rm Re}\big\langle\oml_k, \mathfrak{J}_k^{\rm (e)}[\D_k\oml_k]\big\rangle\\
    &\quad +2{\rm Re}\big\langle \mathbb F_{k;1}+\mathbb F_{k;2}+\pa_y\mathbb F_{k;3},\mathfrak{J}^{\rm (e)}_k[\oml_k]\big\rangle +{\rm Re}\big\langle \oml_k,[\pa_t,\mathfrak{J}_k^{(\e)}][\oml_k]\big\rangle =:T_1+T_2+T_3+T_4.
\end{align*} 
By the construction of $\mathfrak{J}_k^{\rm (e)}$, and taking $f_k=\oml_k$ in Lemma \ref{Lem-SIO}, we obtain
\begin{align*}
    T_1&=2{\rm Re}\big\langle \oml_k,\mathfrak{J}_k^{(\e)}[-ik\mathcal{U}\oml_k]\big\rangle\le -C\nu^{\f13}|k|^{\f23}\|\oml_k\|_{L^2}^2+C\nu\|\na_k\oml_k\|_{L^2}^2.
    \end{align*}
  Since $\oml_k(y)$ and $\mathfrak{J}_{k}^{\rm(e)}[\oml_k](y)$ vanish on the boundary $y=\pm1$, we apply integration by parts, Corollary \ref{corol-bound}, and Lemma \ref{comm-bound} to obtain
    \begin{align*}
    T_2
    &=2\nu{\rm Re}\int_{-1}^1[\mathfrak{J}_k^{\rm (e)},\pa_y][\pa_y\oml_k]\overline{\oml_k}dy-2\nu{\rm Re}\int_{-1}^1\mathfrak{J}_k^{\rm (e)}[\na_k\oml_k]\overline{\na_k\oml_k}dy\\
    &\le \nu \big(|k|^{-1}\|[\mathfrak{J}_k^{\rm (e)},\pa_y]\|_{L^2\to L^2}+\|\mathfrak{J}_k^{\rm (e)}\|_{L^2\to L^2}\big)\|\na_k\oml_k\|_{L^2}^2\le \nu C\|\na_k\oml_k\|_{L^2}^2.
    \end{align*}
 Next, using the fact that $\mathfrak{J}_k^{(\e)}[\oml_k](y=\pm1)=0$, through integration by parts, Corollary \ref{corol-bound}, and Lemma \ref{comm-bound}, it holds
\begin{align*}
    T_3&=2{\rm Re}\langle\mathbb F_{k;1}+\mathbb F_{k;2},\mathfrak{J}_k^{\rm (e)}[\oml_k ]\rangle-2{\rm Re}\big\langle \mathbb F_{k;3},[\pa_y,\mathfrak{J}_k^{\rm (e)}][\oml_k ]+\mathfrak{J}_k^{\rm (e)}[\pa_y\oml_k ]\big\rangle\\
    &\le C \big(\|\mathbb F_{k;1}\|_{L^2}+\mathbb F_{k;2}\|_{L^2}\big)\|\mathfrak{J}_k^{\rm (e)}\|_{L^2\to L^2}\|\oml_k \|_{L^2}\\
    &\qquad+C\|\mathbb F_{k;3}\|_{L^2}\big(\|[\pa_y,\mathfrak{J}_k^{\rm (e)}]\|_{L^2\to L^2}\|\oml_k \|_{L^2}+\|\mathfrak{J}_k^{\rm (e)}\|_{L^2\to L^2}\|\pa_y\oml_k \|_{L^2}\big)\\
    &\le C \|\mathbb F_{k;1}\|_{L^2}\|\oml_k \|_{L^2}+C\|\mathbb F_{k;2}\|_{L^2}\|\oml_k \|_{L^2}+\nu^{-1}C\|\mathbb F_{k;3}\|_{L^2}^2+\nu\|\na_k\oml_k \|_{L^2}^2.
\end{align*}
An application of the commutator $[\pa_t ,\mathfrak J_k]$-estimates in Lemma \ref{lem-s-t} gives
\begin{align*}
    T_4={\rm Re}\big\langle \oml_k,[\pa_t,\mathfrak{J}_k^{(\e)}][\oml_k]\big\rangle \le C\|[\pa_t,\mathfrak{J}_k^{(\e)}]\|_{L^2\to L^2}\|\oml_k\|_{L^2}^2\lesssim C\nu^{\f43}\|\oml_k\|_{L^2}^2.
\end{align*}
Then we can conclude that
\begin{equation}\label{est-in}
\begin{aligned}
    \frac{d}{dt}{\rm Re}\langle \oml_k ,\mathfrak{J}^{\rm (e)}_k[\oml_k ]\rangle&\le -C\nu^{\f13}|k|^{\f23}\|\oml_k\|_{L^2}^2+\nu C\|\na_k\oml_k \|_{L^2}^2+C\nu^{\f43}\|\oml_k\|_{L^2}^2\\
    &\quad +C\|\mathbb F_{k;1}\|_{L^2}\|\oml_k \|_{L^2}+C\|\mathbb F_{k;2}\|_{L^2}\|\oml_k \|_{L^2}+\nu^{-1}C\|\mathbb F_{k;3}\|_{L^2}^2.
    \end{aligned}
    \end{equation}
\step{5: Conclusion.} 
 We conclude from \eqref{est-L2}, \eqref{de-t-Jw}, and \eqref{est-in} that 
 \begin{equation}\label{eq-ddt-E}
\begin{aligned}
        \frac{d}{dt}\mathbb E_k(t)&=\frac{d}{dt}\|\oml_k \|_{L^2}^2+c_{\al}\frac{d}{dt}{\rm Re}\langle \oml_k ,\mathfrak{J}_k[\oml_k ]\rangle+c_{\b}\frac{d}{dt}{\rm Re}\langle \oml_k ,\mathfrak{J}_k^{\rm (e)}[\oml_k ]\rangle\\
        &\le \nu(c_{\al}C+c_{\b}C-\f32)\|\na_k\oml_k \|_{L^2}^2-\f{c_{\al}}{2}k^2\|\na_k\w\p{\rm L}_k\|_{L^2}^2\\
        &\quad  -Cc_{\beta}\nu^{\f13}|k|^{\f23}\|\oml_k\|_{L^2}^2+C(c_{\al}+c_{\b})\nu^{\f43}\|\oml_k\|_{L^2}^2\\
        &\quad+\big(c_{\al}C+c_{\b}C+2\big)\big(\|\mathbb F_{k;1}\|_{L^2}\|\oml_k \|_{L^2}+\|\oml_k \|_{L^2}\|\mathbb F_{k;2}\|_{L^2}+\nu^{-1}\|\mathbb F_{k;3}\|_{L^2}^2\big).
        \end{aligned}
        \end{equation}
  By taking $\nu$, $c_{\al}$, and $c_{\b}$ sufficiently small such that 
    \begin{align}\label{chc_cal_beta}
     & c_{\al}C+c_{\b}C<\f32,\quad
     C(c_{\al}+c_{\beta})\nu^{\f43}\le C c_{\beta}\nu^{\f13}|k|^{\f23},
        \end{align} to obtain 
    \begin{equation}\label{prop-e-E_pf2}
        \begin{aligned}
        & \frac{d}{dt}\mathbb E_k(t)+\nu c\|\na_k\oml_k \|_{L^2}^2+ck^2\|\na_k\w\p{\rm L}_k\|_{L^2}^2+c\nu^{\f13}|k|^{\f23}\|\oml_k\|_{L^2}^2\\
            &\qquad\le C\big(\|\mathbb F_{k;1}\|_{L^2}\|\oml_k \|_{L^2}+\|\mathbb F_{k;2}\|_{L^2}\|\oml_k \|_{L^2}+\nu^{-1}\|\mathbb F_{k;3}\|_{L^2}^2\big),
        \end{aligned}
    \end{equation}
    where $c:=c(c_{\al},c_{\b})>0$ can be any small constant independent of $\nu$. 
    Specifically, there exists a positive constant $c_0>0$, any $c=c(c_{\al},c_{\b})\in (0,c_0)$ makes \eqref{prop-e-E_pf2} hold. We multiply $e^{c\nu^{\f13}t}$ on both sides of \eqref{prop-e-E_pf2} and integrate about $t$, since Corollary \ref{corol-bound} shows
    \begin{align*}
   & \big|{\rm Re}\langle \oml_k ,\mathfrak{J}_k[\oml_k ]\rangle\big|\le C\|\mathfrak{J}_k\|_{L^2\to L^2}\|\oml_k\|_{L^2}^2\le C\|\oml_k\|_{L^2}^2,\\
    &\big|{\rm Re}\langle \oml_k ,\mathfrak{J}_k^{\rm (e)}[\oml_k ]\rangle\big|\le C\|\mathfrak{J}_k^{(\e)}\|_{L^2\to L^2}\|\oml_k\|_{L^2}^2\le C\|\oml_k\|_{L^2}^2.
    \end{align*}
    Taking $c_{\al}$ and $c_{\b}$ small enough, then
    \begin{equation*}
    \begin{aligned}
        &\|e^{c\nu^{\f13}t}\oml_k \|_{L^{\infty}L^2}^2+k^2\|e^{c\nu^{\f13}t}\na_k\w\p{\rm L}\|_{L^2L^2}^2+\nu^{\f13}|k|^{\f23}\|e^{c\nu^{\f13}t}\oml_k \|_{L^{2}L^2}^2+\nu\|e^{c\nu^{\f13}t}\na_k\oml_k \|_{L^2L^2}^2\\
        &\le C\|e^{c\nu^{\f13}t}\mathbb F_{k;1}\|_{L^1L^2}^2+C\min\{\nu^{-\f13}|k|^{-\f23},\nu^{-1}|k|^{-2}\}\|e^{c\nu^{\f13}t}\mathbb F_{k,2}\|_{L^2L^2}^2+C\nu^{-1}\|e^{c\nu^{\f13}t}\mathbb F_{k;3}\|_{L^2L^2}^2.
        \end{aligned}
    \end{equation*}
   Therefore, the proof is finished. 
\end{proof}
\section{Nonlinear Estimates}
\subsection{Organization and Analysis Toolkit}
In this section, we organize the arguments into two parts: the short time regime and the long time regime. In the short time, we prove Proposition \ref{pro:shortT} in Section \ref{sec-shortT}. In the long time, we distinguish between the non-zero modes analysis and the zero mode analysis. For the nonzero mode part, we prove Proposition \ref{lem-re} and Proposition \ref{prop-e} in Section \ref{sec-longT_nz}. Finally, the zero mode estimates in Proposition \ref{prop-wre0} and Proposition \ref{prop-w*0} are presented in Section \ref{sec-longT_z}. 

To smooth out the presentation, we provide two families of simple, while useful, consequences of the bootstrap hypotheses and various boundary corrector/initial data  bounds. The first family of estimates involves $L_t^2L^2_{x,y}$-estimates of the stream function and vorticity, which are consequences of the bootstrap hypotheses\eqref{hyp_h_om_re}-\eqref{hyp_om_*b}, small initial data \eqref{smallness}, and Propositions \ref{ol:estimate-u1}, \ref{pro:linear_boundary_corrector},  \begin{subequations} \label{toolkit_nl}\begin{itemize}\item Short-time Stream Function Bounds $(T_\ast< T_0)$:  
\begin{align}
\label{ineq-size_short} 
       &{\nu^{-\f12}}\big \|e^{\ep_0\nu^{\f13}t}|k|\na_k\w\p{*,b}_k\big\|_{L_t^2([0,T_*);l_k^2L_y^2)}+ {\nu^{-\f{1}{12}}|\ln\nu|}\big \|e^{\ep_0\nu^{\f13}t}|k|\na_k\w\p{b}_k\big\|_{L_t^2([0,T_*);l_k^2L_y^2)}\\
 \nonumber      &+{\big \|e^{\ep_0\nu^{\f13}t}|k|\na_k\w\p{*,i}_k\big\|_{L_t^2([0,T_*);l_k^2L_y^2)}}{+ \nu^{\f14}|\ln\nu|\big \|e^{\ep_0\nu^{\f13}t}|k|\na_k\w\p{a}_k\big\|_{L_t^2([0,T_*);l_k^2L_y^2)}}\\
 &\leq C\ep\nu^{\f{7}{12}}|\ln\nu|;\nonumber\end{align}
\item Long-time Stream Function Bounds $(T_\ast\geq T_0)$:  
\begin{align}
\label{ineq-size} 
       &\nu^{-\f1{3}}\big \|e^{\ep_0\nu^{\f13}t}|k|\na_k\w\p{*,b}_k\big\|_{L_t^2([0,T_*);l_k^2L_y^2)}+ |\ln\nu|\big \|e^{\ep_0\nu^{\f13}t}|k|\na_k\w\p{b}_k\big\|_{L_t^2([0,T_*);l_k^2L_y^2)}\\
 \nonumber      &+{\big \|e^{\ep_0\nu^{\f13}t}|k|\na_k\w\p{*,i}_k\big\|_{L_t^2([T_0,T_*);l_k^2L_y^2)}}{+ \nu^{\f1{4}}|\ln\nu|\big \|e^{\ep_0\nu^{\f13}t}|k|\na_k\w\p{a}_k\big\|_{L_t^2([T_0,T_*);l_k^2L_y^2)}}\\ \nonumber
 &\leq C\ep\nu^{\f{2}{3}}|\ln\nu|;
 \end{align} 
 \item Short-time Vorticity Bounds $(T_*< T_0)$: \begin{align}\label{ineq-size_om_short}  &{\nu^{-\f49}}\big \|e^{\ep_0\nu^{\f13}t}|k|\w\om{*,b}_k\big\|_{L_t^2([0,T_*);l_k^2L_y^2)}+{\nu^{-\f16}|\ln\nu|}\big \|e^{\ep_0\nu^{\f13}t}|k|\w\om{b}_k\big\|_{L_t^2([0,T_*);l_k^2L_y^2)}\\
    \nonumber   &{+\nu^{-\f{1}{12}}|\ln\nu|\big \|e^{\ep_0\nu^{\f13}t}|k|\w\om{a}_k\big\|_{L_t^2([0,T_*);l_k^2L_y^2)}}+\big \|e^{\ep_0\nu^{\f13}t}|k|\w\om{*,i}_k\big\|_{L_t^2([0,T_*);l_k^2L_y^2)}\\
    \nonumber&+\nu^{-\f{5}{12}}\big \|e^{\ep_0\nu^{\f13}t}\w\om{*,i}_k\big\|_{L_t^2([0,T_*);l_k^2L_y^2)}\leq C\ep\nu^{\f16}|\ln\nu|;
\end{align} 
 \item Long-time Vorticity Bounds $(T_*\geq T_0)$: \begin{align}\label{ineq-size_om}  &{\nu^{-\f49}}\big \|e^{\ep_0\nu^{\f13}t}|k|\w\om{*,b}_k\big\|_{L_t^2([0,T_*);l_k^2L_y^2)}+{\nu^{-\f16}|\ln\nu|}\big \|e^{\ep_0\nu^{\f13}t}|k|\w\om{b}_k\big\|_{L_t^2([0,T_*);l_k^2L_y^2)}\\
    \nonumber   &{+|\ln\nu|\big \|e^{\ep_0\nu^{\f13}t}|k|\w\om{a}_k\big\|_{L_t^2([T_0,T_*);l_k^2L_y^2)}}+\sum_{\sigma\in\e, \rm re}\big \|e^{\ep_0\nu^{\f13}t}|k|\w\om{\sigma}_k\big\|_{L_t^2([0,T_*);l_k^2L_y^2)}\leq C\ep\nu^{\f16}|\ln\nu|.
\end{align} 
\end{itemize}
Here, we adopt the notation $\na_k:=(ik,\pa_y)$ and  arrange the elements $\{\w\omega{*,b}, \w\omega{b}, \w\omega{a}, \w\omega{*,i}\}$ and the corresponding stream functions such that their theoretical upper bounds are expanding in size. 

The second family of estimates are spatially weighted $L_t^\infty L^2_{x,y}$-estimates of the vorticity, which are derived in a similar fashion as before (\eqref{hyp_om_*b}, Propositions \ref{pro:linear_boundary_corrector}):
\begin{align}\label{est-*b}&\hspace{-3cm}\text{Weighted Vorticity Bounds:}\\ \nonumber
        &\hspace{-2cm}{\nu^{-\f13}|\ln\nu|^{-1}}\big \|e^{\ep_0\nu^{\f13}t}|k|(1-|y|)\w\om{*,b}_k\big\|_{L_t^{\infty}([0,T_*);l_k^2L_y^2)}\\
        \nonumber&\hspace{-2cm}+  \big \|e^{\ep_0\nu^{\f13}t}|k|(1-|y|)\w\om{b}_k\big\|_{L_t^{\infty}([0,T_*);l_k^2L_y^2)}\leq C \ep\nu^{\f56}.
    \end{align}
The bound \eqref{est-*b} holds on the whole bootstrap horizon $[0,T_*)$ regardless of the short/long time regime.

With the analysis toolkit introduced, we start to present the estimates in the short-time/long-time regime. 

\end{subequations}

\subsection{Short Time Regime}\label{sec-shortT}
In this section, we use energy estimate to prove Proposition \ref{pro:shortT}. We stick to the decomposition 
 \eqref{level_1_error}, \eqref{om*_large_dcmp} in the time interval $t\in[0,T_*)\subset[0, T_0=\nu^{-\f16}]$:
\begin{align}\label{ini-dcmp}
 \omega_\nq=\w\omega a_\nq+\w\omega b_\nq+\w{\omega}{*,i}_\nq+\w{\omega}{*,b}_\nq.
\end{align}
Moreover, no decomposition is implemented on $\omega_0^{(*)}$. Similar relations hold for $u,\, \psi. $ 
\begin{proof}[Proof of Proposition \ref{pro:shortT}]
We decompose the proof into steps. Unless otherwise stated, all the $L_t^pX$-spaces in this proof are referring to $L^p_t([0,\min\{T_0,T_*\});X)$, and all $l_k^p$-norms are taken over $k\in\mathbb Z\backslash\{0\}$.  Additionally, since in the time interval $[0,T_0)$, $e^{\ep_0\nu^{\f13}t}\lesssim1$, we omit this weight during the proof.

\step{1: Energy dissipation relations. } First of all, we perform the $x$-Fourier transform on the $\w\omega{*,i}_\nq$-equation \eqref{eq:om_*i_nq}, and focus on the non-zero modes $( k\ne0)$ of the solution. The energy dissipation relation reads as follows
\begin{equation}\label{inner-we}
\begin{aligned}
  & \f12\f{d}{dt}\|\w\om{*,i}_k\|_{L^2}^2+\nu\|\na_k\w\om{*,i}_k\|_{L^2}^2=\Re\int_{-1}^1\big(ik\msc{E}[\w\om{*,i}_k]+\mathbb F_{k}\big)\overline{\w\om{*,i}_k}dy\\
  &=\Re\int_{-1}^1\Big[ik\pa_y^2\mathcal{U}(\w\p b_k+\w\p {*,i}_k+\w\p {*,b}_k)-\big( u \cdot\na(\w\om b_{\ne}+\w\om {*,i}_{\ne}+\w\om {*,b}_{\ne})\big)_k
  \\
  &\quad-\big((\w u b_{\ne}+\w u {*,i}_{\ne}+\w u {*,b}_{\ne})\cdot\nabla \w\om a\big)_k-u_0^{(*)}ik\w\om{a}_k\\
    &\quad+\Big(\nu k^2t^2\w\om{a}_k+\nu\D_k\w\om{a}_k+ik\pa_y^2\mathcal{U}\w\p a_k-(\w{ u }a\cdot\nabla\w\om a)_k\Big)+ik\mathscr{E}_k[\w\om{*,i}_k+\w\om{a}_k]\Big]\overline{\w\om{*,i}_k}dy\\
    &\quad-\Re\int_{-1}^1 u_k^2\pa_y\w\om{*}_0\overline{\w\om{*,i}_k}dy.
    \end{aligned}
\end{equation}
Here, $\na_k:=(ik,\pa_y)$ and boundary terms in integration by parts  vanish due to the boundary condition $\w\om{*,i}_k|_{y=\pm1}=0$. It is worth emphasizing that the treatment of the last term in \eqref{inner-we} requires a detailed balance hidden in the good unknown  introduced in \eqref{eq: good-unknown}. We will elaborate on that in the sequel.

Next, we develop the energy dissipation relation for the zero-mode. By the same derivation for the vorticity equation \eqref{eq-om*0}, we obtain the velocity equation:
\begin{equation}\label{eq-w0}
    \begin{cases}
        \pa_t\w{u}{*}_0 -\nu\pa_y^2\w{u}{*}_0 =-\displaystyle\sum_{k\in\mathbb Z\backslash\{0\}}u_k^2(t,y)\pa_yu_{-k}^1(t,y)=\sum_{k\in\mathbb Z\backslash\{0\}}u_k^2(t,y)\om_{-k}(t,y),\\
        \al \pa_y\w{u}{*}_0 (t,\pm1)=\mp \w{u}{*}_0 (t,\pm1),\quad \w{u}{*}_0 |_{t=0}=0.
    \end{cases}
\end{equation}
Here, the $\w u*_0$ is the horizontal component of the velocity field. We have also used the relation $\om_k=-\pa_y u^1_k+iku^2_k$ to rewrite the source term. 
Through integration by parts, we get
\begin{equation}\label{eq-w0-inner}
\begin{aligned}
   & \langle \pa_t \w{u}{*}_0-\nu\pa_y^2\w u*_0 ,-\pa_y^2\w{u}{*}_0 \rangle\\
   &=\f12\f{d}{dt}\Big(\f{1}{\al}(\w{u}{*}_0 (1))^2+\f{1}{\al}(\w{u}{*}_0 (-1))^2+\|\w\om*_0\|_{L^2}^2\Big)+\nu\|\pa_y^2\w{u}{*}_0 \|_{L^2}^2\\
   &=\sum_{k\in\mathbb Z\backslash\{0\}}\int_{-1}^1 u^2_k(\w\om a_{-k}+\w\om b_{-k})\overline{\pa_y\w\om*_0}dy+\sum_{k\in\mathbb Z\backslash\{0\}}\int_{-1}^1u^2_k\w\om {*,i}_{-k}\overline{\pa_y\w\om*_0}dy.
\end{aligned}
\end{equation}
Now all the main enemies have revealed themselves. Challenges arise when we consider the last term in \eqref{inner-we} and  the last term in \eqref{eq-w0-inner}. Let us focus on the $\w ua$-contributions $\int_{-1}^1u_k^{(a,2)}\pa_y\w\om*_0\overline{\w\om{*,i}_k}dy$ and $\int_{-1}^1u_k^{(a,2)}\w\om{*,i}_{-k}\overline{\pa_y\w\om*_0}dy$. The $\pa_y$-derivative on $\w\om*_0$ can potentially lead to an $\mathcal O(\nu^{-\f12})$-loss in size, which is unacceptable in our analysis. Fortunately, these two terms are the same in size and opposite in sign, which motivates us to introduce the energy $E_{\rm initial}$ defined in \eqref{E_int}. Moreover, this combination enables us to exploit the divergence-free structure, which yields the cancellation
\begin{align*}
    \sum_{k\in\mathbb Z\backslash\{0\}}\sum_{l\in\mathbb Z\backslash\{0,k\}}\int_{-1}^1\big(u_l^1i(k-l)\w\om{*,i}_{k-l}+u^2_l\pa_y\om_{k-l}^{ (*,i)}\big)\overline{\w\om{*,i}_k}dy=0.
\end{align*}
Therefore, we can obtain 
\begin{align*}
      &\f12\f{d}{dt}\Big(E_{\rm initial}+\f{1}{\al}(\w{u}{*}_0 (1))^2+\f{1}{\al}(\w{u}{*}_0 (-1))^2\Big)+\nu\|\na_k\w\om{*,i}_k\|_{l_k^2L_y^2}^2+\nu\|\pa_y\w\om*_0\|_{L_y^2}^2\\
   & = \sum_{k\in\mathbb Z\backslash\{0\}}\int_{-1}^1\pa_y^2\mathcal{U}ik(\w\p{b}_k+\w\p{*,i}_k+\w\p{*, b}_k)\overline{\w\om{*,i}_k}dy+\sum_{k\in\mathbb Z\backslash\{0\}}\int_{-1}^1\pa_y\p_0^{(*)}ik\big(\w\om{a}_k+\w\om{b}_k+\w\om{*,b}_k\big)\overline{\w\om{*,i}_k}dy\\
    &+\sum_{k\in\mathbb Z\backslash\{0\}}\int_{-1}^1\sum_{l\in\mathbb Z\backslash\{0,k\}}\Big(-il\p_l\pa_y\big(\w\om{b}_{k-l}+\w\om{*,b}_{k-l}\big)+\pa_y\p_li(k-l)\big(\w\om{b}_{k-l}+\w\om{*,b}_{k-l}\big)\Big)\overline{\w\om{*,i}_k}dy\\
    &-\sum_{k\in\mathbb Z\backslash\{0\}}\int_{-1}^1\sum_{l\in\mathbb Z\backslash\{0,k\}}( u ^{(b)}_l+ u ^{(*,i)}_l+ u ^{(*,b)}_l)\nabla_{k-l} \w\om a_{k-l}\overline{\w\om{*,i}_k}dy\\
    &-\sum_{k\in\mathbb Z\backslash\{0\}}\int_{-1}^1\big(-\nu t^2k^2\w\om{a}_k-\nu\D_k\w\om{a}_k-\pa_y^2\mathcal{U}ik\w\p a_k+(\w{ u }a\cdot\nabla\w\om a)_k\big)\overline{\w\om{*,i}_k}dy\\
    &+\sum_{k\in\mathbb Z\backslash\{0\}}ik\int_{-1}^1\mathscr{E}_k[\w\om{*,i}_k+\w\om{a}_k]\overline{\w\om{*,i}_k}dy+\sum_{k\in\mathbb Z\backslash\{0\}}\int_{-1}^1ik\p_k\big(\om_{-k}^{(a)}+\om_{-k}^{(b)}+\om_{-k}^{(*,b)}\big)\overline{\pa_y\w\om*_0}dy=:\sum_{j=1}^{7}T_j.
\end{align*} 
This concludes Step \# 1. The remaining steps is devoted to estimate each term in the decomposition. 

\step{2: Estimates of $T_1,T_2$-terms. }The estimates in this step are warm-ups for the bounds which are coming. We start with the $T_1$-term. Since we are in the short time regime $0\le t\le\min\{T_0, T_*\}$, we deduce that
   \begin{align}\label{est-psi-*i}
     \||k|(k,\pa_y)\w\p{*,i}_k\|_{L^2_tl_k^2L_y^2}\le\|\w\om{*,i}_k\|_{L^2_tl_k^2L_y^2}\le \nu^{-\f1{12}}\|\w\om{*,i}_k\|_{L^{\infty}_tl_k^2L_y^2}\le \ep\nu^{\f{7}{12}}|\ln\nu|.
   \end{align}
 After integrating in the short time region $t\in[0,\min\{T_0,T_*\})$, we combine the $\wt U$-estimate in Lemma \ref{Lem-U}, the stream function/vorticity estimates \eqref{ineq-size_short}, \eqref{ineq-size_om_short} in the analysis toolkit to obtain that 
    \begin{align*}
       &\int_0^{T_*} T_1dt
        \lesssim \ep\nu^{\f13}\big(\||k|\w\p{b}_k\|_{L^2_tl_k^2L_y^2 }+\||k|\w\p{*,i}_k\|_{L^2_tl_k^2L_y^2 }+\||k|\w\p{ *, b}_k\|_{L^2_tl_k^2L_y^2 }\big)\|\w\om{*,i}_k\|_{L_t^2 l_k^2L_y^2 }\\
        &\lesssim  \ep\nu^{\f13} \cdot\ep \nu^{\f 7{12}}|\ln \nu|\cdot \ep\nu^{\f7{12}}|\ln \nu|
        \lesssim \ep^3\nu^{\f{3}{2}}|\ln\nu|^2.
    \end{align*}
Next, we estimate the $T_2$-term. By the Gagliardo-Nirenberg inequality on $\pa_y\w\p{*}_0$ and the bootstrap assumption on $\w\om{*}_0$ \eqref{hyp_om_0_s}, we obtain that $\|\pa_y \w\psi*_0\|_{L_{t,y}^\infty}\leq C\nu^\f23|\ln \nu|$. Combining this with the short-time vorticity estimate \eqref{ineq-size_om_short}, we obtain 
     \begin{align*}
\int_0^{T_*}T_2dt
&\lesssim \|\pa_y\p_0^{(*)}\|_{L^{\infty}_tL^{\infty}_y}\|k \w\om a_{k}+k\w\om{b}_k+k\w\om{*,b}_{k}\|_{L^{2}_tl_k^2L_y^2}\|\w\om{*,i}_k\|_{L^2_tl_k^2L_y^2}\\
&\lesssim \ep\nu^\f23 |\ln \nu|\cdot\ep{\nu^\f14}\cdot \ep\nu^{\f7{12}}|\ln\nu|\ \lesssim \ep^3\nu^{\f32}|\ln\nu|^2.
\end{align*}
This concludes Step \# 2.

\step{3. Estimates of the $T_3,T_4$-terms.} 
For $T_3$, since $\p_k(y=\pm1)=0$, applying integration by parts, Young's convolution inequality and Cauchy-Schwarz inequality, we obtain
     \begin{align*}
       &\int T_3dt =\sum_{k\in\mathbb Z\backslash\{0\}}\sum_{l\in\mathbb Z\backslash\{0,k\}}\int_0^{T_*}\int_{-1}^1\big(\w\om{b}_{k-l}+\w\om{*,b}_{k-l}\big)\Big(ik\pa_y\p_l\overline{\w\om{*,i}_k}+il\p_l\overline{\pa_y\w\om{*,i}_k}\Big)dydt\\
    &   \le \Big(\||k|\pa_y\p_k\|_{L_t^2l_k^2L_y^{\infty} }\norm{\w\om{*,b}_{k}+\w\om{b}_{k}}_{L^2_tl_k^1L^2_y }+\|\pa_y\p_k\|_{L_t^2l_k^1L_y^{\infty}}\norm{|k|\big(\w\om{*,b}_{k}+\w\om{b}_{k}\big)}_{L^2_tl_k^2L_y^2 }\Big)\|\w\om{*,i}_k\|_{L^{\infty}_tl_k^2L_y^2 }\\
        &\quad+\|(1-|y|)^{-1}k\p_k\|_{L^2_tl_k^2L_y^{\infty} }\Big(\norm{(1-|y|)\w\om{*,b}_{k}}_{L^{\infty}_tl_k^1L_y^2 }+\norm{(1-|y|)\w\om{b}_{k}}_{L^{\infty}_tl_k^1L_y^2}\Big)\|\pa_y\w\om{*,i}_k\|_{L^2_tl_k^2L_y^2 }.\end{align*}
        Now we invoke the relation $\|\mathbf{1}_{k\neq 0}f_k\|_{l_k^1}\lesssim\|\mathbf{1}_{k\neq 0}kf_k\|_{l_k^2}$ to swap the $l_k^1$-norm to $l_k^2$-bound, and invoke the (weighted) vorticity estimates \eqref{ineq-size_om_short}, \eqref{est-*b} and the bootstrap assumptions $\|\w\om{*,i}_k\|_{L_t^{\infty}l^2_kL^2_{y}}+\nu^{\f12}\|\na_k\w\om{*,i}_k\|_{L^2_tl^2_kL_y^2}\leq C\ep \nu^\f23|\ln \nu|$ \eqref{hyp_om_e_s} to simplify the bound
        \begin{align}\begin{split}
         &\int T_3dt \lesssim \||k|\pa_y\p_k\|_{L_t^2l_k^2L_y^{\infty} }\Big(\big\||k|\w\om{*,b}_{k}\big\|_{L^2_tl_k^2L_y^2}+\big\||k|\w\om{b}_{k}\big\|_{L^2_tl_k^2L_y^2}\Big)\|\w\om{*,i}_k\|_{L^{\infty}_tl_k^2L_y^2 }\\
        &\quad+\|(1-|y|)^{-1}k\p_k\|_{L^2_tl_k^2L_y^{\infty} }\Big(\|(1-|y|)|k|\w\om{*,b}_{k}\|_{L^{\infty}_tl_k^2L_y^2 }+\|(1-|y|)|k|\w\om{b}_{k}\|_{L^{\infty}_tl_k^2L_y^2}\Big)\|\pa_y\w\om{*,i}_k\|_{L^2_tl_k^2L_y^2 }\\
        &\lesssim \||k|\pa_y\p_k\|_{L_t^2l_k^2L_y^{\infty} }\ \ep\nu^\f13\cdot \ep\nu^\f23|\ln\nu| +\|(1-|y|)^{-1}k\p_k\|_{L^2_tl_k^2L_y^{\infty} }\ \ep\nu^\f56\ \cdot \ep\nu^\f16|\ln\nu| .\end{split}\label{nl_pf_1}
           \end{align}
       Now we can see that the remaining task is to develop a stream function bound. We invoke the mean value theorem, the technical lemma \eqref{eq-weig}, the Gagliardo-Nirenberg interpolation, and the elliptic relation $\om^{(\s)}_k=\Delta_k\psi_k^{(\s)}$ to get
       \begin{align*}
           &\|(1-|y|)^{-1}k\p_k\|_{L^2_tl_k^2L_y^{\infty}}\lesssim\|k\pa_y\p_k\|_{L^2_tl_k^2L_y^{\infty}}\lesssim \sum_{(\sigma)\in\{(a),(b),(*,i),(*,b)\}}\|k\pa_y\p_k^{(\s)}\|_{L^2_tl_k^2L_y^{\infty}}\\
           &\lesssim \|k\pa_y\p_k^{(a)}\|_{L^2_tl_k^2L_y^{\infty}}+\sum_{(\sigma)\in\{(b),(*,i),(*,b)\}}\Big(\|k\pa_y\p_k^{(\s)}\|_{L^2_tl_k^2L_y^2}^{\f12}\|{k\pa_y^2\psi^{(\s)}_k}\|_{L^2_tl_k^2L_y^2}^{\f12}+\|k\pa_y\p_k^{(\s)}\|_{L^2_tl_k^2L_y^2}\Big).
       \end{align*}
       By the $\w \omega a$-estimates in Lemma \ref{lem:estimate-u1_k}, the smallness of $\om_{\rm in}$ \eqref{smallness},
       the estimates \eqref{ineq-size_short}, \eqref{ineq-size_om_short} and standard elliptic estimates $\|\pa_y^2\psi_k^{(\s)}\|_{L^2_y}\lesssim \|\om_k^{(\s)}\|_{L^2_y}$, we obtain
\begin{equation}\label{est-middle-u}
       \begin{aligned}
         \|&(1-|y|)^{-1}k\p_k\|_{L^2_tl_k^2L_y^{\infty}}+\|k\pa_y\p_k\|_{L^2_tl_k^2L_y^{\infty}}\\
         &\lesssim \|\om_{\ne;\rm in}\|_{H_{x,y}^2}+\sum_{(\sigma)\in\{(b),(*,i),(*,b)\}}\Big(\|k\pa_y\p_k^{(\s)}\|_{L^2_tl_k^2L_y^2}^{\f12}\|{k\om^{(\s)}_k}\|_{L^2_tl_k^2L_y^2}^{\f12}+\|k\pa_y\p_k^{(\s)}\|_{L^2_tl_k^2L_y^2}\Big)\\
         &\lesssim\ep\nu^{\f13}+\ep^\f12\nu^{\f7{24}}|\ln \nu|^\f12\cdot \ep^\f12\nu^{\f1{12}}|\ln \nu|^\f12+\ep\nu^{\f7{12}}|
        \ln \nu|\lesssim\ep\nu^{\f13}.
       \end{aligned}
       \end{equation}
        Combining this estimate with prior developed bound \eqref{nl_pf_1}, we conclude that
        \begin{align*}
        \int_0^{T_*}T_3dt
        \lesssim\ep^3\nu^{\f43}|\ln\nu|.
    \end{align*}
Using Lemma \ref{lem:estimate-u1_k}, the smallness of $\om_{\rm in}$ \eqref{smallness}, and an argument similar to that for $T_3$, we can deduce that
\begin{align*}
     \int_0^{T_*}T_4dt
        &\lesssim \Big(\sum_{(\sigma)\in\{(b),(*,i),(*,b)\}}\| u ^{(\s)}_k\|_{L_t^{2}l_k^2L_y^2}\Big)\||k|\na_{k} \w\om a_{k}\|_{L^{\infty}_tl_k^2L_y^{\infty} }\|\w\om{*,i}_k\|_{L^2_tl_k^2L_y^2 }
   \\ &\lesssim \ep\nu^{\f7{12}}|\ln\nu|\cdot \ep\nu^\f16\cdot\ep \nu^\f7{12}|\ln\nu|\lesssim\ep^3\nu^{\f43}|\ln\nu|^2.
\end{align*}

\step{4: Estimates of $T_5$--$T_7$ terms. }
The $T_5$-term bound is a direct consequence of Lemma \ref{lem:estimate-u1_k} and the bootstrap assumption \eqref{hyp_om_e_s}, 
\begin{align*}
    \int_0^{T_*}T_5dt
  & \lesssim \Big(\big\|\nu t^2k^2\w\om{a}_k+\nu\D_k\w\om{a}_k+\pa_y^2\mathcal{U}ik\w\p a_k-(\w{ u }a\cdot\nabla\w\om a)_k\big\|_{L^1_tl_k^2L_y^2}\Big)\big\|\w\om{*,i}_k\big\|_{L_t^{\infty}l_k^2L_y^2}\\
    &\lesssim \big(\nu^{\f12}\ep\nu^{\f13}+(\ep\nu^{\f13})^2+\ep^2\nu^{\f23}|\ln\nu|\big)\ep\nu^{\f23}|\ln\nu|\lesssim  \ep^3\nu^{\f43}|\ln\nu|^2+\ep^2\nu^{\f32}|\ln\nu|.
\end{align*}
Next we treat the $T_6$-term which involves the error function $\mathscr{E}_k$ in \eqref{Frozen_err}. Applying the $\mathscr{E}_k[g_k]$-estimates in Lemma \ref{Lem-F-Y}, the $\w\om{a}$ bounds in Lemma \ref{lem:estimate-u1_k} (coupled with $\om_{\rm in}$-smallness \eqref{smallness}),  as well as the bounds in analysis toolkit \eqref{ineq-size_short}, \eqref{ineq-size_om_short}, there holds
\begin{align*}
  \int_0^{T_*} T_6dt    &\lesssim \nu^{\f76}\Big(\||k|^{\f56}\w\om{a}_k\|_{L^2_t([0,T_*);l_k^2L_y^{\infty})}+\||k|^{\f56}\w{\pa_y\p}{a}_k\|_{L^2_t([0,T_*);l_k^2L_y^{\infty})}\\
    &\quad+\big\||k|\pa_y\w\p{*,i}_k\big\|_{L^2_t([0,T_*);l_k^2L_y^2) }^{\f12}\big\||k|^{\f23}\w\om{*,i}_k\big\|_{L^2_t([0,T_*);l_k^2L_y^2) }^{\f12}\Big)\big\|\w\om{*,i}_k\big\|_{L_t^{2}l_k^2L_y^2}    \\
    &\lesssim \nu^{\f76}\big(\nu^{-\f{1}{12}}\ep\nu^{\f13}+\nu^{\f{1}{12}}\ep\nu^{\f13}+\ep\nu^{\f38}{|\ln\nu|}\big)\ep\nu^{\f{7}{12}}|\ln\nu|\lesssim \ep^2\nu^{2}|\ln\nu| .
\end{align*}
Finally, we come to the $T_7$-term. Since $\p_k(\pm1)=0$, we integrate by parts. By the $\w\om{a}$-estimates in Lemma \ref{lem:estimate-u1_k} (coupled with $\om_{\rm in}$-smallness \eqref{smallness}), the velocity estimate \eqref{ineq-size_short}, the weighted vorticity bound \eqref{est-*b}, the weighted velocity bound \eqref{est-middle-u}, and the bootstrap assumption \eqref{hyp_om_0_s}, we get
\begin{align*}
       & \int_0^{T_*}T_{7}dt        \le  \Big\||k|\big(\pa_y\p_k^{(b)}+\pa_y\p_k^{(*,i)}+\pa_y\p_k^{(*,b)}\big)\Big\|_{L_t^2l_k^2L_y^2 }\Big\|\w\om a_{-k}\Big\|_{L^2_tl_k^2L^{\infty}_y }\|\w\om*_0\|_{L^{\infty}_tL^2_y }\\
          &\quad + \|(1-|y|)^{-1}k\p_k\|_{L^2_tl_k^2L_y^{\infty} }
          \big(\|(1-|y|)\om_{-k}^{(b)}\|_{L^{\infty}_tl_k^2L_y^2 }+\|(1-|y|)\om_{-k}^{(*,b)}\|_{L^{\infty}_tl_k^2L_y^2 }\big)\|\pa_y\w\om*_0\|_{L^2_tL_y^2 }\\
         &\quad+\Big\||k|\big(\p_k^{(b)}+\p_k^{(*,i)}+\p_k^{(*,b)}\big)\Big\|_{L^2_tl_k^2L^2_y }\|\pa_y\w\om a_{-k}\|_{L^2_tl_k^2L_y^{\infty} }\|\w\om*_0\|_{L^{\infty}_tL_y^2 }\\
         &\quad +\big(\||k|\p_k^{(a)}\pa_y\w\om a_{-k}\|_{L^1_tl_k^1L^{2}_y}+\||k|\pa_y\p_k^{(a)}\|_{L^1_tl_k^2L^{2}_y}\|\w\om a_{-k}\|_{L^{\infty}_tl_k^2L_y^{\infty} }\big)\|\w\om*_0\|_{L^{\infty}_tL_y^2 }\\
         &\lesssim \ep\nu^{\f{7}{12}}|\ln\nu|\cdot\nu^{-\f{1}{12}}\ep\nu^{\f13}\cdot\ep\nu^{\f23}|\ln\nu|+\ep\nu^{\f13}\cdot \ep\nu^{\f56}\cdot\ep\nu^{\f16}|\ln\nu|\\
         &\quad+\ep\nu^{\f{7}{12}}|\ln\nu|\cdot \ep\nu^{\f{1}{12}}\cdot\ep\nu^{\f23}|\ln\nu|+\ep^3\nu^{\f43}|\ln\nu|^2\lesssim \ep^3\nu^{\f43}|\ln\nu|^2.
    \end{align*}

    \step{5: Conclusion. } We further observe that at $t=0$, the quantity $E_{\rm initial}(0)=0$.
In summary, the main decomposition, the $T_1$--$T_7$ bounds, when combined with \eqref{est-psi-*i}, yield the final bound
\begin{equation}\label{est-S-T-*i}
\begin{aligned}
   & \sup_{t\in[0,\min\{T_0,T_*\})}\lf(E_{\rm initial}(t)+\al^{-1}(\w{u}{*}_0 (t,1))^2+\al^{-1}(\w{u}{*}_0 (t,-1))^2\rg)+\nu^{\f16}\||k|\na_k\w\p{*,i}_k\|_{L^2l_k^2L_y^2}^2\\
   &\qquad+\nu\|\na_k\w\om{*,i}_k\|_{L^2_tl_k^2L^2_y}^2+\nu\|\pa_y\w\om{*}_0\|_{L_t^2L_y^2}^2\lesssim \ep^3\nu^{\f43}|\ln\nu|^2+\ep^2\nu^{\f32}|\ln\nu|.
\end{aligned}
\end{equation}
Since $\al>0$ is a constant independent of \{$\nu$, $k$, $\ep$\}, by taking $\ep$ and $\nu$ small enough, we finish the proof.
\end{proof}

\subsection{Long Time Regime: Estimates of Non-zero Modes} \label{sec-longT_nz}

In this part, we derive the estimates of non-zero modes. Specifically, in the long time regime, we take the reaction part into consideration and prove Proposition \ref{lem-re} and Proposition \ref{prop-e}. In this section, we 
stick to the decomposition 
 \eqref{level_1_error}, \eqref{om*_large_dcmp}:
\begin{align}\label{lol-dcmp}
\omega_0^{(*)}=\w\omega{\rm re}_0+\w\omega{\rm e}_0,\qquad \omega_\nq=\w\omega a_\nq+\w\omega b_\nq+\w{\omega}{\rm re}_\nq+\w{\omega}{\rm e}_\nq+\w{\omega}{*,b}_\nq.
\end{align}
Similar relations hold for $u,\, \psi$. Additionally, unless otherwise stated, all the $L_t^pX$-spaces in this subsection are referring to $L^p_t([T_0=\nu^{-\f16},T_*);X), \, T_0< T_*$ , and all $l_k^p$-norms are taken over $k\in\mathbb Z\backslash\{0\}$. 
\subsubsection{Estimate of the Reaction Part}
 The estimate of the reaction part $\w\omega{\rm re}_\nq$ \eqref{eq-om_re} is identical to \cite[Proposition 2.5]{WeiZhangQuasiL}. We recall the statement here and refer the interested readers to its proof in \cite{WeiZhangQuasiL}.
\begin{lemma}[\cite{WeiZhangQuasiL}]\label{pro:om_re}It holds that 
\begin{align*}
  &  \|e^{\ep_0\nu^{\f13}t}\w\om{\rm re}_{\ne}\|_{L_t^{\infty}([T_0,T_\ast);L^2_{x,y})}+\nu^{\f12}\|e^{\ep_0\nu^{\f13}t}\na\w\om{\rm re}_{\ne}\|_{L^2_t([T_0,T_\ast);L^2_{x,y})}\\
  &\quad+\|e^{\ep_0\nu^{\f13}t}\pa_x\nabla\w\p{\rm re}\|_{L^2_t([T_0,T_\ast);L^2_{x,y})}+\nu^{\f12}\|e^{\ep_0\nu^{\f13}t}\pa_x\w\om{\rm re}\|_{L^1_t([T_0,T_\ast);L^2_{x,y})}\\
  &\quad\lesssim  \nu^{-\f13}\|e^{\ep_0\nu^{\f13}t}\na(u^2-u^{(a,2)})\|_{L^2_t([T_0,T_\ast);L^2_{x,y})}\|\om_{\rm in}\|_{H^3}.
\end{align*}
\end{lemma}
Based on the above lemma, we proceed to establish the proof of Proposition \ref{lem-re}.
\begin{proof}[Proof of Proposition \ref{lem-re}] 
Recall that $\w\om{\rm re}_{\ne}$ solves \eqref{eq-om_re}, Lemma \ref{pro:om_re} gives
\begin{equation}\label{eq-est-Re}
\begin{aligned}
    &  \|e^{\ep_0\nu^{\f13}t}\w\om{\rm re}_{\ne}\|_{L_t^{\infty}([T_0,T_\ast);L^2_{x,y})}+\nu^{\f12}\|e^{\ep_0\nu^{\f13}t}\na\w\om{\rm re}_{\ne}\|_{L^2_t([T_0,T_\ast);L^2_{x,y})}\\
  &\quad+\|e^{\ep_0\nu^{\f13}t}\pa_x\nabla\w\p{\rm re}\|_{L^2_t([T_0,T_\ast);L^2_{x,y})}+\nu^{\f12}\|e^{\ep_0\nu^{\f13}t}\pa_x\w\om{\rm re}\|_{L^1_t([T_0,T_\ast);L^2_{x,y})}\\
  &\quad\lesssim  \nu^{-\f13}\|e^{\ep_0\nu^{\f13}t}\na(u^2-u^{(a,2)})\|_{L^2_t([T_0,T_\ast);L^2_{x,y})}\|\om_{\rm in}\|_{H^3}.
\end{aligned}
\end{equation}
Compared with the conclusion in Proposition \ref{lem-re}, estimate of $\nu^{\f14}\|e^{\ep_0\nu^{\f13}t}|D_x|^{\f12}\w\om{\rm re}_{\ne}\|_{L_t^2([T_0,T_*);L_{x,y}^2)}$ is lacked.
However, an application of Cauchy-Schwarz inequality and \eqref{eq-est-Re}, yields that
\begin{align*}
&\nu^{\f12}\|e^{\ep_0\nu^{\f13}t}|D_x|^{\f12}\w\om{\rm re}_{\ne}\|_{L_t^2([T_0,T_*);L_{x,y}^2)}^2
\lesssim \nu^{\f12}\int_{T_0}^{T_*}\|ke^{\ep_0\nu^{\f13}t}\w\om{\rm re}_k\|_{l_k^2L_y^2}dt\sup_{t\in[T_0,T_*)}\|e^{\ep_0\nu^{\f13}t}\w\om{\rm re}_k\|_{l_k^2L_y^2}\\
&\lesssim \nu^{\f12}\|e^{\ep_0\nu^{\f13}t}\pa_x\w\om{\rm re}\|_{L_t^1([T_0,T_*);L^2_{x,y})} \|e^{\ep_0\nu^{\f13}t}\w\om{\rm re}_{\ne}\|_{L_t^{\infty}([T_0,T_\ast);L^2_{x,y})}\\
    &\lesssim  \nu^{-\f23}\|e^{\ep_0\nu^{\f13}t}\na(u^2-u^{(a,2)})\|_{L^2_t([T_0,T_\ast);L^2_{x,y})}^2\|\om_{\rm in}\|_{H^3}^2.
\end{align*}
Now, we only need to estimate $\|e^{\ep_0\nu^{\f13}t}\na(u^2-u^{(a,2)})\|_{L^2_t([T_0,T_\ast);L^2_{x,y})}$. Based on the decomposition \eqref{lol-dcmp} and the Plancherel equality, we use the velocity estimate \eqref{ineq-size} to obtain
\begin{align*}
  &\|e^{\ep_0\nu^{\f13}t}\na(u^2-u^{(a,2)})\|_{L^2_t([T_0,T_*);L^2_{x,y})}^2\lesssim \ep^2\nu^{\f43}|\ln\nu|^2.
\end{align*}
At this end, combined with the smallness \eqref{smallness}, we can conclude that
\begin{align*}
   & \|e^{\ep_0\nu^{\f13}t}\w\om{\rm re}_{\ne}\|_{L_t^{\infty} L^2_{x,y}}+\nu^{\f12}\|e^{\ep_0\nu^{\f13}t}\na\w\om{\rm re}_{\ne}\|_{L^2_t L^2_{x,y}}+\|e^{\ep_0\nu^{\f13}t}\pa_x\nabla\w\p{\rm re}\|_{L^2_t L^2_{x,y}}\\
   &\quad+\nu^{\f14}\|e^{\ep_0\nu^{\f13}t}|D_x|^{\f12}\w\om{\rm re}_{\ne}\|_{L_t^2L_{x,y}^2}\lesssim  \nu^{-\f13}\|e^{\ep_0\nu^{\f13}t}\na(u^2-u^{(a,2)})\|_{L^2_t L^2_{x,y}}\|\om_{\rm in}\|_{H^3}\lesssim\ep^2\nu^{\f23}|\ln\nu|.
\end{align*}
Taking $\ep$ small enough, we finish the proof.
\end{proof}
\subsubsection{Estimates of the Error Part}\label{sec:error}
In this part, we apply Corollary \ref{cor-e-E} (taking $c=\ep_0$) to establish the estimates of $\w\om{\rm e}_k$ which solves \eqref{eq-we}.
The $x$-Fourier transform on \eqref{eq-we} yields the equation ($k\ne0$)
\begin{align} \begin{cases}
\pa_t\w\om{\e}_k+\mathcal{U}ik\w\om{\e}_k-\nu\de_k\w\om{\e}_k=ik\msc{E}[\w\om{\e}_k]+\w{\mathbb F}\e_{k},\\
         \de_k\w\psi{\e}_k=\w\om{\e}_k,\quad
        \w\om{\e}_k|_{t=T_0}=\w\om{*,i}_k(T_0),\quad \w\om{\e}_k(t,\pm1)=0.
\end{cases}
\end{align}
The source $ik\mathscr{E}[\w\om{\rm e}_k]+\mathbb F_k^{\rm (e)}$ can be further decomposed as $
   ik\mathscr{E}[\w\om{\rm e}_k]+\mathbb F_k^{\rm (e)}=\mathbb F_{k;1}+\mathbb F_{k;2}+\pa_y\mathbb F_{k;3}.$
Define the index set $\mathcal{I}:= \{(b),(\e), ({\rm re}),(*,b)\}$, the forcings have the representations
\begin{align}\label{defn_F_k1}\begin{split}
    \mathbb F_{k;1}:=&-\w{u}{*}_0 ik\w\om{a}_k-\sum_{l\in\mathbb Z\backslash\{0,k\}}{ \sum_{(\sigma)\in\mathcal I}\w u\sigma_l}\cdot \big(i(k-l),\pa_y+i(k-l)t\big)\w\om a_{k-l}\\
     & +\big(\nu k^2t^2\w\om{a}_k+\nu\D_k\w\om{a}_k+\pa_y^2\mathcal{U}ik\w\p a_k-(\w{ u }a\cdot\nabla\w\om a)_k\big)\\
     &+ik\sum_{l\in\mathbb Z\backslash\{0,k\}}\sum_{(\sigma)\in\mathcal I}\big(\pa_y\w\p{b}_l+\pa_y\p_l^{(*,b)}\big)\w\om\sigma_{k-l}+ik\sum_{l\in\mathbb Z\backslash\{0,k\}}\pa_y\w\p a_l\w\om{b}_{k-l}=:\sum_{j=1}^5\mathbb{F}_{k;1j},
     \end{split}\\ \label{defn_F_k2}\begin{split}
     \mathbb F_{k;2}:=&ik\pa_y^2\mathcal{U}\sum_{(\sigma)\in\mathcal I}\w \p\sigma_k+ik\sum_{l\in\mathbb Z\backslash\{0,k\}}\pa_y\w\p a_l\big(\w\om{\rm re}_{k-l}+\w\om{\rm e}_{k-l}+\w\om{*,b}_{k-l}\big)\\
     &+ik\Big(\pa_y\p_0^{(*)}\sum_{(\sigma)\in\mathcal I}\om_k^{(\sigma)}+\pa_y\p_k\om_0^{(*)}\Big)+ik\mathscr{E}_k[\w\om{a}_k+\w\om{*,i}_k]\\
     & +ik\sum_{l\in\mathbb Z\backslash\{0,k\}}\sum_{(\sigma)\in\mathcal I}\pa_y(\w\p{\e}_l+\p_l^{(\rm re)})\w\om{\sigma}_{k-l}=:\sum_{j=1}^5\mathbb{F}_{k;2j},
     \end{split}\\ \label{defn_F_k3}
    \mathbb F_{k;3}:=&-ik\p_k\w\om{*}_{0}-\sum_{l\in\mathbb Z\backslash\{0,k\}}\sum_{(\sigma)\in\mathcal I}il\p_l\w\om{\sigma}_{k-l}=:\mathbb{F}_{k;31}+\mathbb{F}_{k;32}.
\end{align}

We recall from Corollary \ref{cor-e-E} that to prove the estimate for the $\w \om \e$ in Proposition \ref{prop-e}, one needs to develop the $L_t^1l_k^2L_y^2$-estimate for $\mathbb{F}_{k;1}$ and the $L_t^2l_k^2L_y^2$-estimate for $\mathbb{F}_{k;2},\ \mathbb{F}_{k;3}$. Unless otherwise stated, all the $L_t^pX$-spaces in this subsection are referring to $L^p_t([T_0=\nu^{-\f16},T_*);X)$, and \emph{all $l_k^p$-norms are taken over $k\in\mathbb Z\backslash\{0\}$}.

We present the estimates of $\mathbb F_{k;j}$, $(j=1,2,3)$  in the following three lemmas.
\begin{lemma}
Under the bootstrap assumption \eqref{Btstrp_Hyp}, the $\mathbb{F}_{k;1}$  \eqref{defn_F_k1} satisfies
\begin{align}
 \|e^{\ep_0\nu^{\f13}t}\mathbb F_{k;1}\|_{L_t^1([T_0,T_*);l_k^2L_y^2)}\leq C\ep^2\nu^{\f23}|\ln\nu|. \label{F_k_1_est}
\end{align}
\end{lemma}
\begin{proof}
From the expansion \eqref{defn_F_k1}, we observe that in order to derive \eqref{F_k_1_est}, it is sufficient to estimate the sum
\begin{align}\label{F_k_1_est_pf1}
\sum_{j=1}^5\|e^{\ep_0\nu^{\f13}t}\mathbb F_{k;1j}\|_{L^1_t([T_0,T_*);l_k^2L_y^2)}.
\end{align}

First of all, we estimate the $\mathbb{F}_{k;11}$-contribution. Recalling the definition of $\mathbb{F}_{k;11}$ \eqref{defn_F_k1}, the $\w\om a$-estimates in Lemma \ref{lem:estimate-u1_k}, the smallness assumption of $\om_{\rm in}$ \eqref{smallness}, the Gagliardo-Nirenberg on $\w u*_0$ and the bootstrap assumption of $\w\omega*_0$ \eqref{hyp_om_0_l}, we estimate the norm as follows
\begin{align}\label{F_k_1_est_pf2}\begin{split}
    &\hspace{-.25cm}
    \|e^{\ep_0\nu^{\f13}t}\mathbb F_{k;11}\|_{L^1_tl_k^2L_y^2}\le \|\w{u}{*}_0 \|_{L^{\infty}_tL_y^{\infty}}\|e^{\ep_0\nu^{\f13}t}|k|\w\om{a}_k\|_{L_t^1l_k^2L_y^2}\\
    &\leq C\Big(\|\w\om{*}_0\|_{L^{\infty}L^{2} }+\sup_{t\in[T_0,T_*)}|u_0^{(*)}(\pm1)|\Big)\cdot \nu^{-\f13}\|\om_{\ne;\rm in}\|_{H_{x,y}^3}\leq C\ep^2\nu^{\f23}|\ln\nu|. 
\end{split}
\end{align}

Next, we estimate the $\mathbb{F}_{k;12}$-contribution. Using Young's convolution inequality in the $k$-variable, combining with the $\w\om a_k$-estimates in Lemma \ref{lem:estimate-u1_k}, the smallness assumption of $\om_{\rm in}$ in \eqref{smallness}, and the velocity estimate \eqref{ineq-size}, we estimate the norm as follows

\begin{align} 
    \nonumber &
    \|e^{\ep_0\nu^{\f13}t}\mathbb F_{k;12}\|_{L_t^1l_k^2L_y^2}
   \leq C\Big(\sum_{(\sigma)\in\mathcal I }\|e^{\ep_0\nu^{\f13}t}u^{(\sigma)}_k\|_{L_t^2l_k^2L_y^2}\Big)\big\|e^{\ep_0\nu^{\f13}t}(k,\pa_y+ikt)\w\om a_{k}\big\|_{L_t^2l_k^1L^{\infty}_y}\\
 \label{F_k_1_est_pf12} &\lesssim \ep\nu^{\f23}|\ln\nu|\int_{T_0}^{T_*}\sum_{k\in\mathbb Z\backslash\{0\}}e^{-\nu^{\f13}|k|^{\f23}t}\|\om_{k;\rm in}\|_{H^{2}_{k,y}}dt\lesssim\ep^2\nu^{\f23}|\ln\nu|.
    \end{align}

    The estimate for $\mathbb{F}_{k;13}$ is a direct consequence of the $\w\om{a}$ bounds in Lemma \ref{lem:estimate-u1_k}, 
      \begin{align}
        \label{F_k_1_est_pf3}
        \begin{split}
        &   \|e^{\ep_0\nu^{\f13}t}\mathbb F_{k;13}\|_{L_t^1l_k^2L_y^2}
         \le\nu^{\f13}\|\om_{\ne;\rm in}\|_{H_{x,y}^4}+\ep\nu^{\f13}\big\|\langle t\rangle^{-2}\|\om_{\ne;\rm in}\|_{H_{x,y}^3}\big\|_{L_t^1}\\
         &\qquad+\big\|\langle t\rangle^{-1}\|\om_{\ne;\rm in}\big\|_{H^4_{x,y}}^2\big\|_{L^1_t}
        \lesssim \ep\nu^{\f23}+\ep^2\nu^\f23|\ln\nu|.\end{split}
    \end{align}
    
For the $\mathbb{F}_{k;14}$-term, using the Gagliardo-Nirenberg inequality and standard elliptic estimate, for $\D\p=\om$, we have
$
    \|\pa_y\p\|_{L^{\infty}_y}^2\lesssim\|\pa_y\p\|_{L_y^2}\|\pa_y^2\p\|_{L_y^2}\lesssim\|\pa_y\p\|_{L_y^2}\|\om\|_{L_y^2}
$. Then, we apply Young's convolution inequality, Cauchy-Schwarz inequality, the velocity estimate \eqref{ineq-size} and the vorticity estimate \eqref{ineq-size_om} to obtain
\begin{equation}\label{F_k_1_est_pf4}
\begin{aligned}
    &\|e^{\ep_0\nu^{\f13}t}\mathbb F_{k;14}\|_{L_t^1l_k^2L_y^2}
    \leq \big \||k|e^{\ep_0\nu^{\f13}t}(\pa_y\p_k^{(*,b)}+\pa_y\p_k^{(b)})\big\|_{L_t^2l_k^2L_y^{\infty}}\Big(\sum_{(\sigma)\in\mathcal{I}}\big\|e^{\ep_0\nu^{\f13}t}\om_{k}^{(\sigma)}\big\|_{L_t^2l_k^1L^2_y}\Big)\\
    & \quad+ \|e^{\ep_0\nu^{\f13}t}(\pa_y\p_k^{(*,b)}+\pa_y\p_k^{(b)})\|_{L_t^2l_k^1L^{\infty}_y}\Big(\sum_{(\sigma)\in\mathcal{I}}\||k|e^{\ep_0\nu^{\f13}t}\om_{k}^{(\sigma)}\|_{L_t^2l_k^2L_y^2}\Big)\\
    &\lesssim \big\||k|e^{\ep_0\nu^{\f13}t}(\pa_y\p_k^{(*,b)}+\pa_y\p_k^{(b)})\big\|_{L_t^2l_k^2L_y^{2}}^{\f12}\big\||k|e^{\ep_0\nu^{\f13}t}(\om_k^{(*,b)}+\om_k^{(b)})\big\|_{L_t^2l_k^2L_y^{2}}^{\f12}\\
    &\qquad\times\Big(\sum_{(\sigma)\in\mathcal{I}}\big\||k|e^{\ep_0\nu^{\f13}t}\om_{k}^{(\sigma)}\big\|_{L_t^2l_k^2L_y^2}\Big)\lesssim (\ep\nu^{\f23})^{\f12}(\ep\nu^{\f{1}{3}})^{\f12}\ep\nu^{\f16}|\ln\nu|\lesssim \ep^2\nu^{\f23}|\ln\nu|.
\end{aligned}
\end{equation}

Finally, for the $\mathbb F_{k;15}$-term, by Young's convolution inequality, together with $\w\om{a}$-estimates in Lemma \ref{lem:estimate-u1_k}, $\w\om{b}$-bound in \eqref{ineq-size_om}, and the smallness \eqref{smallness}, we get
\begin{equation}\label{F_k_1-est-pf5}
\begin{aligned}
   & \|e^{\ep_0\nu^{\f13}t}\mathbb F_{k;15}\|_{L_t^1l_k^2L_y^2}\le \||k|e^{\ep_0\nu^{\f13}t}\pa_y\p^{(a)}_k\|_{L_t^2l_k^2L_y^{\infty}}\|e^{\ep_0\nu^{\f13}t}\w\om{b}_{k}\|_{L_t^2l_k^1L_y^2}\\
   &\qquad+\|e^{\ep_0\nu^{\f13}t}\pa_y\p^{(a)}_k\|_{L_t^2l_k^1L_y^{\infty}}\||k|e^{\ep_0\nu^{\f13}t}\w\om{b}_{k}\|_{L_t^2l_k^2L_y^2}\\
   &\lesssim \big\||k|e^{\ep_0\nu^{\f13}t}\pa_y\p^{(a)}_k\big\|_{L_t^2l_k^2L_y^{\infty}}\big\||k|e^{\ep_0\nu^{\f13}t}\w\om{b}_{k}\big\|_{L_t^2l_k^2L_y^2}\lesssim \nu^{\f1{12}}\|\om_{\ne;\rm in}\|_{H_{x,y}^3}\ep\nu^{\f13}\lesssim \ep^2\nu^{\f34}.
\end{aligned}
\end{equation}
Combining the estimates \eqref{F_k_1_est_pf1}-\eqref{F_k_1-est-pf5}, we obtain the conclusion of the lemma.\end{proof}

Next, we consider the $\mathbb F_{k;2}$-forcing in \eqref{defn_F_k2}.
\begin{lemma}
Under the bootstrap assumption \eqref{Btstrp_Hyp}, the $\mathbb{F}_{k;2}$  \eqref{defn_F_k2} satisfies the estimate
\begin{equation}
\begin{aligned}
  & \sum_{j=1}^4\lf\|(\nu|k|^2)^{-\f14}e^{\ep_0\nu^{\f13}t}\mathbb F_{k;2j}\rg\|_{L^2_tl_k^2L_y^2}+\lf\|(\nu|k|^2)^{-\f{101}{400}}e^{\ep_0\nu^{\f13}t}\mathbb F_{k;25}\rg\|_{L^2_tl_k^2L_y^2}\lesssim \ep^2\nu^{\f23}|\ln\nu|+\ep\nu^{\f{13}{12}}.
\end{aligned}\label{F_k_2_est}
\end{equation}
\end{lemma}
\begin{proof}
Recall the decomposition of $\mathbb F_{k;2}$ defined in \eqref{defn_F_k2}. We organize this proof into steps.
\step{1. The estimate of $\mathbb{F}_{k;21}$:} First, we consider the $\mathbb F_{k;21}$-contribution in \eqref{defn_F_k2}. For the velocity $\w\p{\s}$, we expand the weight from $|k|^{\f12}$ to $|k|$ and invoke the estimate \eqref{ineq-size} in the toolkit. Then, combined with the $\mathcal{U}$-estimate  \eqref{ineq-U}, it yields
\begin{equation*}
\begin{aligned}
    \||k|^{-\f12}e^{\ep_0\nu^{\f13}t}\mathbb F_{k;21}\|_{L_t^2l_k^2L_y^2}\lesssim  \ep\nu^{\f13}\sum_{(\sigma)\in \mathcal{I}}\||k|^{\f12}e^{\ep_0\nu^{\f13}t}\w\psi{\sigma}_k\|_{L^2_tl_k^2L_y^2}
    \lesssim \ep^2\nu|\ln\nu|.
\end{aligned}
\end{equation*}
      \step{2. Estimates of $\mathbb{F}_{k;22}$-$\mathbb{F}_{k;24}$ terms:} For the $\mathbb F_{k;22}$-term, applying Young's convolution inequality yields that
      \begin{equation*}
      \begin{aligned}
         & \big\||k|^{-\f12}e^{\ep_0\nu^{\f13}t}\mathbb F_{k;22}\big\|_{L^2_tl_k^2L_y^2}\le \big\||k|^{\f12}\pa_y\w\p{a}_k\big\|_{L^{\infty}_tl_k^1L_y^{\infty}}\big\|\w\om{\rm e}_{k}+ \w\om{\rm re}_{k}+\w\om{*,b}_{k}\big\|_{L^2_tl_k^2L_y^2}\\
         &\qquad +\big\|\pa_y\w\p{a}_k\big\|_{L^{\infty}_tl_k^1L_y^{\infty}}\big\||k|^{\f12}(\w\om{\rm e}_{k}+ \w\om{\rm re}_{k}+\w\om{*,b}_{k})\big\|_{L^2_tl_k^2L_y^2}\\
         &\lesssim\big\||k|^{\f12}\pa_y\w\p{a}_k\big\|_{L^{\infty}_tl_k^1L_y^{\infty}}\big\||k|^{\f12}(\w\om{\rm e}_{k}+ \w\om{\rm re}_{k}+\w\om{*,b}_{k})\big\|_{L^2_tl_k^2L_y^2}.
      \end{aligned}
      \end{equation*}
       The bootstrap assumptions \eqref{hyp_h_om_re}-\eqref{hyp_om_*b} imply 
       \begin{equation}\label{est-L2L2-weight}
           \begin{aligned}
           & \big\||k|^{\f12}(\w\om{\rm e}_{k}+ \w\om{\rm re}_{k}+\w\om{*,b}_{k})\big\|_{L^2_tl_k^2L_y^2}\le\big\||k|^{\f12}(\w\om{\rm e}_{k}+ \w\om{\rm re}_{k})\big\|_{L_t^2l_k^2L_y^2}+\big\||k|^{\f{13}{15}}\w\om{*,b}_{k}\big\|_{L_t^2l_k^2L_y^2}\\
           &\lesssim \ep\nu^{\f23}|\ln\nu|\nu^{-\f14}+\ep\nu|\ln\nu|\nu^{-\f13}\lesssim \ep\nu^{\f{5}{12}}|\ln\nu|.
           \end{aligned}
       \end{equation}
       Then, combined with the $\w\om{a}$-estimates in Lemma \ref{lem:estimate-u1_k}, and the smallness \eqref{smallness}, we obtain
       \begin{align*}
           \big\||k|^{-\f12}e^{\ep_0\nu^{\f13}t}\mathbb F_{k;22}\big\|_{L^2_tl_k^2L_y^2}\lesssim \nu^{\f16}\ep\nu^{\f13}\cdot \ep\nu^{\f{5}{12}}|\ln\nu|\lesssim \ep^2\nu^{\f{11}{12}}|\ln\nu|.
       \end{align*}
      
  Next, we estimate the term $\mathbb F_{k;23}$, a direct calculation shows
      \begin{equation*}
          \begin{aligned}
              &\||k|^{-\f12}e^{\ep_0\nu^{\f13}t}\mathbb F_{k;23}\|_{L^2_tl_k^2L_y^2}\\
              &\lesssim \|\pa_y\w\p{*}_0\|_{L^{\infty}_tL_y^{\infty}}\Big(\sum_{(\s)\in\mathcal{I}}\||k|^{\f12}e^{\ep_0\nu^{\f13}t}\w\om{\s}_k\|_{L^2_tl_k^2L_y^2}\Big)+\|\w\om{*}_0\|_{L^{\infty}_tL_y^2}\||k|^{\f12}e^{\ep_0\nu^{\f13}t}\pa_y\p_k\|_{L^2_tl_k^2L_y^{\infty}}.
          \end{aligned}
      \end{equation*}
     Combining the $\w\om{b}$-bound in \eqref{ineq-size_om} and \eqref{est-L2L2-weight}, we get
      \begin{align}\label{omega-sigma-k}
          \sum_{(\s)\in\mathcal{I}}\||k|^{\f12}e^{\ep_0\nu^{\f13}t}\w\om{\s}_k\|_{L^2_tl_k^2L_y^2}\lesssim \ep\nu^{\f13}+\ep\nu^{\f{5}{12}}|\ln\nu|\lesssim\ep\nu^{\f13}.
      \end{align}
     Then through a similar discussion to that in \eqref{est-middle-u}, and using \eqref{omega-sigma-k}, we get
      \begin{equation*}
      \begin{aligned}
      &\||k|^{\f12}e^{\ep_0\nu^{\f13}t}\pa_y\p_k\|_{L^2_t l_k^2L^{\infty}_y}
      \lesssim \nu^{\f{1}{12}}\ep\nu^{\f13}+(\ep\nu^{\f23}|\ln\nu|)^{\f12}(\ep\nu^{\f13})^{\f12}\lesssim \ep\nu^{\f5{12}}.
      \end{aligned}
      \end{equation*}
      Therefore, combined with the Gagliardo-Nirenberg inequality on $\w{\pa_y\p}{*}_0$ and the bootstrap assumption \eqref{hyp_om_0_l} of zero mode part, we obtain
      \begin{equation*}
          \begin{aligned}
            \||k|^{-\f12}e^{\ep_0\nu^{\f13}t}\mathbb F_{k;23}\|_{L_t^2l_k^2L_y^2}\lesssim \ep\nu^{\f23}|\ln\nu|\big(\ep\nu^{\f13}+\ep\nu^{\f5{12}}\big)\lesssim \ep^2\nu|\ln\nu|.
          \end{aligned}
      \end{equation*}
 
  For the estimate of $\mathbb{F}_{k;24}$, since $\w\om{*,i}(y=\pm1)=0$, Lemma \ref{Lem-F-Y}, Lemma \ref{lem:estimate-u1_k}, and the bootstrap assumptions \eqref{hyp_h_om_re}-\eqref{hyp_om_e_l} imply that
      \begin{align*}
      &   \||k|^{-\f12}e^{\ep_0\nu^{\f13}t}\mathbb F_{k;24}\|_{L_t^2l_k^2L_y^2}
       \lesssim \nu^{\f76}\big(\||k|^{\f13}e^{\ep_0\nu^{\f13}t}\w\om{a}_k\|_{L_t^{2}[0,T_*);l_k^2L_y^{\infty})}\\
       &\quad+\||k|^{\f13}e^{\ep_0\nu^{\f13}t}\pa_y\w\p{a}_k\|_{L_t^{2}[0,T_*);l_k^2L_y^{\infty})}+\||k|^{\f13}e^{\ep_0\nu^{\f13}t}\pa_y\w\p{*,i}_k\|_{L_t^{2}[0,T_*);l_k^2L_y^{\infty})}\big)\\
       &\lesssim \nu^{\f76}\Big(\nu^{-\f16}\|\om_{\ne;\rm in}\|_{H_{x,y}^3}+\big(\||k|e^{\ep_0\nu^{\f13}t}\pa_y\w\p{*,i}_k\|_{L_t^{2}[0,T_0);l_k^2L^{2}_y)}+\||k|e^{\ep_0\nu^{\f13}t}\pa_y\w\p{*,i}_k\|_{L_t^{2}[T_0,T_*);l_k^2L^{2}_y)}\big)^{\f12}\\
       &\qquad\qquad\times\big(\|e^{\ep_0\nu^{\f13}t}\w\om{*,i}_k\|_{L_t^{2}[0,T_0);l_k^2L^{2}_y)}+\||k|^{\f12}e^{\ep_0\nu^{\f13}t}\w\om{*,i}_k\|_{L_t^{2}[T_0,T_*);l_k^2L^{2}_y)}\big)^{\f12}\Big)\\
       &\lesssim\nu^{\f76}\big(\nu^{-\f16}\ep\nu^{\f13}+(\ep\nu^{\f23}|\ln\nu|\nu^{-\f{1}{12}}\big)^{\f12}\big(\ep\nu^{\f23}|\ln\nu|\nu^{-\f14}\big)^{\f12}\lesssim\ep\nu^{\f43}.
    \end{align*}
    
    \step{3. The estimate of $\mathbb{F}_{k;25}$:} The estimate of the $\mathbb F_{k;25}$-term in \eqref{defn_F_k2} is slightly trickier. We split the $l$-summation  into two regions: $|k-l|<|l|$ and $|k-l|\ge |l|$
    \begin{align*}
        |k|^{-\f{101}{200}}\mathbb |\mathbb{F}_{k;25}|\le  |k|^{\f{99}{200}}\sum_{\substack{l\in\mathbb Z\backslash\{0,k\}\\ |k-l|<|l|}}\sum_{(\s)\in\mathcal{I}}|\pa_y\p_l^{(*,i)}|\big|\w\om{\s}_{k-l}\big|+|k|^{\f{99}{200}}\sum_{\substack{l\in\mathbb Z\backslash\{0,k\}\\ |k-l|\ge|l|}}|\pa_y\p_l^{(*,i)}\big|\w\om{\s}_{k-l}\big|.
    \end{align*}
For the first term, we applied the relation that $|k|\leq 2|l|$ to swap the $|k|$ by $|l|$. For the second term, we use the relation $|k|\leq 2|k-l|$, and the fact $1=|l|^{-\f1{200}} |l|^{\f1{200}}\le |l|^{-\f1{200}}|k-l|^{\f1{200}}$ to insert the factor $|l|^{-\f{1}{200}}|k-l|^{\f{1}{200}}$ to improve the summability. Then, Young's convolution inequality and Cauchy-Schwarz inequality ensure the bound
\begin{align}
    \nonumber &\||k|^{-\f{101}{200}}e^{\ep_0\nu^{\f13}t}\mathbb F_{k;25}\|_{L^2_{t}l_k^2L_y^2} 
     \\\nonumber &\lesssim \sum_{(\sigma)\in\mathcal I}\lf(\int_{T_0}^{T_*}e^{2\ep_0\nu^{\f13}t}\sum_{k\in\mathbb{Z}}\lf(\sum_{l\in \mathbb{Z}}\||l|^\f{99}{200}\pa_y\w\psi{*,i}_l\mathbf{1}_{l\neq 0}\|_{L_y^\infty}\|\w\om{\sigma}_{k-l}\mathbf{1}_{k-l\neq 0}\|_{L_y^2}\rg)^2dt\rg)^\f12\\
     \nonumber&\quad+\sum_{(\sigma)\in\mathcal I}\lf(\int_{T_0}^{T_*}e^{2\ep_0\nu^{\f13}t}\sum_{k\in\mathbb{Z}}\lf(\sum_{l\in \mathbb{Z}}\||l|^{-\f{1}{200}}\pa_y\w\psi{*,i}_l\mathbf{1}_{l\neq 0}\|_{L_y^\infty}\||k-l|^\f12\w\om{\sigma}_{k-l}\mathbf{1}_{k-l\neq 0}\|_{L_y^2}\rg)^2dt\rg)^\f12\\
     \nonumber&\lesssim \sum_{(\sigma)\in\mathcal I}  \big\||k|^{\f{99}{200}}e^{\ep_0\nu^{\f13}t}\pa_y\w\p{*,i}_k\big\|_{L_t^{\infty}l_k^2L_y^{\infty}}\big\|e^{\ep_0\nu^{\f13}t}\w\om{\sigma}_{k}\mathbf{1}_{k\neq 0}\Big\|_{L^2_{t}l_k^1L^2_{y}}\\
     &\quad +\sum_{(\sigma)\in\mathcal I}\big\||k|^{-\f{1}{200}}e^{\ep_0\nu^{\f13}t}\pa_y\w\p{*,i}_k\mathbf{1}_{k\neq 0}\big\|_{L_t^{\infty}l_k^1L^{\infty}_y} \Big\||k|^{\f12}e^{\ep_0\nu^{\f13}t}\w\om{\sigma}_{k}\mathbf{1}_{k\neq 0}\Big\|_{L^2_tl_k^2L^2_{y}}=:T_1+T_2.\label{lng_F_pf_3}
\end{align}
We estimate each term in the decomposition in the sequel. 

To begin the estimate for the $T_1$-term, we focus on the vorticity component and observe the following relations for all $p,q\in [1,\infty]$ and all $\varsigma>0$ \begin{align}
\label{lng_F_pf_4}\|f_k\mathbf{1}_{k\neq 0}\|_{L_t^pl_k^1L_y^q}\leq C_{\varsigma}\||k|^{\f{1}{2}+\varsigma} f_k\|_{L^p_tl_k^2L_y^q},\quad \||k|^{\f58}f_k\|_{L_t^2l_k^2L_y^2}\le \||k|f_k\|_{L_t^2l_k^2L_y^2}^{\f14}\||k|^{\f12}f_k\|_{L_t^2l_k^2L_y^2}^{\f34}
\end{align}through the lens of Cauchy-Schwarz/H\"older inequality. 
    Next let us consider vorticity corresponding to each element in $\mathcal{I}=\{(b), (\e), ({\rm re}), (*,b)\}$ with the aid of \eqref{lng_F_pf_4}. For $\w\om{b}$ and $\w\om{*,b}$, we expand the weight from $|k|^{\f58}$ to $|k|$ and invoke the estimate \eqref{ineq-size_om} in the toolkit. For the error  $\w\om{\e}$ and resonant $\w\om{\rm re}$, we use the interpolation \eqref{lng_F_pf_4}, the estimate \eqref{ineq-size_om} in the toolkit, and the bootstrap assumptions \eqref{hyp_h_om_re}, \eqref{hyp_om_e_l}, we get 
      \begin{align}\label{lng_F_pf_5}\begin{split}
     &  \sum_{(\sigma)\in\mathcal I}\big\|e^{\ep_0\nu^{\f13}t}\w\om{\sigma}_{k}\mathbf{1}_{k\neq 0}\Big\|_{L^2_{t}l_k^1L^2_{y}}
       \lesssim\sum_{(\sigma)\in\mathcal I}\big\||k|^{\f{5}{8}}e^{\ep_0\nu^{\f13}t}\w\om{\sigma}_{k}\big\|_{L^2_tl_k^2L_y^2}\\
      &\lesssim \lf(\big\||k|e^{\ep_0\nu^{\f13}t}\w\om{b}_{k}\big\|_{L^2_tl_k^2L_y^2}+\big\||k|e^{\ep_0\nu^{\f13}t}\w\om{*,b}_{k}\big\|_{L^2_tl_k^2L_y^2}\rg)\\
      &\qquad +\sum_{\sigma=\e,\rm re}\big\||k|e^{\ep_0\nu^{\f13}t}\w\om{\sigma}_{k}\big\|_{L^2_tl_k^2L_y^2}^\f14\big\||k|^{\f{1}{2}}e^{\ep_0\nu^{\f13}t}\w\om{\sigma}_{k}\big\|_{L^2_tl_k^2L_y^2}^\f34\\
      &\lesssim \ep\nu^{\f13}+(\ep\nu^\f16|\ln\nu|)^\f14(\ep\nu^{\f{5}{12}}|\ln\nu|)^\f34\lesssim\ep\nu^{\f13}.
      \end{split}
      \end{align}
Next we consider the stream function bound in \eqref{lng_F_pf_3}. Applying  Gagliardo-Nirenberg and elliptic estimate $\||k|\pa_{y}\w\psi{*,i}_k\|_{L^2_y}+\|\pa_{y}^2\w\psi{*,i}_k\|_{L^2_y}\lesssim \|\w\om{*,i}_k\|_{L^2_y}$, and bootstrap assumptions \eqref{hyp_h_om_re} and \eqref{hyp_om_e_l}, we get
\begin{align}
\nonumber&\big\||k|^{\f12}e^{\ep_0\nu^{\f13}t}\pa_y\w{\p}{*,i}_k\big\|_{L_t^{\infty}l_k^2L_y^{\infty}}\\
\nonumber&\lesssim\sup_{t}e^{\ep_0\nu^{\f13}t}\Big(\sum_{k\in\mathbb Z\backslash\{0\}}|k|\|\pa_y\w{\p}{*,i}_k\|_{L_y^2}\|\pa_y^2\w{\p}{*,i}_k\|_{L_y^2}+\sum_{k\in\mathbb Z\backslash\{0\}}|k|\|\pa_y\w{\p}{*,i}_k\|_{L_y^2}^2\Big)^\f12\\
 \nonumber   &\lesssim\big\|e^{\ep_0\nu^{\f13}t}|k|\pa_y\w{\p}{*,i}_k\big\|_{L_t^{\infty}l_k^2L_y^{2}}^{\f12}\big\|e^{\ep_0\nu^{\f13}t}\w{\om}{*,i}_k\big\|_{L_t^{\infty}l_k^2L_y^{2}}^{\f12}+\big\|e^{\ep_0\nu^{\f13}t}|k|^\f12\pa_y\w{\p}{*,i}_k\big\|_{L_t^{\infty}l_k^2L_y^{2}}\\
\label{est-infty-infty}    
&\lesssim \big\|e^{\ep_0\nu^{\f13}t}\w{\om}{*,i}_k\big\|_{L_t^{\infty}l_k^2L_y^{2}}
    \lesssim \ep \nu^{\f23}|\ln\nu|.
\end{align}

By combing the vorticity estimate and the stream function estimate, we obtain that 
\begin{align*}
    T_1\lesssim \ep \nu^{\f23}|\ln\nu|\cdot\ep\nu^{\f13}\lesssim \ep^2\nu|\ln\nu|,
\end{align*}
which is consistent with the result \eqref{F_k_2_est}. 

Next we consider the $T_2$-term in \eqref{lng_F_pf_3}. By the relation $\eqref{lng_F_pf_4}\  (p=q=\infty,\varsigma=\f1{200})$ and the introduction of $|k|^{-\f1{200}}$-weight, we can safely bound the stream function factor by the expression \eqref{est-infty-infty} without losing $k$-regularity, i.e.,
\begin{align*}
    \big\||k|^{-\f{1}{200}}e^{\ep_0\nu^{\f13}t}\pa_y\w\p{*,i}_k\mathbf{1}_{k\neq 0}\big\|_{L_t^{\infty}l_k^1L^{\infty}_y} \lesssim
\big\||k|^{\f12}e^{\ep_0\nu^{\f13}t}&\pa_y\w{\p}{*,i}_k\big\|_{L_t^{\infty}l_k^2L_y^{\infty}}\lesssim \ep \nu^{\f23}|\ln\nu|.
\end{align*}
Since the vorticity factor is controlled by $\sum_{(\sigma)\in\mathcal I}\big\||k|^{\f{5}{8}}e^{\ep_0\nu^{\f13}t}\w\om{\sigma}_{k}\big\|_{L^2_tl_k^2L_y^2}$, which automatically satisfies the bound \eqref{lng_F_pf_5}, we have that $T_2\lesssim \ep^2\nu|\ln\nu|$. Hence, 
      \begin{equation*}
          \begin{aligned}
              &\||k|^{-\f{101}{200}}e^{\ep_0\nu^{\f13}t}\mathbb F_{k;25}\|_{L_t^2l_k^2L_y^2}\lesssim \ep^2\nu|\ln\nu|.
          \end{aligned}
      \end{equation*}This is consistent with the final result  \eqref{F_k_2_est}. Combing all the bounds developed thus far, we conclude the proof of  \eqref{F_k_2_est}.\end{proof}
\begin{lemma}\label{Lem-Fk3}
Under the bootstrap assumption \eqref{Btstrp_Hyp}, the forcing $\mathbb{F}_{k;3}$ \eqref{defn_F_k3} satisfies the following estimate
\begin{align}
 \|e^{\ep_0\nu^{\f13}t}\mathbb F_{k;3}\|_{L^2_t([T_0,T_*);l_k^2L_y^2)}\lesssim\ep^2\nu^{\f54}|\ln\nu|^2.\label{F_k_3_est}
\end{align}
\end{lemma}
\begin{proof}
Recalling the decomposition of $\om$ and the definition of $\mathbb{F}_{k;3}$ in \eqref{defn_F_k3}, for the term $\mathbb F_{k;31}$, we first use Gagliardo-Nirenberg inequality and elementary estimate $\|k\p_k\|_{L^{\infty}_y}\lesssim |k|\|\p_k\|_{L_y^2}^{\f12}\|\pa_y\p_k\|_{L_y^2}^{\f12}\lesssim |k|^{\f12}\|(ik,\pa_y)\p_k\|_{L_y^2}$, the $\w\om{a}$-estimates in Lemma \ref{lem:estimate-u1_k}, the velocity estimate in \eqref{ineq-size}, and the smallness \eqref{smallness} to give the following velocity contribution
\begin{equation}\label{est-phi-2inf}
\begin{aligned}
    &\|e^{\ep_0\nu^{\f13}t}k\p_k\|_{L_t^{2}l_k^2L_y^{\infty}}\le \|e^{\ep_0\nu^{\f13}t}k\p_k^{(a)}\|_{L_t^{2}l_k^2L_y^{\infty}}+\sum_{(\s)\in\mathcal{I}}\|e^{\ep_0\nu^{\f13}t}|k|(ik,\pa_y)\p_k^{(\s)}\|_{L_t^{2}l_k^2L_y^{2}}\\
    &
      \lesssim \nu^{\f{1}{4}}\|\om_{\ne;\rm in}\|_{H_{x,y}^3}+\ep\nu^{\f23}|\ln\nu|\lesssim\ep\nu^{\f7{12}}+\ep\nu^{\f23}|\ln\nu|.
\end{aligned}
\end{equation}
Therefore, combined \eqref{est-phi-2inf} and the bootstrap assumption \eqref{hyp_om_0_l} of zero mode, we obtain
\begin{align}\label{eq-pf-F31}
   & \|e^{\ep_0\nu^{\f13}t}\mathbb F_{k;31}\|_{L^2_tl_k^2L_y^2}\le\|e^{\ep_0\nu^{\f13}t}k\p_k\|_{L^{2}_tl_k^2L_y^{\infty}}\|\w\om{*}_0\|_{L^{\infty}_tL_y^2} \lesssim \ep^2\nu^{\f54}|\ln\nu|+\ep^2\nu^{\f43}|\ln\nu|^2.
   \end{align}
   Next, we estimate the $\mathbb F_{k;32}$-contribution. We apply Young's convolution inequality to get
   \begin{align*}
&\|e^{\ep_0\nu^{\f13}t}\mathbb F_{k;32}\|_{L^2_tl_k^2L_y^2}
    \lesssim \big\|e^{\ep_0\nu^{\f13}t}k\p_k\big\|_{L^2_tl_k^1L^{\infty}_y}\big\|e^{\ep_0\nu^{\f13}t}\big(\w\om{\rm e}_{k}+\w\om{\rm re}_{k}\big)\big\|_{L^{\infty}_tl_k^2L^{2}_y}\\
    &\quad + \big\|e^{\ep_0\nu^{\f13}t}(1-|y|)^{-1}k\p_k\big\|_{L^2_tl_k^2L_y^{\infty}}\big\|e^{\ep_0\nu^{\f13}t}(1-|y|)(\w\om{b}_{k}+\w\om{*, b}_{k})\big\|_{L_t^{\infty}l_k^1L^{2}_y }=:\mathcal{F}_1\mathcal{F}_2+\mathcal{F}_3\mathcal{F}_4.
    \end{align*}
    Then, using \eqref{eq-weig}, taking $p=2$, $q=\infty$, $\varsigma=\f16$ in \eqref{lng_F_pf_4}, together with the $\w\om{a}$-estimates in Lemma \ref{lem:estimate-u1_k}, the smallness assumption \eqref{smallness}, and the bootstrap assumptions \eqref{hyp_h_om_re}-\eqref{hyp_om_e_l}, it yields that
    \begin{align*}
     \mathcal{F}_1&=\big\|e^{\ep_0\nu^{\f13}t}k\p_k\big\|_{L^2_tl_k^1L^{\infty}_y}\lesssim \big\||k|^{\f23}e^{\ep_0\nu^{\f13}t}k\p_k\big\|_{L^2_tl_k^2L_y^{\infty}}\\
       &\lesssim \nu^{\f14}\|\om_{\ne;\rm in}\|_{H_{x,y}^3}+\||k|e^{\ep_0\nu^{\f13}t}\na_k\p_k^{(\rm e)}\|_{L^2_t l_k^2L^{2}_y}^{\f23}\||k|^{\f12}e^{\ep_0\nu^{\f13}t}\om_k^{(\rm e)}\|_{L_t^2 l_k^2L^{2}_y}^{\f13}\\
        &\quad+\||k|e^{\ep_0\nu^{\f13}t}\na_k\p_k^{(\rm re)}\|_{L^2_t l_k^2L^{2}_y}^{\f23}\||k|^{\f12}e^{\ep_0\nu^{\f13}t}\om_k^{(\rm re)}\|_{L^2_t l_k^2L^{2}_y}^{\f13}+\||k|^{\f76}e^{\ep_0\nu^{\f13}t}(ik,\pa_y)\p_k^{(*,b)}\|_{L^2_tl_k^2 L^{2}_y)}\\
        &\lesssim \ep\nu^{\f{7}{12}}+(\ep\nu^{\f23}|\ln\nu|)^{\f23}(\ep\nu^{\f23}|\ln\nu|\nu^{-\f14})^{\f13}+\ep\nu|\ln\nu|\lesssim \ep\nu^{\f7{12}}|\ln\nu|.
    \end{align*}
     Moreover, combined with Lemma \ref{lem:estimate-u1_k}, the velocity estimate in \eqref{ineq-size} and the vorticity estimate in \eqref{ineq-size_om}, we get
      \begin{align*}
        &\mathcal{F}_3\lesssim  \||k|\pa_y\p_k^{(a)}\|_{L^2_tl_k^2L_y^{\infty}}+\sum_{(\s)\in\mathcal{I}}\|e^{\ep_0\nu^{\f13}t}k\pa_y\p_k^{(\s)}\|_{L_t^{2}l_k^2L_y^{2}}^{\f12}\|e^{\ep_0\nu^{\f13}t}k\om_k^{(\s)}\|_{L_t^{2}l_k^2L_y^{2}}^{\f12}\\
        &\lesssim \nu^{\f1{12}}\ep\nu^{\f13}+(\ep\nu^{\f23}|\ln\nu|)^{\f12}(\ep\nu^{\f16}|\ln\nu|)^{\f12}\lesssim \ep\nu^{\f5{12}}|\ln\nu|.
    \end{align*}
For $\mathcal{F}_4$, taking $p=\infty$, $q=2$, $\varsigma=\f12$ in \eqref{lng_F_pf_4}, using \eqref{est-*b}, it yields 
\begin{align*}
    \mathcal{F}_4\lesssim \big\||k|e^{\ep_0\nu^{\f13}t}(1-|y|)(\w\om{b}_{k}+\w\om{*, b}_{k})\big\|_{L_t^{\infty}l_k^2L^{2}_y}\lesssim \ep\nu^{\f56}.
\end{align*}
Finally, with an $\mathcal{F}_2$-bound following directly from the bootstrap assumptions \eqref{hyp_h_om_re}, \eqref{hyp_om_e_l}, we get
\begin{align}\label{eq-pf-F32}
    \|e^{\ep_0\nu^{\f13}t}\mathbb F_{k;32}\|_{L^2_t([T_0,T_*);l_k^2L_y^2)}
    \lesssim \ep\nu^{\f{7}{12}}|\ln\nu|\cdot\ep\nu^{\f23}|\ln\nu|+\ep\nu^{\f{5}{12}}|\ln\nu|\cdot\ep\nu^{\f56}\lesssim\ep^2\nu^{\f54}|\ln\nu|^2.
\end{align}
Therefore, combining \eqref{eq-pf-F31} and \eqref{eq-pf-F32}, we conclude the proof of \eqref{F_k_3_est}.
\end{proof}
Based on the above works, we can reach Proposition \ref{prop-e}.
\begin{proof}[Proof of Proposition \ref{prop-e}]At first, Applying Proposition \ref{pro:shortT}, for $t\in [0,T_0]$, we obtain
\begin{align*}
    \|\w\om{*,i}_k\|_{L^{\infty}_tl_k^2L_y^2}\le\ep\nu^{\f23}|\ln\nu|.
\end{align*}
Then, by Corollary \ref{cor-e-E}, the $\mathbb F_{k;1}$-estimate \eqref{F_k_1_est}, the $\mathbb F_{k;2}$-estimate \eqref{F_k_2_est}, the $\mathbb F_{k;2}$-estimate \eqref{F_k_3_est}, after taking $\ep$ and $\nu$ small enough, we obtain
\begin{align*}
   &\|e^{\ep_0\nu^{\f13}t}\w\om{\rm e}_k\|_{L^{\infty}_tl_k^2L_y^2}+\||k|e^{\ep_0\nu^{\f13}t}\na_k\w\p{\rm e}_k\|_{L^2_tl_k^2L_y^2}+\nu^{\f14}\||k|^{\f12}e^{\ep_0\nu^{\f13}t}\w\om{\rm e}_k\|_{L_t^{2}l_k^2L_y^2}+\nu\|e^{\ep_0\nu^{\f13}t}\na_k\w\om{\rm e}_k \|_{L^2_tl_k^2L_y^2}\\
        &\lesssim \|\w\om{*,i}_k(T_0)\|_{l_k^2L_{y}^2}+\|e^{\ep_0\nu^{\f13}t}\mathbb F_{k;1}\|_{L^1_tl_k^2L^2_y}+\nu^{-\f12}\|e^{\ep_0\nu^{\f13}t}\mathbb F_{k;3}\|_{L^2_tl_k^2L_y^2}\\
        &\quad+\nu^{-\f14} \big(\sum_{l=1}^4\||k|^{-\f12}e^{\ep_0\nu^{\f13}t}\mathbb F_{k;2l}\|_{L^2_tl_k^2L_y^2}\big)+ \nu^{-\f14}\nu^{-\f{1}{400}}\||k|^{-\f12}|k|^{-\f{1}{200}}e^{\ep_0\nu^{\f13}t}\mathbb F_{k;25}\|_{L^2_tl_k^2L_y^2}\\
        &\le \ep\nu^{\f23}|\ln\nu|+C\big(\ep^2\nu^{\f23}|\ln\nu|+\ep\nu^{\f{13}{12}}+\ep^2\nu^{\f54}|\ln\nu|^2\big)\le 2\ep\nu^{\f23}|\ln\nu|.
\end{align*}
\end{proof}
\subsection{Long Time Regime: Estimates of the Zero Mode} \label{sec-longT_z}

In this section, we focus on the zero-mode evolution. By incorporating the singular integral operator $\wt{\mathfrak{J}}_0$ \eqref{J0} and the associated energy functional $\mathbb E_0$ \eqref{energy-F0}, we prove the $\w \omega{\rm re}_0$-estimate in Proposition \ref{prop-wre0}. Then we apply energy-type argument to prove the $\w\om\e_0$-estimate in Proposition \ref{prop-w*0}.

\begin{proof}[Proof of Proposition \ref{prop-wre0}] We organize the proof in steps. Moreover, to simplify the notation,  unless otherwise stated, all the $L_t^pX$-spaces in this proof are referring to $L^p_t([T_0=\nu^{-\f16},T_*);X)$, and all $l_k^p$-norms are taken over $k\in\mathbb Z\backslash\{0\}$. 

\step{1: Setup.} 
We start by considering a modified elliptic problem:
\begin{align*}
    \D^L_kh_k:=-k^2h_k+(\pa_y-ikt)^2h_k= f_k,\quad h_k\big|_{y=\pm 1}=0. 
\end{align*} By recalling  the Green's function $G_k(y,y')$ for the Dirichlet $\Delta_k$ in \eqref{eq-Green}, we can solve the problem as follows
    \begin{align}\label{dfn_G_k0}
      & h_k(t,y)= (\D^L_k)^{-1}f_k(t,y):=\int_{-1}^1G_k^{(0)}(t,y,y')f_k(t,y')dy',\quad G_k^{(0)}(t,y,y'):=e^{ikt(y-y')}G_k(y,y').
      \end{align}We also apply the notations $\na_k^L:=(ik,\pa_y-ikt), \, \Delta_k^Lg=\na_k^L\cdot\na_k^Lg$ in the sequel. 
We consider the energy functional associated with the singular integral operator $\wt{\mathfrak{J}}_0$ \eqref{J0} on $[T_0,T_*)$,
\begin{equation}
 \begin{aligned}\label{energy-F0}
     \mathbb E_0=\f12\Big(\|\w\om{\rm re}_0\|_{L^2}^2+\al^{-1}\big(\w u{\rm re}_0(1)\big)^2+\al^{-1}\big(\w u{\rm re}_0(-1)\big)^2\Big)+c_0{\rm Re}\langle \wt{\mathfrak{J}}_0[\w \om{\rm re}_0],\om_0^{\rm (re)}\rangle     =:\mathbb{E}_{0;\rm Main}+c_0\mathbb{E}_{0;\mf J}.
 \end{aligned}
 \end{equation}
Here, the universal constant $c_0$ is chosen in Step \# 4 to guarantee that the energy is bounded in time. 
When computing the $\mathbb{E}_0$-energy dissipation relation, one  can extract the following coercive \textit{CK}-term:
\begin{align}
\mathcal{CK}_0(t):=\sum_{k\in\mathbb Z\backslash\{0\}}|k|^{-1}\|\na_k^L(\D_k^L)^{-1}\w\omega{\rm re}_0\|_{L_y^2}^2. \label{CK}
\end{align}

One can recall the equation \eqref{eq-0re} and derive the $  \mathbb E_0$-energy dissipation relation. Through integration by parts and the vanishing of the source $\mathbb F_0$ \eqref{eq-0re} on the boundary $ y=\pm1$, we get
 \begin{align}\label{dE_0Main}\begin{split}
    \frac{d}{dt}\mathbb{E}_{0,\rm Main}&= \f12\f{d}{dt}\Big(\|\w\om{\rm re}_0\|_{L^2}^2+\al^{-1}\big(\w u{\rm re}_0(1)\big)^2+\al^{-1}\big(\w u{\rm re}_0(-1)\big)^2\Big)=\langle \pa_tu_0^{(\rm re)},-\pa_y^2u_0^{\rm (re)}\rangle\\
    &=-\nu\|\pa_y^2u_0^{\rm (re)}\|_{L^2}^2+{\rm Re}\langle \pa_y\mathbb F_0,\pa_yu_0^{\rm (re)}\rangle,\\
 \pa_y\mathbb F_0&=    \sum_{k\in\mathbb Z\backslash\{0\}}ik\big(\pa_y\p_k-\pa_y\p_k^{(a)}\big)\om_{-k}^{(a)}+\sum_{k\in\mathbb Z\backslash\{0\}}ik\big(\p_k-\p_k^{(a)}\big)\pa_y\om_{-k}^{(a)}=:\mathbb F_{0;1}+\mathbb F_{0;2}.
 \end{split}
\end{align}
Next we derive the energy dissipation relation for $\mathbb{E}_{0;\mf J}$. 
 Moreover, after taking $y$-derivative on the equation \eqref{eq-0re}, it follows that
$    \pa_t\w\om{\rm re}_0-\nu\pa_y^2\w\om{\rm re}_0=-\pa_y\mathbb F_0.$  
 Then, by the symmetry property of $\wt{\mathfrak{J}}_0$ \eqref{eq-symm}, we can derive
  
      \begin{equation}\label{dE_0J}   \begin{aligned} 
    \frac{d}{dt}\mathbb{E}_{0,\mf J}
        &={\rm Re}\big\langle[\pa_t,\wt{\mathfrak{J}}_0][\w\om{\rm re}_0] ,\om_0^{\rm (re)}\big\rangle-2{\rm Re}\langle \wt{\mathfrak{J}}_0[\pa_y\mathbb F_0],\om_0^{\rm (re)}\rangle+2{\rm Re}\langle \wt{\mathfrak{J}}_0[\w \om{\rm re}_0],\nu\pa_y^2\om_0^{\rm (re)}\rangle=:\sum_{j=1}^3T_j.
    \end{aligned}
      \end{equation} The main part of the proof is to develop estimates for all the terms in the decompositions above. This concludes Step \# 1.
      
\step{2: Estimate of  $\mathbb{E}_{0;\rm Main}$.} 
It is sufficient to estimate the source term on the right hand side of \eqref{dE_0Main}. Since $\mathbb F_{0;2}(y=\pm1)=0$, we use the relation $\pa_yu_0^{\rm (re)}=-\om_0^{\rm (re)}$ and integrate by parts and obtain
 \begin{equation}\label{E0_pf0}\begin{aligned}
  &   \frac{d}{dt}\mathbb E_{0,\rm Main}+\nu\|\pa_y^2u_0^{\rm (re)}\|_{L^2}^2={\rm Re}\langle \pa_y\mathbb F_0,\pa_yu_0^{\rm (re)}\rangle\\
     &={\rm Re}\langle  \mathbb F_{0;1},-\om_0^{\rm (re)}\rangle+{\rm Re}\Big\langle \sum_{k\in\mathbb Z\backslash\{0\}}ik\big(\p_k-\p_k^{(a)}\big)\pa_y\om_{-k}^{(a)},-\D_k^L(\D_k^L)^{-1}\om_0^{\rm (re)}\Big\rangle\\
     &\le \|\mathbb F_{0;1}\|_{L_y^2}\|\om_0^{\rm (re)}\|_{L_y^2}+\| |k|\na_k^L\big(ik(\p_k-\p_k^{(a)})\pa_y\om_{-k}^{(a)}\big)\|_{l_k^2L_y^2}\|\mathbf{1}_{k\ne0}|k|^{-1}\na_k^L(\D_k^L)^{-1}\om_0^{\rm (re)}\|_{l_k^2L_y^2}.
 \end{aligned}\end{equation} 
 Next, we estimate various contributors in this differential relation. By recalling the decomposition of $\om$ in \eqref{lol-dcmp},  we use the $\w\om{a}$-estimates in Lemma \ref{lem:estimate-u1_k} (coupled with the smallness \eqref{smallness}) and the velocity estimate in \eqref{ineq-size} to obtain
\begin{equation}\label{E0_pf1}\begin{aligned}
   & \| \mathbb F_{0;1}\|_{L^1_tL^2_y}+\| \mathbb F_{0;2}\|_{L^2_tL_y^2}\le \Big\||k|\big(\pa_y\p_k-\pa_y\p_k^{(a)}\big)\om_{-k}^{(a)}\Big\|_{L^1_tl_k^1L^2_y}+\Big\||k|\big(\p_k-\p_k^{(a)}\big)\pa_y\om_{-k}^{(a)}\Big\|_{L^2_tl_k^1L_y^2}\\
    & \le \Big(\sum_{(\s)\in\mathcal{I}}\big\||k|\pa_y\p_k^{(\s)}\|_{L^2_tl_k^2L_y^2}\Big)\|\om_{-k}^{(a)}\|_{L^2_tl_k^2L_y^{\infty}}+\Big(\sum_{(\s)\in\mathcal{I}}\||k|\p_k^{(\s)}\|_{L_t^2l_k^2L_y^2}\Big)\|\pa_y\om_{-k}^{(a)}\|_{L^{\infty}_tl_k^2L_y^{\infty}}\\
    &\lesssim \ep^2\nu^{\f56}|\ln\nu|+\ep^2\nu^{\f23}|\ln\nu|\lesssim\ep^2\nu^{\f23}|\ln\nu|.
   \end{aligned}\end{equation}
   Furthermore, the same set of estimates yields the stream function bound
   \begin{align}
      &\Big\|\na_k(\p_k-\p_k^{(a)})|k|^2\na_k^L\pa_y\w\om{a}_{-k}\Big\|_{L_t^2l_k^2L_y^2}^2
    +\Big\||k|^2\pa_y(\p_k-\w\p{a}_k)\pa_y\w\om{a}_{-k}\Big\|_{L_t^2l_k^2L_y^2}^2\notag\\
    &\lesssim \Big\|\na_k(\p_k-\p_k^{(a)})\Big\|_{L_t^2l_k^2L_y^2}^2\Big\||k|^2\na_k^L\pa_y\w\om{a}_{-k}\Big\|_{L_t^{\infty}l^{\infty}_kL_y^{\infty}}^2\label{E0_pf2}
   \\
    &\quad  +\Big\||\pa_y(\p_k-\w\p{a}_k)\Big\|_{L_t^2l_k^2L_y^2}^2\Big\||k|^2\pa_y\w\om{a}_{-k}\Big\|_{L_t^{\infty}l^{\infty}_kL_y^{\infty}}^2\lesssim \ep^4\nu^{\f43}|\ln\nu|^2\notag.
    \end{align}
Recalling the definition of $\mathcal{CK}_0$ \eqref{CK}, the initial condition $\w\omega{\rm re}_0\big|_{t=T_0}\equiv 0$ and the bounds derived above, we integrate in time to obtain that 
\begin{align}\label{E0_pf3}
\sup_{t\in[T_0,T_\ast)}\mathbb{E}_{0,\rm Main}(t)+\nu\|\pa_y\w\omega{\rm re}_0\|_{L_t^2L_y^2}^2\leq C_1 \ep^2\nu^{\f23}|\ln\nu|\|\w\omega{\rm re}_0\|_{L_t^\infty L_y^2}+C_1B\ep^4\nu^{\f43}|\ln\nu|^2+\frac{1}{B}\int_{T_0}^{T_\ast}\mathcal{CK}_0dt.
\end{align}
Here, the parameter $B$ is any positive number. This concludes Step \# 2.
 
    \step{3: Estimate of   $\mathbb{E}_{0;\mf J}$.} We estimate each term in \eqref{dE_0J}. 
    
The main gain by introducing the $\wt{\mathfrak{J}}_0$ is that we gain a coercive term. By setting $f_0=\w\om{\rm re}_0$ in Lemma \ref{Lem-SIO}, we get
    \begin{align*}
       T_1=&{\rm Re}\big\langle[\pa_t,\wt{\mathfrak{J}}_0][\w\om{\rm re}_0] ,\om_0^{\rm (re)}\big\rangle=-\f12\|\mathbf{1}_{k\ne0}|k|^{-\f12}\na_k^L(\D_k^L)^{-1}\w\om{\rm re}_0\|_{l_k^2L_y^2}^2=-\f12\mathcal{ CK}_0.
    \end{align*}
 Next, for the treatment of the $T_2$--$T_3$ term, we make the following observation: for two regular functions $  f, g$, we have the identity for $\wt {\mf J}_0$ \eqref{J0} and $ {\mf J}_l$ with $\mathcal{U}=y$ ($l\neq0$):
 
\begin{align}
\nonumber\Re\br{\wt{\mf J}_0[ f],g}_{L^2}
=&\sum_{l\in\mathbb Z\backslash\{0\}}\frac{1}{|l|^3}\Re\int  P.V.\int \frac {lG_l(y,y')}{2i(y-y')} e^{-ilty'}f(y') dy'\overline{e^{-ilty}g(y)} dy\\
\label{E0_pf3.5}=&\sum_{l\in\mathbb Z\backslash\{0\}}\frac{1}{|l|^3}\Re\br{{\mf{J}}_l[e^{-ilty}f] ,e^{-ilty}g}_{L^2}. 
\end{align}
Since $G_k(y=\pm1,y')=0$ for any $y'\in[-1,1]$, and $\mathbb F_0(y=\pm1)= \mathbb F_{0;2}(y=\pm1)=0$, we can apply integration by parts. Moreover, we recall the relation $[\pa_y, \mf J_k]=\mf H_k$ from Lemma \ref{comm-bound} and obtain
 \begin{align*}
     &T_2=-2\Re\sum_{l\in\mathbb Z\backslash\{0\}}|l|^{-3}\Big\langle \mathfrak{J}_l[e^{-ilty} \mathbb F_{0;1}],e^{-ilty}\w\om{\rm re}_0\Big\rangle\\
     &+2\Re\sum_{k\in\mathbb Z\backslash\{0\}}\sum_{l\in\mathbb Z\backslash\{0\}}|l|^{-3}\Big\langle ik\mathfrak{J}_l[ik(\p_k-\p_k^{(a)})\pa_y\w\om{a}_{-k}e^{-ilty}], ike^{-ilty}(\D_k^L)^{-1}\w\om{\rm re}_0\Big\rangle\\
 &+2\Re\sum_{k\in\mathbb Z\backslash\{0\}}\sum_{l\in\mathbb Z\backslash\{0\}}|l|^{-3}\Big\langle \mathfrak{J}_l\lf[ik(\p_k-\p_k^{(a)})(\pa_y-ikt)\pa_y\w\om{a}_{-k}e^{-ilty}+ik\pa_y(\p_k-\w\p{a}_k)\pa_y\w\om{a}_{-k}e^{-ilty}\rg]\\
     &\hspace{3cm}+\mathfrak{H}_l\lf[ik((\p_k-\w\p{a}_k)\pa_y\w\om{a}_{-k}e^{-ilty})\rg], e^{-ilty}(\pa_y-ikt)(\D_k^L)^{-1}\w\om{\rm re}_0\Big\rangle.
 \end{align*}
Here we change the summation index of $\wt{\mathfrak{J}}_0$ into $l$ and take $\mathcal{U}(t,y)=y$ in $\mathfrak{J}_l$ and $\mathfrak{H}_l$. 
Therefore, using the estimates of $\mathfrak{J}_l$, $\mathfrak{H}_l$ in Corollary \ref{corol-bound} and Lemma \ref{comm-bound}, we can derive that
\begin{align*}
  |T_2|&\le2\sum_{l\in\mathbb Z\backslash\{0\}}|l|^{-3}\big\| \mathfrak{J}_l\big\|_{L^2\to L^2}\|e^{-ilty} \mathbb F_{0;1}\|_{L^2}\|e^{-ilty}\w\om{\rm re}_0\|_{L^2}\\
    &\quad +2\sum_{k\in\mathbb Z\backslash\{0\}}\sum_{l\in\mathbb Z\backslash\{0\}}|l|^{-3}\Big(\big\|\mathfrak{J}_l\big\|_{L^2\to L^2}\Big\||k|^2(\p_k-\p_k^{(a)})\na_k^L\pa_y\w\om{a}_{-k}\Big\|_{L_y^2}\\
     &\hspace{1cm}+\big\|\mathfrak{H}_l\big\|_{L^2\to L^2}\Big\||k|^2(\p_k-\w\p{a}_k)\pa_y\w\om{a}_{-k}\Big\|_{L_y^2}+\big\|\mathfrak{J}_l\big\|_{L^2\to L^2}\Big\||k|^2\pa_y(\p_k-\w\p{a}_k)\pa_y\w\om{a}_{-k}\Big\|_{L_y^2}\Big)\\
     &\hspace{1cm}\times\Big\||k|^{-1}\na_k^L(\D_k^L)^{-1}\w\om{\rm re}_0\Big\|_{L^2_y}\\
     &\le C \|\pa_{y}\mathbb F_{0;1}\|_{L^2_y}\|\w\om{\rm re}_0\|_{L^2_y}+C\Big\||k|^2(\p_k-\p_k^{(a)})\na_k^L\pa_y\w\om{a}_{-k}\Big\|_{l_k^2L_y^2}^2
    \\
    &\quad+C\Big\||k|^2\pa_y(\p_k-\w\p{a}_k)\pa_y\w\om{a}_{-k}\Big\|_{l_k^2L_y^2}^2+\f18\Big\||k|^{-1}\na_k^L(\D_k^L)^{-1}\w\om{\rm re}_0\Big\|_{l_k^2L_y^2}^2.
\end{align*}
Finally, we  estimate $T_3$. First observe the vanishing $\mf J_k[g_k]\big|_{y=\pm 1}=0$ as a consequence of  $G_k(y,\pm1)=G_k(\pm1,y')=0$ and apply integration by parts to obtain
\begin{align*}
T_3&=2\nu\sum_{k\in \mathbb{Z}\backslash \{0\}}\frac{1}{|k|^3}\Re\br{\mf J_k[e^{-ikty}\w \om{\rm re}_0], e^{-ikty}\pa_y^2\w \om{\rm re}_0}_{L^2}\\
&=-2\nu\sum_{k\in \mathbb{Z}\backslash \{0\}}\frac{1}{|k|^3}\Re\Big\lan\mf J_k[\pa_y(e^{-ikty})\w \om{\rm re}_0+e^{-ikty}\pa_y\w \om{\rm re}_0]+\underbrace{[\pa_y,\mf J_k]}_{\mf H_k}[(e^{-ikty}\w \om{\rm re}_0)], e^{-ikty}\pa_y\w \om{\rm re}_0\Big\ran_{L^2}\\
&\quad -2\nu\sum_{k\in \mathbb{Z}\backslash \{0\}}\frac{ikt}{|k|^3}\Re\br{\mf J_k[e^{-ikty}\w \om{\rm re}_0)], e^{-ikty}\pa_y\w \om{\rm re}_0}_{L^2} .
\end{align*} We observe that there is a key cancellation among the first term and the last term so that potential dangerous $t$-growth are balanced out.  We use the estimates of $\mathfrak{J}_k$, $\mathfrak{H}_k=[\pa_y, \mf J_k]$ in Corollary \ref{corol-bound} and Lemma \ref{comm-bound} (taking $\mathcal{U}(t,y)=y$) to obtain 
\begin{align*}
|T_3|&\le 2\nu \sum_{k\in\mathbb Z\backslash\{0\}}\f{1}{|k|^3}\|\mathfrak{J}_k\|_{L^2\to L^2}\|{e^{-ikty}\pa_y\w\om{\rm re}_0}\|_{L^2}^2\\
&\qquad+2\nu \sum_{k\in\mathbb Z\backslash\{0\}}\f{1}{|k|^3}\|\mathfrak{H}_k\|_{L^2\to L^2}\|e^{-ikty}\w\om{\rm re}_0\|_{L_2}\|\pa_y\w\om{\rm re}_0\|_{L_2}\\
&\le C\nu\|\pa_y\w\om{\rm re}_0\|_{L^2}^2+C\nu\|\w\om{\rm re}_0\|_{L^2}\|\pa_y\w\om{\rm re}_0\|_{L^2}.
\end{align*}
Combining the above estimates for $T_1$--$T_3$, we apply the established bounds \eqref{E0_pf2}, \eqref{E0_pf3} and integrate in time to obtain
\begin{equation}
\begin{aligned}\label{E0_pf4}
&\sup_{t\in[T_0,T_\ast)}\mathbb{E}_{0,\mf J}(t)+\frac{3}{8}\int_{T_0}^{T_\ast}\mathcal{CK}_0dt\\
&\qquad\leq C_2 \ep^2\nu^{\f23}|\ln\nu|\|\w\omega{\rm re}_0\|_{L_t^\infty L_y^2}+C_2\ep^4\nu^{\f43}|\ln\nu|^2+C_2\nu(\|\w\om{\rm re}_0\|_{L^2_tL_y^2}^2+\|\pa_y \w\om{\rm re}_0\|_{L^2_tL_y^2}^2).
\end{aligned}
\end{equation}
This concludes the Step \# 3. 

\step{4: Concluding the $\mathbb{E}_0$-estimate.} 

Combining the estimates \eqref{E0_pf3}, \eqref{E0_pf4}, we choose $c_0 C_2\leq \f12$, and $B^{-1}=c_0/8$ to obtain that
\begin{align*}
\sup_{t\in[T_0,T_\ast)}\mathbb{E}_{0}(t)+\frac12\nu\|\pa_y \w\om{\rm re}_0\|_{L^2_tL_y^2}^2+\frac{c_0}{4}\int_{T_0}^{T_\ast}\mathcal{CK}_0dt\leq 
\frac{C_3}{c_0}\ep^4\nu^{\f43}|\ln\nu|^2+C_3\nu\|\w\om{\rm re}_0\|_{L^2_tL_y^2}^2.
\end{align*} 
Finally, it is sufficient to derive the $\|\w\om{\rm re}_0\|_{L^2_tL_y^2}$-control. To this end, we take the inner product of \eqref{eq-0re} and $\w u{\rm re}_0$ to obtain
\begin{align*}
    &\f{1}{2}\f{d}{dt}\|\w u{\rm re}_0\|_{L^2}^2+\nu\|\pa_y\w u{\rm re}_0\|_{L^2}^2+\al^{-1}\nu\big(|\w u{\rm re}_0(1)|^2+|\w u{\rm re}_0(-1)|^2\big)\\
    &=\Big\langle \sum_{k\in\mathbb Z\backslash\{0\}}ik(\p_k-\p_k^{(a)})\w\om{a}_{-k},\w u{\rm re}_0\Big\rangle \le \||k|(\p_k-\p_k^{(a)})\|_{l_k^2L_y^2}\|\w\om{a}_{-k}\|_{l_k^2L_y^{\infty}}\|\w u{\rm re}_0\|_{L^2_y}.
\end{align*}
Integrating from $T_0$ to $T_*$, using the $\w\om{a}$-estimates in Lemma \ref{lem:estimate-u1_k}, the velocity estimate \eqref{ineq-size}, together with the smallness assumption \eqref{smallness}, it follows
\begin{align*}
     &\|\w u{\rm re}_0\|_{L^{\infty}L^2}^2+\nu\|\pa_y\w u{\rm re}_0\|_{L^{2}L^2}^2\\
     &\quad \le C\||k|(\p_k-\p_k^{(a)})\|_{L^2_tl_k^2L_y^2}\|\w\om{a}_{-k}\|_{L^2_tl_k^2L_y^{\infty}}\|\w u{\rm re}_0\|_{L_t^{\infty}L^2_y}\leq C \ep^2\nu^{\f56}|\ln\nu|\|\w u{\rm re}_0\|_{L^{\infty}_tL_y^2}\\
    \leadsto &\|\w u{\rm re}_0\|_{L^{\infty}L^2}^2+\nu\|\w \omega{\rm re}_0\|_{L^{2}L^2}^2\leq C\ep^4\nu^{\f53}|\ln\nu|^2.
\end{align*}
At this end, we can conclude that
\begin{align}
\sup_{t\in[T_0,T_\ast)}\mathbb{E}_{0}(t)+\frac12\nu\|\pa_y \w\om{\rm re}_0\|_{L^2_tL_y^2}^2+\frac{c_0}{4}\int_{T_0}^{T_\ast}\mathcal{CK}_0dt\lesssim \ep^4\nu^{\f43}|\ln\nu|^2.
\end{align}
Finally, taking $\ep$ small enough, then we finish the proof.
\end{proof}

\begin{proof}[Proof of Proposition \ref{prop-w*0}]
    Recall that $\w u{\e}_0$ solves \eqref{eq-0*}, for $t\in[T_0,T_*)$, since $\w\om{\e}_0(y=\pm1)=0$, through integration by parts, it yields
\begin{equation*}
\begin{aligned}
   & \langle \pa_t u_0^{(\e)},-\pa_y^2u_0^{(\e)}\rangle+\nu\|\pa_y^2u_0^{(\e)}\|_{L^2}^2=\f12\f{d}{dt}\big(\f{1}{\al}(u_0^{(\e)}(1))^2+\f{1}{\al}(u_0^{(\e)}(-1))^2+\|\om_0^{(\e)}\|_{L^2}^2\big)+\nu\|\pa_y^2u_0^{(\e)}\|_{L^2}^2\\
    &=-\int_{-1}^1\sum_{k\in\mathbb Z\backslash\{0\}}\big(ik\pa_y\p_k^{(a)}\om_{-k}^{(a)}+ik\p_k^{(a)}\pa_y\w\om{a}_{-k}\big)\overline{\om_0^{(\e)}}dy+\int_{-1}^1\sum_{k\in\mathbb Z\backslash\{0\}}ik\p_k\big(\om_{-k}-\w\om{a}_{-k}\big)\overline{\pa_y\om_0^{(\e)}}dy\\
    &=:T_1+T_2.
\end{aligned}
\end{equation*}
By the $\w\om{a}$-estimates in Lemma \ref{lem:estimate-u1_k} and the bootstrap assumption \eqref{hyp_om_0_l}, we obtain
\begin{align*}
   &\lf|\int_{T_0}^{T_*}T_1dt\rg|
   \lesssim \|ik\pa_y\p_k^{(a)}\om_{-k}^{(a)}+ik\p_k^{(a)}\pa_y\w\om{a}_{-k}\|_{L^{1}_t([T_0,T_*);l_k^1L^{2}_y)}\|\om_0^{(\e)}\|_{L^{\infty}L^{2} }\\
   &\lesssim |\ln\nu|\|\om_{\ne;\rm in}\|_{H_{x,y}^3}^2\|\om_0^{(\e)}\|_{L^{\infty}L^{2} }\lesssim \ep^3\nu^{\f43}|\ln\nu|^2.
\end{align*}
Next, for $T_2$, through a similar argument to $\mathbb F_{k;32}$ in Lemma \ref{Lem-Fk3}, we get
\begin{align*}
    &\lf|\int_{T_0}^{T_*}T_2dt\rg|
    \lesssim \big\||k|\p_k\big\|_{L^2_t([T_0,T_*);l_k^2L_y^{\infty}})\Big\|\w\om{\rm e}_{-k}+\w\om{\rm re}_{-k}\Big\|_{L^{\infty}_t([T_0,T_*);l_k^2L_y^2)}\\
    &\quad +\big\|(1-|y|)^{-1}|k|\p_k\big\|_{L^2_t([T_0,T_*);l_k^2L_y^{\infty})}\big\|(1-|y|)(\om_{-k}^{(b)}+\w\om{*,b}_{-k})\big\|_{L^{\infty}_t([0,T_*);l_k^2L_y^2)}\|\pa_y\om_0^{(\e)}\|_{L^2_t([T_0,T_*);L^2_y)}\\
    &\le C\ep^4\nu^{\f32}|\ln\nu|^4+\f14\nu\|\pa_y\om_0^{(\e)}\|_{L^2_t([T_0,T_*);L_y^2)}^2.
\end{align*}
Using the fact $\nu\|\pa_y\om_0^{(\e)}\|_{L^2_t([T_0,T_*);L_y^2)}^2=\nu\|\pa_y^2u_0^{(\e)}\|_{L^2_t([T_0,T_*);L_y^2)}^2$, we can obtain
\begin{align*}
    \sup_{t\in[T_0,T_*)} \Big(\f{1}{\al}(u_0^{(\e)}(1))^2+\f{1}{\al}(u_0^{(\e)}(-1))^2\Big)+\|\om_0^{(\e)}\|_{L^{\infty}L^2}^2\le C\ep^3\nu^{\f43}|\ln\nu|^2+C\ep^4\nu^{\f32}|\ln\nu|^4.
\end{align*}
After taking $\ep$ and $\nu$ small enough, we completes the proof.
\end{proof}

\appendix
\section{Airy Function}
Assume $k>0$ for simplicity (otherwise replace $k$ by $|k|$). To capture the behavior of the Evans function $D$, the following Airy primitive plays a central role (see, e.g., \cite{Romanov73}, \cite{ChenLiWeiZhang18}),
\begin{equation}\label{eq:def-A0}
A_0(z):=\int_{e^{i\pi/6}z}^{+\infty} Ai(t)\,dt
      =e^{i\pi/6}\int_z^{+\infty} Ai(e^{i\pi/6}t)\,dt.
\end{equation}
The derivative of the primitive Airy function has a simple form: 
\begin{equation}\label{eq:A0-deriv}
A_0'(z)=-e^{i\pi/6} Ai(e^{i\pi/6}z).
\end{equation}
$A_i(z)$ is the Airy function which satisfies $
    \pa_z^2A_i(z)-zA_i(z)=0.$ 
Two quantities are essential in capturing the $A_0$, namely, the two-point ratio and  the logarithmic-derivative:
\begin{equation}\label{eq:def-omega-q}
\al(z,t):=\frac{A_0(z+t)}{A_0(z)}=\exb{-\int_0^t q(z+t)dt},\qquad
q(z):=-\frac{A_0'(z)}{A_0(z)}.
\end{equation}
We have the following asymptotic formula for $|{\rm Arg} z|\le \pi-\ep$, $\ep>0$ (see \cite{Romanov73, Wasow})
\begin{align*}
  &  A_i(z)=\frac{1}{2\sqrt{\pi}}z^{-\f14}e^{-\f23z^{\f32}}\big(1+\mathcal{O}(z^{-\f32})\big),\quad
    A_0(z)=\frac{1}{2\sqrt{\pi}}(ze^{\frac{\pi}{6}i})^{-\f34}e^{-\f23(ze^{\frac{\pi}{6}i})^{\f32}}\big(1+\mathcal{O}(z^{-\f32})\big).
\end{align*}
\begin{lemma}[Lemma 8.1 in \cite{ChenLiWeiZhang18}]\label{Lem-Airy1} It holds that
\begin{enumerate}
    \item There exists $\d_0>0$ such that $A_0(z)$ has no zeros in the half plane ${\rm Im}z\leq \d_0$.
    \item  Let $a(\d)=\sup\lf\{{\rm Re}\lf(\frac{A_0'(z)}{A_0(z)}\rg):\ {\rm Im}z\le \d\rg\}$. There exists $\d_0>0$ such that $a(\d)\in C([0,\d_0])$ and $a(0)=-0.4843\cdots<-1/3$.
    \item For $|{\rm Arg}(ze^{\frac{\pi}{6}i})|\leq \pi-\ep$, $\ep>0$, we have the following asymptotic formula
    \begin{align}\label{asy-Airy}
        \frac{A_0'(z)}{A_0(z)}=-e^{\frac{\pi}{6}i}\lf(ze^{\frac{\pi}{6}i}\rg)^{\f12}+\mathcal{O}(z^{-1}).
    \end{align}
\end{enumerate}
\end{lemma}
\begin{lemma}\label{ratio}
    Let $\d_0$ be as in Lemma \ref{Lem-Airy1}. There exists $C_0>0$ such that for ${\rm Im}z\le\d_0$,
  \begin{align*}
  |q(z)|
  \le C(1+|z|^{\f12}),\ \lf|\frac{A_0(z)}{A_0'(z)}\rg|\leq \frac{C}{1+|z|^{\f12}},\ {\rm Re}\frac{A_0'(z)}{A_0(z)}\le -C_0(1+|z|^{\f12}),\ \lf|\f{A_0''(z)}{A_0(z)}\rg|\le C(1+|z|). 
\end{align*} Moreover, for ${\rm Im}z\le \delta_0$ and $x\ge0$, 
\begin{align} |\al(z,x)|\le  e^{-c((1+|z|^{1/2})x+x^{3/2})}.\label{eq:al-bnd}  
\end{align} 
\end{lemma}
\begin{proof}The proofs of these are elementary. We refer the readers to \cite{ChenLiWeiZhang18,ChenWeiZhang3D} for the proof of the first and third inequality. The estimates of the remaining bounds follow directly from the asymptotic behavior of $q(z)$, namely, \begin{align*}
    q(z)=e^{\pi i/6}(ze^{\pi i/6})^{1/2}+O(|z|^{-1}),
\end{align*}
and the fact that 
$
\frac{A_0''(z)}{A_0(z)}=p'(z)+p(z)^2.
$
\end{proof}

\section{Estimates of the Boundary Correctors}\label{sec:BC}
In this section, we prepare technical lemmas required for estimating $\mathscr{C}_k^{[j]}[g_k]$. Assume $k\ge 1$ without loss of generality. Define trace operators that encode the boundary conditions:
\begin{equation}  
\T_{+;k}[\omega_k]:=\lf[(1-\kappa)\omega_k+ \kappa\pa_y\Delta_k^{-1}\omega_k\rg]\Big|_{y=1},\ 
\T_{-;k}[\omega_k]:=\lf[(1-\kappa)\omega_k- \kappa\pa_y\Delta_k^{-1}\omega_k\rg]\Big|_{y=-1}.\label{Tr_pm_BC}
\end{equation} We will drop the $(\cdots)_{k}$ when possible. 
Compared with \eqref{Trace_op}, since $\al=(1-\ka)/\ka$, we obtain
\begin{equation}\label{eq-relation}
 \f{1}{\ka}   \T_{\pm;k}[\omega_k]=\mathbb{T}_{\pm}[\om_k].
\end{equation}
Recalling the Green's function $G_k(y,y')$ \eqref{eq-Green} for the Dirichlet Laplacian $\de_k=\pa_y^2-k^2$ on $[-1,1]$, we obtain that for any function $f\in H^1_y(-1,1)$,
\begin{equation}\label{eq:Kpm}\begin{split}
\T_{+;k}[f]:=&(1-\kappa)f(1)+\kappa\int_{-1}^1 K_+(y)\,f(y)\,dy,\quad K_k^+(y):=\frac{\sinh (|k|(y+1))}{\sinh(2|k|)}\in[0,1],\\
\T_{-;k}[f]:=&(1-\kappa)f(-1)+\kappa\int_{-1}^1 K_-(y)\,f(y)\,dy,
\quad 
K_k^-(y):=\frac{\sinh (|k|(1-y))}{\sinh(2|k|)}\in[0,1].
\end{split} 
\end{equation}
We observe that $0\le K_k^\pm\le1$ on $[-1,1]$.

To account for the Robin boundary condition, we construct the boundary correctors and derive its estimates. Let $(w_+,\phi_{+})$ and $(w_-,\phi_{-})$ be the unit boundary correctors solving the homogeneous Orr-Sommerfeld equation
\begin{equation}\label{phi-OS}
\begin{cases}
    -\nu\de_k^2\phi_{+;k}+ik(\mathcal{V}(y)-\lambda)\de_k\phi_{+;k}=0,\\ \phi_{+;k}(\pm1)=0,\quad \de_k\phi_{+;k}=w_{+;k},\\
   \T_{+;k}[w_{+;k}]=1, \quad   \T_{-;k}[w_{+;k}]=0.
\end{cases}\quad
\begin{cases}
    -\nu\de_k^2\phi_{-;k}+ik(\mathcal{V}(y)-\lambda)\de_k\phi_{-;k}=0,\\
        \phi_{-;k}(\pm1)=0,\quad \de_k\phi_{-;k}=w_{-;k},\\
  \T_{-;k}[w_{-;k}]=1,\quad  \T_{+;k}[w_{-;k}]=0 .
\end{cases}
\end{equation}

 To select the physically relevant homogeneous solutions, we define the upper/lower boundary layer scale $L_\pm$ and the corresponding spectral shifts $d_\pm$:
\begin{align}\label{eq:def-Lpm-dpm}
    L:=\nu^{-\f13}|k|^\f13,\quad L_{\pm}:=\lf(\frac{\nu}{\mathcal{V}'(\pm1)k}\rg)^{-\f13},\quad d_{\pm}:=\frac{\mathcal{V}(\pm1)\mp\mathcal{V}'(\pm1)-\lambda-i\nu k}{\mathcal{V}'(\pm1)}.
\end{align}
With these parameters, we consider the approximate homogeneous Airy solutions:
\begin{align}\label{W_a_pm}
    W_{\pm;a}(y)= Ai\lf(L_{\pm}(y+d_{\pm})e^{i(\frac{\pi}{2}\pm\frac{\pi}{3})}\rg).
\end{align}
Recalling the equation satisfied by the Airy function, namely, $ Ai''(y)-y Ai(y)=0$, one checks that the approximate solutions solve the modified Airy equation
\begin{align*}
    \nu(\pa_y^2-k^2)W_{\pm;a}-ik\lf(\mathcal{V}(\pm1)+\mathcal{V}'(\pm1)(y\mp1)-\lambda\rg)W_{\pm;a}=0.
\end{align*} Since the shear profile is still linear, rigorous computations can be made on the $W_{\pm;a}$. Before continuing, we define some augmented variables to capture the structure of the functions $W_{\pm;a}$. Define 
\begin{equation}
\begin{aligned}\label{eq:def-zpm}
&Z_-:=L_-(d_- -1),\qquad
Z_+:=-L_+\,(1+\overline{d_+});\\
&\text{Rescaled spatial variables: }\begin{cases}
Y_-:=L_-(y+1),\quad\text{near }y=-1;\\
Y_+:=L_+(1-y),\quad\text{near } y=1.
\end{cases}
\end{aligned}
\end{equation}
With these variables, we can rephrase the $W_{\pm;a}$ as follows:
\begin{equation}\label{eq:W-A0prime}
\begin{aligned}
&W_{-;a}(y)=-e^{-i\pi/6}A_0'(Z_-+Y_-),\ \
W_{+;a}(y)=-e^{i\pi/6}\overline{A_0'(Z_+ + Y_{+})}.
\end{aligned}
\end{equation}

To compensate for the mismatch between the linear profile $\mathcal{V}(\pm1)+\mathcal{V}'(\pm1)(y\mp1)$ and the actual shear profile $\mathcal{V}(y)$, we introduce the correctors $W_{\pm;\e}$ which satisfy
\begin{equation}\label{eq:W_e}
    \begin{cases}
        -\nu(\pa_y^2-k^2)W_{\pm;\e}+ik(\mathcal{V}-\lambda)W_{\pm;\e}=-ik\big(\mathcal{V}-\mathcal{V}(\pm1)-\mathcal{V}'(\pm1)(y\mp1)\big)W_{\pm;a},\\
        \Phi_{\pm;a}=\Delta_k^{-1}W_{\pm;a},\quad \Phi_{\pm;\e}=\Delta_k^{-1}W_{\pm;\e},\quad W_{\pm;\e}(\pm1)=0.
    \end{cases}
\end{equation}
Finally, we observe that the sum of these components solves the resolvent equation:
\begin{align}\label{W_pm_defn}
   & W_{\pm}=W_{\pm;a}+W_{\pm;\e},\quad \Phi_{\pm}=\Delta_k^{-1}W_{\pm},\quad  -\nu(\pa_y^2-k^2)W_{\pm}+ik(\mathcal{V}(y)-\lambda)W_{\pm}=0.
\end{align}
We summarize the estimates of $W_{\pm;a}$ and $W_{\pm; e}$ in the following lemmas.
\begin{lemma}\label{Lem-Wpma}
    There exist $c>0$, $C>0$ independent of $L_{\pm}$, $d_{\pm}$ such that
   \begin{align*}
      &\lf| \frac{W_{\pm;a}(y)}{A_0(Z_\pm)}\rg|\le C(1+|Z_\pm|)^{\f12}e^{-cY_\pm(1+|Z_\pm|)^{\f12}},\qquad \|W_{\pm,a}\|_{L^1}\le CL^{-1}|A_0(Z_\pm)|,\\
       &\|(\pm 1-y)W_{\pm;a}\|_{L^2}+L\big(1+|Z_\pm|\big)^{\f12}\|(\pm1- y)^2W_{\pm;a}\|_{L^2}+\|(ik,\pa_y)\Phi_{\pm;a}\|_{L^2}\le \frac{C|A_0(Z_\pm)|}{ L^{\f32}\big(1+|Z_\pm|\big)^{\f14}}.
       \end{align*}
\end{lemma}
\begin{proof} The proof of these estimates are similar to the ones presented in \cite[Lemma B.3]{ChenWeiZhang23}, \cite[Lemma 5.6]{ChenWeiZhang3D}, so we omit them for the sake of brevity.
\end{proof}
Next, we present the estimate for the $W_{\pm; \e}$. We recall the following technical lemma.
\begin{lemma}\label{Lem-Na}
Consider the solution $w_k\in H_{y}^2$ to the following problem:
\begin{equation}\label{eq-Na-c}
    \begin{cases}
        -\nu(\pa_y^2-k^2)w_k+ik(\mathcal{V}-\la)w_k=F_k,\\
        w_k(t,y=\pm1)=0,\quad (\pa_y^2-k^2)\phi_k=w_k,\quad \phi_k(t,y=\pm1)=0.
    \end{cases}
\end{equation}
Let $\la=\la_r+i\la_i$ $(k\la_i\geq -\d \nu^{\f13}|k|^{\f23})$, $F\in L^2_y$, $\|\mathcal{V}-y\|_{H^4}\ll\ep\nu^{\f13}$. Then it holds
    \begin{align*}
        \nu^{\f16}|k|^{\f43}\|(ik,\pa_y)\phi_k\|_{L^2}+\nu^{\f16}|k|^{\f56}\|w_k\|_{L^1}+\nu^{\f13}|k|^{\f23}\|w_k\|_{L^{2}}+\nu^{\f23}|k|^{\f13}\|(ik,\pa_y)w_k\|_{L^2}\le C\|F_k\|_{L^2}.
    \end{align*}
\end{lemma}
\begin{proof}
   The proof is similar to the one for  \cite[Proposition 4.1] {ChenWeiZhang3D}. So we omit further details.
\end{proof}
Since $W_{\pm;\e}$ solves the equation \eqref{eq-Na-c}, direct application of Lemma \ref{Lem-Na} and the estimates in \eqref{Lem-Wpma} yields the following lemma.
\begin{lemma}\label{Lem-error}
Assume that $\|\mathcal{V}-y\|_{H^4}\ll\ep\nu^{\f13}$ and consider the regime $k\lambda_i\geq -\ep_0\nu^{\f13}|k|^{\f23}$, then 
    \begin{align*}
       &\|W_{\pm,\rm e}\|_{L^{\infty}}+L\|W_{\pm;\rm e}\|_{L^1}\le CL^{-1}\big(1+|Z_\pm|\big)^{-\f34}|A_0(Z_\pm)|,\\
       &\|(ik,\pa_y)\Phi_{\pm;\rm e}\|_{L^2}\le CL^{-2}|k|^{-\f12}\big(1+|Z_\pm|\big)^{-\f34}|A_0(Z_\pm)|.
    \end{align*}
\end{lemma}
\begin{proof}
    Since $|\mathcal{V}-\mathcal{V}(1)-\mathcal{V}'(1)(y-1)|\le C\|\mathcal{V}''\|_{L^{\infty}}(y-1)^2$, then by Lemma \ref{Lem-Na}, the Gagliardo-Nirenberg inequality, and Lemma \ref{Lem-Wpma}, we immediately get
    \begin{align*}
        &\nu^{\f16}|k|^{\f43}\|(ik,\pa_y)\Phi_{+;\rm e}\|_{L^2}+\nu^{\f12}|k|^{\f12}\|W_{+;\rm e}\|_{L^{L^{\infty}}}+\nu^{\f16}|k|^{\f56}\|W_{+;\rm e}\|_{L^{2}}\\
        &\quad\le C|k|\|(y-1)^2W_{+;a}\|_{L^2}\le C|k|L^{-\f52}\big(1+|L_+(\overline{d}_++1)|\big)^{-\f34}|A_0(-L_+(\overline{d}_++1))|.
    \end{align*}
   The corresponding estimates of $W_{-;\rm e}$ and $\Phi_{-;\rm e}$ are similar. Finally, recall the notations that $Z_-=L_-(-1+d_-)$, $Z_+=-L_+(1+\overline{d_+})$, and $Y_-=L_-(1+y)$, $Y_+=L_+(1-y)$, then we finish the proof.
\end{proof}
Due to the relationship \eqref{W_pm_defn}, Lemma \ref{Lem-Wpma} and Lemma \ref{Lem-error} immediately give
\begin{lemma}\label{W12}
 For $k\lambda_i\geq -\d \nu^{\f13}|k|^{\f23}$, it holds that
    \begin{align}\begin{split}
      &\lf\| W_{\pm}\rg\|_{L^{\infty}}\le C(1+|Z_\pm|)^{\f12}|A_0(Z_\pm)|,\qquad \|W_{\pm}\|_{L^1}\le CL^{-1}|A_0(Z_\pm)|,\\
       &\|(y\mp 1)W_{\pm}\|_{L^2}+L\big(1+|Z_\pm|\big)^{\f12}\|(y\mp 1)^2W_{\pm}\|_{L^2}+\|(ik,\pa_y)\Phi_{\pm}\|_{L^2}\le \frac{C|A_0(Z_\pm)|}{ L^{\f32}\big(1+|Z_\pm|\big)^{\f14}}.\end{split}\label{W_pm_est}
       \end{align}
\end{lemma}
Now, we have identified the homogeneous equation $W_\pm$ in \eqref{W_pm_defn}, we can use them to represent the unit boundary correctors $(w_{\pm},\phi_{\pm})$:
\begin{align}\label{Unit_BCorrector}
w_\pm (\lambda, y)= C_\pm^+(\lambda)W_+(\lambda, y)+C_\pm^-(\lambda)W_-(\lambda, y), \quad \phi_\pm(\lambda, y) =\Delta^{-1}w_\pm(\lambda, y).
\end{align}
By plugging into the equation, one ends up with the following linear system:
\begin{align*}
\begin{pmatrix}
\T_{+;k}(W_+) &\T_{+;k} (W_-)\\
\T_{-;k} (W_+) & \T_{-;k} (W_-)
\end{pmatrix}\binom{C_+^+}{C_+^-}=\binom{1}{0},\qquad\begin{pmatrix}
\T_{+;k}(W_+) &\T_{+;k} (W_-)\\
\T_{-;k} (W_+) & \T_{-;k} (W_-)
\end{pmatrix}\binom{C_-^+}{C_-^-}=\binom{0}{1}.
\end{align*}
Set $A_\pm:=\T_{+;k}(W_\pm)$ and $B_\pm:=\T_{-;k}(W_\pm)$. The Evans function is
\begin{equation}\label{eq:Evans}
D(\lambda,k):=A_+ B_- - A_- B_+.
\end{equation}
If $D(\lambda,k)\neq0$, we can uniquely determine the solution
\begin{align}\label{Cpm_pmexpression}
\binom{C_+^+}{C_+^-}=\frac{1}{D }\binom{B_-}{-B_+},\qquad \binom{C_-^+}{C_-^-}=\frac{1}{D }\binom{-A_-}{A_+}.
\end{align}

In order to determine the Evans function \eqref{eq:Evans} associated with the system, we recall from \eqref{eq:Kpm} that it is essential to understand the trace of the solution $w(\pm 1)$ together with the integrated quantities $\int_{-1}^1 K_\pm (y) f(y)dy$. To this end, we first focus on the approximate Airy solutions $W_{\pm;a}$.  
\begin{lemma}[Airy trace formulas]\label{lem:A0-traces}
The approximate Airy solutions satisfy
\begin{equation}
\begin{aligned}  
&W_{-;a}(-1)=e^{-i\pi/6}A_0(Z_-)\,q(Z_-),&&\quad|W_{-;a}(1)|\le C L_-^{-1}|A_0(Z_-)|
;\\
&
W_{+;a}(1)=e^{i\pi/6}\overline{A_0(Z_+)}\,\overline{q(Z_+)},&&\quad|W_{+;a}(-1)|\le C L_+^{-1}|A_0(Z_+)|, 
\label{eq:wall-traces}
\end{aligned}
\end{equation}
and the kernel integrals can be represented as follows,
\begin{align}
\int_{-1}^1 K_-(y)\,W_{-;a}(y)\,dy
&=e^{-i\pi/6}L_-^{-1}A_0(Z_-)\,m_-(Z_-,L_-,k),\label{eq:diag-minus}\\
\int_{-1}^1 K_+(y)\,W_{+;a}(y)\,dy
&=e^{i\pi/6}L_+^{-1}\overline{A_0(Z_+)}\,m_+(Z_+,L_+,k),\label{eq:diag-plus}\\
\int_{-1}^1 K_+(y)\,W_{-;a}(y)\,dy
&=-e^{-i\pi/6}L_-^{-1}A_0(Z_-)\,n_-(Z_-,L_-,k),\label{eq:off-minus}\\
\int_{-1}^1 K_-(y)\,W_{+;a}(y)\,dy
&=-e^{i\pi/6}L_+^{-1}\overline{A_0(Z_+)}\,n_+(Z_+,L_+,k).\label{eq:off-plus}
\end{align}
Here, 
\begin{align}\label{eq:def-mnminus}\begin{split}
m_-(z,L,k)&:=1-\frac{k}{L(1-e^{-4k})}\int_0^{2L}\big(e^{-kt/L}+e^{-4k}e^{kt/L}\big)\,\al(z,t)\,dt,\\
n_-(z,L,k)&:=\alpha(z,2L)-\frac{k\,e^{-2k}}{L(1-e^{-4k})}\int_0^{2L}\big(e^{kt/L}+e^{-kt/L}\big)\,\al(z,t)\,dt,\end{split}
\\
\label{eq:def-mnplus}
m_+(z,L,k)&:=\overline{m_-(z,L,k)},\qquad
n_+(z,L,k):=\overline{n_-(z,L,k)}.
\end{align}
\end{lemma}

\begin{proof}
The proof is a direct consequence of calculations and the definitions.
\end{proof}
Next we develop some estimate about the quantifiers $m_\pm,\ n_\pm$.
\begin{lemma}
\label{lem:airy-bounds}Assume all conditions as in Lemma \ref{Lem-Airy1} and  \ref{ratio}. 
There exist $\delta,\,\nu_0\in(0,\frac{1}{2}]$ and constants
$c_0,C_0>0$ (depending only on $\delta_0$ and $\mathcal V'(\pm1)$)
such that if
\begin{equation}\label{eq:k-band}
0<\nu\le \nu_0,\qquad 1\le |k|, \qquad k\Im\ \lambda\geq -\delta  \nu^\f13 |k|^{\f23}, 
\end{equation}
then $\Im Z_\pm\le\delta_0$ with $\delta_0$ being defined in Lemma \ref{Lem-Airy1}, and the following estimates hold
\begin{equation}\label{eq:mnbounds-new}
\Re\ m_\pm(Z_\pm,L_\pm,|k|)\ge c_0>0,\ 
|m_\pm(Z_\pm,L_\pm,|k|)|\le C_0,\ 
|n_\pm(Z_\pm,L_\pm,|k|)|\le C_0 e^{-c_0|k|} L_\pm^{-1}.
\end{equation}
\end{lemma}

\begin{proof}

The proof is organized in steps. Throughout the proof, the parameter $c\in(0,1)$ is a small constant that changes from line to line. Moreover, it is direct to check that as long as $\delta>0$ is small enough, $\Im Z_\pm \leq \delta_0$. 

\smallskip
\noindent\textbf{Step 1: A uniform upper bound for $|m_\pm|$.}
We focus on the estimate for $m_-$. First of all, we recall the estimate \eqref{eq:al-bnd}   \begin{align} |\al(Z,X)|\le e^{-c((1+|Z|^{\f12})X+X^{\f32})}\leq 1,\quad \text{for } \Im\ Z \leq \delta_0\text{  and } X\geq 0.\label{Airybnd_pf1}
\end{align} 
Furthermore, one can derive that $
\int_0^{2L}\big(e^{-|k|Y_-/L_-}+e^{-4|k|}e^{|k|Y_-/L_-}\big)\,dY_-=\frac{L_-}{|k|}(1-e^{-4|k|}).
$ 
Hence,
\[
|m_-(Z_-,L_-,|k|)|
\le 1+\frac{|k|}{L_-(1-e^{-4|k|})}\frac{L_-}{|k|}(1-e^{-4|k|})=2.
\]
So we may take $C_0\ge2$. Similar arguments applies for $m_+(Z_+,L_+,|k|)$. This concludes Step 1. 

\smallskip
\noindent\textbf{Step 2: A lower bound for $\Re\ m_\pm$.}
From the definition of $m_-$ and \eqref{eq:al-bnd},
\begin{align*}
&\Re\ m_-(Z,L_-,|k|)
\ge 1-\frac{|k|}{L_-(1-e^{-4|k|})}\int_0^{2L_-}\big(e^{-|k|Y_-/L_-}+e^{-4|k|}e^{|k|Y_-/L_-}\big)|\al(Z,Y_-)|\,dY_-\\
&\ge 1-\frac{|k|}{L_-(1-e^{-4|k|})}\int_0^{2L_-}\big(e^{-|k|Y_-/L_-}+e^{-4|k|}e^{|k|Y_-/L_-}\big)e^{-c(1+|Z_-|^\f12)Y_-}\,dY_-=:1+T_1+T_2.
\end{align*}
Estimate the $T_1$-term as follows
\[
T_1= \frac{-|k|}{L_-(1-e^{-4|k|})}\int_0^{2L_-}e^{-(L_-^{-1}|k|+c(1+|Z_-|^\f12))Y_-}\,dY_-=\frac{-|k|}{L_-(1-e^{-4|k|})}\frac{1-e^{-2|k|-c2L_-(1+|Z_-|^{\f12})}}{L_-^{-1}|k|+c(1+|Z_-|^\f12)}.
\]
For the integral involved in $T_2$, we fix an arbitrary constant $c$ of order $1$ and distinguish between two cases, 
\begin{itemize}
\item[\textbf{a)}] $L_-^{-1}|k|\leq \f c2$: In this case, we have that
\begin{align*}
e^{-4|k|}\int_0^{2L_-}e^{(L_-^{-1}|k|-c)Y_-}\,dY_-
=
e^{-4|k|}\frac{e^{2|k|-2cL_-}-1}{L_-^{-1}|k|-c}\leq C(c^{-1})e^{-2|k|}\mathbf{1}_{|k|\le(\f c2)^{\f32}\nu^{-\f12}}.
\end{align*}
\item[\textbf{b)}] $L_-^{-1}|k|>\f c2$: In this case, we have that $
|k|\geq \lf(\frac{c}{2}\rg)^\f32\nu^{-\f12}.$
As a consequence, we estimate the integral as follows. 
\begin{align*} 
e^{-4|k|}\int_0^{2L_-}e^{(L_-^{-1}|k|-c)Y_-}\,dY_-
&\leq e^{-4|k|}\int_0^{2L_-}e^{L_-^{-1}|k| Y_-}\,dY_-\leq 2c^{-1}e^{-2|k|}\mathbf{1}_{|k|\geq (\f c2)^{\f32}\nu^{-\f12}}. 
\end{align*}
\end{itemize}
Hence, $T_2\geq -\frac{ C(c^{-1})|k|}{L_-(1-e^{-4|k|})} e^{-2|k|}.$
In conclusion, we obtain that 
\begin{align*}
&\hspace{-.25cm}\Re\ m_-(Z,L_-,|k|)\geq 1-\frac{L_-^{-1}|k|}{(1-e^{-4|k|})}\frac{1}{L_-^{-1}|k|+c(1+|Z_-|^\f12)}-C\,\frac{|k|}{ L_-(1-e^{-4|k|})}\,e^{-2|k|}\\
\geq &\frac{c(1+|Z_-|^\f12)}{c(1+|Z_-|^\f12)+L_-^{-1}|k|}\lf(1-\frac{|k|e^{-4|k|}}{ c(1+|Z_-|^\f12)(1-e^{-4|k|})L_- }-C\frac{|k|e^{-2|k|}(c(1+|Z_-|^\f12)+L_-^{-1}|k|)}{c(1+|Z_-|^\f12)L_-(1-e^{-4|k|})}\rg).
\end{align*}
One can see that, as long as the $\nu$ is chosen small enough, the expression in the parenthesis is greater than $\f12$ and we obtain that \begin{align*}
\Re\ m_-(Z_-,L_-,|k|)\geq \frac{c(1+|Z_-|^\f12)}{2(c(1+|Z_-|^\f12)+|k|L_-^{-1})}>0. 
\end{align*}
In order to derive a concrete lower bound, we distinguish between two regimes:
\begin{itemize}
\item [\textbf{a)}] If $1\leq |k|\leq \nu^{-\f12}$, then $
L_-^{-1}|k|=(\mathcal{V}'(-1))^{-\f13}\nu^{\f13}|k|^{\f23}\leq  (\mathcal{V}'(-1))^{-\f13}.$ 
Hence, we have that 
\begin{align*}
\Re\ m_-(Z,L_-,|k|)\geq \frac{c}{2c+ 2(\mathcal{V}'(-1))^{-\f13}}>0. 
\end{align*}
\item[\textbf{b)}] If $|k|\geq \nu^{-\f12}$ (i.e., $\nu^{\f13}|k|^\f23\geq 1$), then
\begin{align*}
|Z_-|
=&\Big|L_-\Big( \frac{\mathcal{V}(-1) -\Re\ \lambda-i\Im\ \lambda-i\nu k}{\mathcal{V}'(-1)}  \Big)\Big|\geq \frac{\nu^{-\f13}|k|^\f13}{|\mathcal{V}'(-1)|^{\f23}}|\Im\ \lambda+\nu k|
\end{align*}
Since we assume that $k\Im\ \lambda\geq -\delta \nu^\f13|k|^\f23$. We have that 
\begin{align}\label{hi_k_Z_-}
\nu^{-\f13}k^{\f13}(\Im\ \lambda+\nu k)\geq -\delta +\nu^{\f23} k^{\f43}\geq \frac{\nu^{\f23} k^{\f43}}{2}\;\leadsto\; |Z_-|\geq \frac{\nu^{\f23} k^{\f43}}{2|\mathcal{V}'(-1)|^{\f23}}.
\end{align}
Hence,
\begin{align*}
\Re\ m_-(Z_-,L_-,|k|)\geq \frac{c(1+|\mathcal{V}'(-1)|^{-\f13}\nu^\f13|k|^\f23)}{2(c(1+|\mathcal{V}'(-1)|^{-\f13}\nu^\f13|k|^\f23)+|\mathcal{V}'(-1)|^{-\f13}\nu^{\f13}|k|^\f23)}>c_0>0. 
\end{align*}
\end{itemize}
The same holds for $m_+$ by conjugation. This concludes the estimate for the $\Re\ m_\pm $.

\medskip
\noindent\textbf{Step 3: Estimates for $n_-$.}
We distinguish between two cases. 
 
 If $1\leq |k|\leq \nu^{-\f12}$ (i.e., $|k|/L_-\leq \mathcal{V'}(-1)^{-\frac{1}{3}}$), we recall the definition of $n_-$ \eqref{eq:def-mnminus}, and the bound \eqref{eq:al-bnd} to obtain that 
\begin{align*} 
|n_-(Z_-,L_-,k)|&\leq |\alpha(Z_-,2L_-)|+e^{-2|k|}\lf|\frac{|k|\,}{L_-(1-e^{-4k})}\int_0^{2L_-}\big(e^{kt/L_-}+e^{-kt/L_-}\big)\,\al(Z_-,t)\,dt\rg|\\
&\leq e^{-cL_-^{3/2} }+\frac{C|k|e^{-2|k|}}{L_-}\int_0^{\infty}(e^{Ct}+1)e^{-c t^\f32}dt\leq CL_-^{-1}e^{-c|k|}.
\end{align*}

If $|k|\ge\nu^{-\f12}$, then the estimate \eqref{eq:al-bnd} and the constraint of $Z_-$ \eqref{hi_k_Z_-} in this region yield that
\[
|\alpha(Z_-,t)|\le e^{-c(1+|Z_-|^\f12)t}\leq  e^{-c(1+(\mathcal{V}'(-1))^{-\f13} \nu^\f13|k|^\f23)t}.
\]
Therefore,
\[
\begin{aligned}
&|n_-(Z_-,L_-,k)|
\le
e^{-2ck}
+\frac{e^{-2k}}{1-e^{-4k}}
\int_0^{2k}\big(e^{(1-c)s}+e^{-(1+c)s}\big)\,ds.
\end{aligned}
\]
The second term is bounded by $C e^{-ck}$ for some $c>0$. Hence, $
|n_-(Z_-,L_-,k)|\le C e^{-ck}.$
Since $k\geq \nu^{-1/2}$, $
|n_-(Z_-,L_-,k)|\le C L_-^{-1}e^{-c|k|}$.
The estimate for $n_+$ follows by the similar argument and hence is omitted.  Combining Steps 1--3, and choosing $C_0$ large enough, complete the proof.
\end{proof}

\subsection{Factorization of $A_\pm,B_\pm$}

Define
\begin{equation}\label{eq:def-Skappa}
S_\pm
:=(1-\kappa)\bigl(1+|Z_\pm|^{\f12}\bigr)+\kappa L_\pm^{-1}.
\end{equation}

\begin{lemma}\label{lem:AkBk-factorization}
Assume Lemmas \ref{Lem-error}, \ref{lem:A0-traces}, and \ref{lem:airy-bounds}.
Then one may write
\begin{align}
    A_+ 
    &= e^{i\pi/6}\overline{A_0(Z_+)}\,E_+ ,\qquad B_-  = e^{-i\pi/6}A_0(Z_-)\,E_-,
    \label{eq:Aplus-factor}\\
    A_-
    &= e^{-i\pi/6}A_0(Z_-)\,R_-,\qquad   B_+
    = e^{i\pi/6}\overline{A_0(Z_+)}\,R_+,
    \label{eq:Aminus-factor}
\end{align}
where
\begin{align}
    E_+
    &=
    (1-\kappa)\overline{q(Z_+)}
    +\kappa L_+^{-1}m_+(Z_+,L_+,k)
    +\mathcal O(\kappa L_+^{-2}),
    \label{eq:Eplus}\\
    E_-
    &=
    (1-\kappa)q(Z_-)
    +\kappa L_-^{-1}m_-(Z_-,L_-,k)
    +\mathcal O(\kappa L_-^{-2}),
    \label{eq:Eminus}
\end{align}
and
\begin{equation}\label{eq:R-bounds}
    |R_+|
    \le C\bigl((1-\kappa)L_+^{-1}+\kappa L_+^{-2}\bigr),
    \qquad
    |R_-|
    \le C\bigl((1-\kappa)L_-^{-1}+\kappa L_-^{-2}\bigr).
\end{equation} Here, the parameter $c_0$ is defined in \eqref{eq:mnbounds-new}. 
In particular,
\begin{equation}\label{eq:E-bounds}
    c S_+ \le |E_+| \le C S_+,
    \qquad
    c S_- \le |E_-| \le C S_- .
\end{equation}
\end{lemma}

\begin{proof}
We give a detailed bookkeeping proof. We use the abbreviation
\[
I_+(W):=\int_{-1}^1K_+(y)W(y)\,dy,\qquad
I_-(W):=\int_{-1}^1K_-(y)W(y)\,dy.
\]

\smallskip
\noindent\textbf{Step 1: Estimation of  $A_+, B_-$.}
Using $W_+=W_{+;a}+W_{+;\e}$,
\[
A_+=(1-\kappa)\big(W_{+;a}(1)+W_{+;\e}(1)\big)
+\kappa\big(I_+(W_{+;a})+I_+(W_{+;\e})\big).
\]
By Lemma \ref{Lem-error}, $W_{+;\e}(1)=0$.
By Lemma \ref{lem:A0-traces},
\[
W_{+;a}(1)=e^{i\pi/6}\overline{A_0(Z_+)}\,\overline{q(Z_+)},
\qquad
I_+(W_{+;a})=e^{i\pi/6}L_+^{-1}\overline{A_0(Z_+)}\,m_+(Z_+,L_+,k).
\]
Since $0\le  K_+\le1$, $|I_+(W_{+;\e})|\le \|W_{+;\e}\|_{L^1}$, hence by Lemma \ref{Lem-error},
$|I_+(W_{+;\e})|\le C L_+^{-2}|A_0(Z_+)|.$ In conclusion,
\[
A_+
=
e^{i\pi/6}\overline{A_0(Z_+)}
\Big((1-\kappa)\overline{q(Z_+)}+\kappa L_+^{-1}m_+(Z_+,L_+,k)\Big)
+\mathcal O(\kappa L_+^{-2}|\overline{A_0(Z_+)}|),
\]
which yields the identity of $A_+$ in \eqref{eq:Aplus-factor} and \eqref{eq:Eplus}.

The estimate of $B_-$ is similar and hence omitted. 

\smallskip
\noindent\textbf{Step 2: Estimation of  $A_-, B_+$.}
Expand
\[
A_-=(1-\kappa)\big(W_{-;a}(1)+W_{-;\e}(1)\big)
+\kappa\big(I_+(W_{-;a})+I_+(W_{-;\e})\big).
\]
By Lemma \ref{Lem-error} and \eqref{eq:wall-traces}, 
\[
|W_{-;a}(1)|\le C e^{c|Z_-+2L_-|^{\f12}L_-}L_-^{-1}|A_0(Z_-)|,\qquad
|W_{-;\e}(1)|\le \|W_{-;\e}\|_{L^\infty}\le C L_-^{-1}|A_0(Z_-)|,
\]
so $(1-\kappa)|W_-(1)|\le C(1-\kappa)L_-^{-1}|A_0(Z_-)|$.
By Lemma \ref{lem:A0-traces} and Lemma \ref{lem:airy-bounds},
\[
I_+(W_{-;a})=-e^{-i\pi/6}L_-^{-1}A_0(Z_-)\,n_-(Z_-,L_-,k),
\qquad
|n_-|\le C e^{-c_0|k|}L_-^{-1},
\]
hence $|I_+(W_{-;a})|\le C e^{-c_0|k|} L_-^{-2}|A_0(Z_-)|$.
Also $|I_+(W_{-;\e})|\le \|W_{-;\e}\|_{L^1}\le C L_-^{-2}|A_0(Z_-)|$.
Thus
\[
|A_-|\le C|A_0(Z_-)|\big((1-\kappa)L_-^{-1}+\kappa L_-^{-2}\big),
\]
which yields the identity of $A_-$ in \eqref{eq:Aminus-factor} and the bound for $R_-$. Similar arguments yield bounds for $B_+$. 

\smallskip
\noindent\textbf{Step 3: Bounds for $E_\pm$.}
We show \eqref{eq:E-bounds} for $E_+$; the $E_-$ case is identical.
From \eqref{eq:Eplus},
\[
E_+=(1-\kappa)\overline{q(Z_+)}+\kappa L_+^{-1}m_+(Z_+,L_+,k)+\mathcal O(\kappa L_+^{-2}).
\]
Using Lemma \ref{ratio} and \eqref{eq:mnbounds-new}, the upper bound follows from $|q(Z_+)|\le C(1+|Z_+|^{1/2})$, $|m_+|\le C_0$.
For the lower bound, using real parts (to avoid cancellation),
\[
\Re \overline{q(Z_+)}=\Re q(\overline{Z_+})\ge c(1+|Z_+|^{1/2}),\qquad
\Re m_+\ge c_0,
\]
hence
\[
\Re E_+\ge c(1-\kappa)(1+|Z_+|^{1/2})+c_0\kappa L_+^{-1}-C\kappa L_+^{-2}.
\]
Since $L_+^{-2}=L_+^{-1}\cdot L_+^{-1}$ and $L_+^{-1}\le (\nu/\min\{\mathcal V'(1),\mathcal V'(-1)\})^{1/3}$
uniformly for $|k|\ge1$, by shrinking $\nu_0$ (as in Lemma \ref{lem:airy-bounds}) we ensure
$C L_+^{-1}\le c_0/2$ and therefore $C L_+^{-2}\le (c_0/2)L_+^{-1}$.
This yields
\[
\Re E_+\ge c(1-\kappa)(1+|Z_+|^{1/2})+\frac{c_0}{2}\kappa L_+^{-1}\ge c\,S_+,
\]
and hence $|E_+|\ge \Re E_+\ge cS_+$.
The upper bound $|E_+|\le CS_+$ follows from $|q|,|m_+|$ bounds.
The proof for $E_-$ is identical.
\end{proof}

\subsection{Nonvanishing of the Evans Function for all $\kappa\in[0,1]$}

\begin{proposition}\label{prop:Evans-nonzero}
Assume Lemma \ref{lem:AkBk-factorization} and the small-viscosity / bounded-frequency band
\[
0<\nu\le \nu_0,\qquad 1\le |k|,\qquad k\Im\ \lambda\geq -\delta  \nu^\f13 |k|^{\f23} .
\]
Then
\begin{equation}\label{eq:Evans-factorization}
D(\lambda,k)=\overline{A_0(Z_+)}\,A_0(Z_-)\Big(E_+ E_- - R_+ R_-\Big),
\end{equation}
and there exists $c>0$ such that
\begin{equation}\label{eq:Evans-lower-bound}
|D(\lambda,k)|
\ge
c\,|A_0(Z_+)A_0(Z_-)|
\,S_+ S_-.
\end{equation}
In particular, in the resolvent region $k\Im\la\ge -\delta\nu^{\f13}|k|^{\f23}$,
\[
D_\kappa(\lambda,k)\neq0\qquad\text{for all }\kappa\in[0,1].
\]
\end{proposition}

\begin{proof}
The factorization follows immediately from \eqref{eq:Evans} and Lemma \ref{lem:AkBk-factorization}.
By \eqref{eq:E-bounds}, $|E_+ E_-|\ge c S_+ S_-$.
By \eqref{eq:R-bounds} and $S_\pm\ge \kappa L_\pm^{-1}$,
\[
|R_+ R_-|
\le C L_+^{-1}L_-^{-1}S_+ S_-.
\]
Since $L_\pm^{-1}=(\nu/(|k|\mathcal V'(\pm1)))^{1/3}\le C\nu^{1/3}$ for all $|k|\ge1$, we have
\[
L_+^{-1}L_-^{-1}\le C\nu^{2/3}.
\]
Shrinking $\nu_0$ if necessary we may assume $C\nu^{2/3}\le c/2$, hence
\[
|R_+ R_-|
\le C L_+^{-1}L_-^{-1} S_+ S_-
\le \frac{c}{2} S_+ S_-,
\]
which yields the desired lower bound on $|E_+ E_--R_+ R_-|$.
\end{proof}

The estimate for the coefficients $C_{(\cdots)}^{(\cdots)}$ are summarized in the following.

\begin{lemma}\label{lem:Cpm} The following estimates hold
\begin{align}\label{Cij_est}\begin{split}
&|C_+^+|+L_-|C_-^+|\leq \frac{C}{
 |A_0(Z_+) |
\lf( (1-\kappa)(1+|Z_+|^\f12)+\kappa L_+^{-1}\rg)  }, \\
\quad &L_+|C_+^-|+|C_-^-|\leq\frac{C}{
 |A_0(Z_-) |
\lf( (1-\kappa)(1+|Z_-|^\f12)+\kappa L_-^{-1}\rg)  } .\end{split}
\end{align}
\end{lemma}
\begin{proof}
    We invoke Lemma \ref{lem:AkBk-factorization} and the estimate  \eqref{eq:Evans-lower-bound} to obtain the following
\begin{align*}
|C_+^+|+L_-|C_-^+|\leq& \frac{1}{|D|}(|B_-|+L_-|A_-|)\leq \frac{C}{
|A_0(Z_+)A_0(Z_-)|
\,S_+ S_-}|A_0(Z_-)|\lf(S_-+(1-\kappa)+\kappa L_-^{-1}\rg)\\
\leq &\frac{C}{
|A_0(Z_+) |
\,S_+  } =\frac{C}{
 |A_0(Z_+) |
\lf( (1-\kappa)(1+|Z_+|^\f12)+\kappa L_+^{-1}\rg)  } .
\end{align*}
Here, the $S_\pm = (1-\kappa)(1+|Z_\pm|^\f12)+\kappa L_\pm^{-1}$. The proof for the remaining part is similar and hence omitted. 
\end{proof}
Finally, we are able to derive the estimate for the unit boundary correctors

With above preparations done, we now give the proof of boundary correctors $w_+$ and $w_-$.
\begin{lemma}\label{lem:w1w2} Under the assumption
\begin{align}\label{epsilon_0}
0<\nu\le \nu_0,\qquad 1\le |k|,\qquad k{\rm Im}\ \lambda\geq -\delta \nu^\f13 |k|^{\f23} .
\end{align}
We have the following estimates:
\begin{subequations} \label{wpm_est}
    \begin{align}
        \label{west_1}&\|w_\pm\|_{L^{\infty}}\leq         \frac{ CL  }{\displaystyle (1-\kappa )L +\kappa\br{Z_\pm}^{-1/2}        },\qquad 
        \|w_\pm\|_{L^1}\leq \frac{C}{\displaystyle L(1-\kappa)\br{Z_\pm}^{1/2}
     +\kappa },\\
   \label{west_3}  & \|w_{\pm}\|_{L^2}\leq \f{C\br{Z_\pm}^{\f14}L^\f12}{ L(1-\kappa)\br{Z_\pm}^{\f12}
   +\kappa},
   \\
     \label{west_4} & \|(1-|y|)w_\pm\|_{L^2} +\|(ik,\pa_y)\phi_\pm\|_{L^2}\le   \frac{C}{L^{\f32} (1-\kappa)\br{Z_\pm}^\f34+\kappa L^{\f12}\br{Z_\pm}^\f14}      .
    \end{align} \end{subequations}
    \end{lemma}
\begin{proof}
    Throughout the proof, we invoke the $W_{\pm}$-estimates in \eqref{W_pm_est} and the $C_\pm^+,\ C_\pm^-$-estimates \eqref{Cij_est}. Without loss of generality, we focus on the estimate for $w_+$.
  
    \step{1.} 
 We start with the decomposition
\begin{align}\label{w_+_dcmp}
    \|w_+\|_{L^\infty}
    &\leq |C_+^+ |\|W_{+}\|_{L^\infty}+|C_+^{-}|\|W_{-}\|_{L^\infty}=:T_1+T_2.
\end{align}
Recall that \(\mathcal V\) is close to the Couette profile \(y\) in the sense that $\mathcal V'(\pm1)\sim1,\, L_\pm\sim L,
\, |\mathcal V(1)-\mathcal V(-1)|\lesssim1.$ 
Thanks to the definitions of \(Z_\pm\) \eqref{eq:def-zpm}, we obtain \begin{align*}
    |Z_-|^2
    =\frac{k^{\f23}}{\nu^\f23}\frac{|\mathcal{V}(-1)-\lambda_r|^2+|\lambda_i+\nu k|^2}{|\mathcal{V}'(-1)|^{\f43}},\quad 
     |Z_+|^2
    =\frac{k^{\f23}}{\nu^\f23}\frac{|\mathcal{V}(1)-\lambda_r|^2+|\lambda_i+\nu k|^2}{|\mathcal{V}'(1)|^{\f43}}.
    \end{align*} Hence, 
\[
|Z_\pm|
\approx
L\left(|\mathcal V(\pm 1)-\lambda_r|^2+|\lambda_i+\nu k|^2\right)^\f12\approx L\left(|\mathcal V(\pm 1)-\lambda_r| +|\lambda_i+\nu k| \right)  .
\]
Since 
$|\mathcal V(-1)-\lambda_r|
\le
|\mathcal V(1)-\lambda_r|+|\mathcal V(1)-\mathcal V(-1)|$, we obtain $
|Z_-|\le C(|Z_+|+L).$ 
Consequently,
\begin{align*}
\br{Z_-}^{1/2}
=(1+|Z_-|^2)^{1/4}
\le
C(1+|Z_+|^2+L^2)^{1/4}
\le
C(\br{Z_+}^{1/2}+ L^{1/2})
\le
CL^{1/2}\br{Z_+}^{1/2},
\end{align*}
where we used \(L\ge1\) and \(\br{Z_+}^{1/2}\ge1\). A symmetric argument yields that
\begin{align}
C^{-1}L^{-1/2}\br{Z_-}^{1/2}\leq\br{Z_+}^{1/2} 
\le
CL^{1/2}\br{Z_-}^{1/2},
    \label{Zpm_rel}
\end{align}

Next we consider the $T_1$-term in \eqref{w_+_dcmp}. By invoking the $C^\pm_+$-estimate \eqref{Cij_est}, and the $W_\pm$-estimate \eqref{W_pm_est} in Lemma \ref{W12}, we obtain 
\begin{align}
T_1\leq |C_+^+| \lf\| W_{+}\rg\|_{L^{\infty}}\le \frac{C(1+|Z_+|)^{\f12}|A_0(Z_+)|}{|A_0(Z_+)|\lf((1-\kappa)\br{Z_+}^\f12+\kappa L_+^{-1}\rg)}\le \frac{C L_+}{ (1-\kappa)L_++\kappa \br{Z_+}^{-\f12}}.
\end{align} Here, we applied $1+b\approx (1+b^2)^\f12= \br b$ for $b\geq 0$. 

Finally, we consider the $T_2$-term in \(w_+\). By the extra \(L^{-1}\) gain in \(C_+^-\) \eqref{Cij_est},
\[
|C_+^-|
\le
\frac{C}{L\,|A_0(Z_-)|\big((1-\kappa)\br{Z_-}^{1/2}+\kappa L_-^{-1}\big)}.
\]
Together with the bound $ \lf\| W_{-}\rg\|_{L^{\infty}}\le C\br {Z_-}^{\f12}|A_0(Z_-)|$ \eqref{W_pm_est},
this yields
\[
|C_+^-|\|W_-\|_{L^\infty}
\le
\frac{C L^{-1}\br{Z_-}^{1/2}}{(1-\kappa)\br{Z_-}^{1/2}+\kappa L_-^{-1}}
\le
\frac{C}{(1-\kappa)L+\kappa \br{Z_-}^{-1/2}}.
\]
Since \(\br{Z_-}^{1/2}\le CL^{1/2}\br{Z_+}^{1/2}\), we have  
$
(1-\kappa)L+\kappa \br{Z_+}^{-1/2}
\le
CL^{1/2}\big((1-\kappa)L+\kappa \br{Z_-}^{-1/2}\big)$, and 
\[T_2=
|C_+^-|\|W_-\|_{L^\infty}\lesssim \frac{1}{(1-\kappa)L+\kappa \br{Z_-}^{-1/2}}
\lesssim
\frac{L^{1/2}}{(1-\kappa)L+\kappa \br{Z_+}^{-1/2}}.
\] Hence, we see that the contribution is much smaller than $T_1. $

Combining the estimates yields $
\|w_+\|_{L^\infty}
\le
\frac{CL}{(1-\kappa)L+\kappa\br{Z_+}^{-1/2}}. $ 
Similarly, one obtain the estimate for 
$\|w_-\|_{L^\infty}$.

    \step{2.} We decompose the term as follows 
    \begin{align*}
   \|w_+\|_{L^1}
     &\leq  |C_+^+ | \|W_{+}\|_{L^1} + |C_+^{-}| \|W_{-}\|_{L^1}=:T_3+T_4.
     \end{align*}
     
     The estimate for the $T_3$-term is direct. By invoking the $W_{\pm}$-estimates in \eqref{W_pm_est} and the $C_\pm^+,\ C_\pm^-$-estimates \eqref{Cij_est}, we obtain
     \begin{align*} 
     T_3&\leq \frac{C}{|A_0(Z_+)|( (1-\kappa)(1+|Z_+|^\f12)+\kappa L_+^{-1})}|A_0(Z_+)|L^{-1}\leq  \frac{C}{L(1-\kappa) (1+|Z_+|)^{\f12}+\kappa}. 
    \end{align*}
    For the $T_4$-term, we further invoke the relation \eqref{Zpm_rel} to obtain
    \begin{align*}
     T_4\leq \frac{CL^{-1}}{|A_0(Z_-)|( (1-\kappa)\br{Z_-}^\f12+\kappa L_-^{-1})}|A_0(Z_-)|L^{-1}
   \leq \frac{CL^{-1/2}}{L(1-\kappa) \br{Z_+}^{\f12}+\kappa}. 
     \end{align*}
     Hence, we obtain that 
      $
   \|w_+\|_{L^1}
     \leq  \frac{C}{L(1-\kappa) \br{Z_+}^{\f12}+\kappa}. 
   $ 
    The estimate of $\|w_-\|_{L^1}$ is similar. 
    
  \step{3.} Here, we apply the interpolation to obtain the following estimates
  \begin{align*}
\|w_{\pm}\|_{L^2}\leq& \|w_{\pm}\|_{L^1}^\f12\|w_{\pm}\|_{L^\infty}^\f12
\leq  \lf(\frac{1}{L(1-\kappa)\br{Z_\pm}^{\f12}+\kappa}\rg)^\f12 \lf(\frac{ L}{(1-\kappa )L+\kappa\br{Z_\pm}^{-\f12}}\rg)^\f12.
  \end{align*}
One observe that the estimate is equivalent to the result \eqref{west_3}.

\step{4.} The estimates of $\|(1-|y|)w_+\|_{L^2}+\|(ik,\pa_y)\phi_+\|_{L^2}$ are similar to the estimates presented before. One can invoke the established bounds  \eqref{W_pm_est}, \eqref{Cij_est}  and \eqref{Zpm_rel} to derive the bounds. We omit further details. 

This completes the proof.
\end{proof}

Since in this paper, we consider the case that $1-\ka\ge c$, $c$ is a constant independent of $\nu$ and $k$, then as a corollary of Lemma \ref{lem:w1w2}, we have the following estimates.

\begin{corol}\label{w1w2} For any $\al=\frac{1-\kappa}{\kappa}>0$, there exist $k_0$, $\nu_0$ and $\d_0$ such that for $0<\nu<\nu_0$,  there holds
    \begin{align*}
        &\|w_+\|_{L^{\infty}}+\|w_-\|_{L^{\infty}}\leq C,\quad
       \|w_{\pm}\|_{L^1}\leq C\nu^{\f13}|k|^{-\f13}(1+|Z_{\pm}|)^{-\f12},\\
&      \|w_{\pm}\|_{L^2}\leq C\nu^{\frac{1}{6}}|k|^{-\f16}(1+|Z_{\pm}|)^{-\f14},
          \ \|(ik,\pa_y)\phi_{\pm}\|_{L^2}+\|(1-|y|)w_{\pm}\|_{L^2}\le  C\nu^{\f12}|k|^{-\f12}(1+|Z_{\pm}|)^{-\f34}.
    \end{align*}
\end{corol}

\section{Auxiliary Lemmas}
In this section, we study the one-dimensional heat equation with Robin-type boundary conditions and introduce two fundamental lemmas. 
\begin{lemma}\label{Lem-U}
    For $\wt U $ satisfying \eqref{eq-L-H-Z} initiated from consistent initial data \eqref{eq-v_in}, the following estimate holds
    \begin{align}
        \|\mathcal{U}-y\|_{L_t^{\infty}C_y^3}=\|\wt{U}\|_{L^\infty_tC_y^3}\lesssim\ep\nu^{\f13}.\label{ineq-U}
    \end{align}
\end{lemma}
\begin{proof}
   The equation \eqref{eq-L-H-Z} can be solved as follows
\begin{align*}
   & \wt U (t,y)=\sum_{n=1} a_n e^{-\nu(k_n^e)^2t}\cos (k_n^ey)+\sum_{n=1}b_n e^{-\nu(k_n^o)^2t}\sin (k_n^oy),\\
   & \tan k_n^e=(\al k_n^e)^{-1}, \quad \tan k_n^o=-\al k_n^o,\quad a_n=\frac{\langle v_{0,\rm in}^1,\cos(k_n^e\cdot)\rangle}{\|\cos(k_n^e\cdot)\|_{L^2}^2},\quad b_n=\frac{\langle v_{0,\rm in}^1,\sin(k_n^o\cdot)\rangle}{\|\sin(k_n^o\cdot)\|_{L^2}^2}.
\end{align*}
By \eqref{eq-L-H-Z}, we can derive that $\pa_y^2v_{0,\rm in}^1(\pm1)=\mp\alpha\pa_y^2v_{0,\rm in}^1(\pm1)$, $\pa_y^4v_{0,\rm in}^1(\pm1)=\mp\alpha\pa_y^5v_{0,\rm in}^1(\pm1)$. Then through integration by parts, and the fact that  $\tan k_n^e=(\al k_n^e)^{-1}$, $\tan k_n^o=-\al k_n^o$, we get
\begin{align*}
    \langle v_{0,\rm in}^1(y),\cos(k_n^ey)\rangle&=-\f{\sin k_n^e}{(k_n^e)^3}\big(\pa_y^2v_{0,\rm in}^1(1)+\pa_y^2v_{0,\rm in}^1(-1)\big)-\f{\cos k_n^e}{(k_n^e)^4}\big(\pa_y^3v_{0,\rm in}(1)-\pa_y^3v_{0,\rm in}^1(-1)\big)\\
    &\quad+\f{1}{(k_n^e)^5}\langle \pa_y^4v_{0,\rm in}^1(y),\pa_y(\sin(k_n^ey)\rangle=-\frac{\langle\cos (k_n^ey),\pa_y^6v_{0,\rm in}^1(y)\rangle}{(k_n^e)^6},\\
    \langle v_{0,\rm in}^1(y),\sin(k_n^oy)\rangle&=-\frac{1}{(k_n^o)^6}\int_{-1}^1\sin (k_n^oy)\pa_y^6v_{0,\rm in}^1(y)dy=-\frac{\langle \sin (k_n^oy),\pa_y^6v_{0,\rm in}^1(y)\rangle}{(k_n^o)^6}.
\end{align*}
Thus
\begin{align*}
    |a_n|&=\lf|\frac{\langle\cos (k_n^ey),\pa_y^6v_{0,\rm in}^1 (y)\rangle}{(k_n^e)^6\|\cos(k_n^ey)\|_{L^2}^2}\rg|\lesssim\frac{\|\cos (k_n^ey)\|_{L^2}\|\pa_y^6v_{0,\rm in}^1 (y)\|_{L^2}}{(k_n^e)^6\|\cos(k_n^e y)\|_{L^2}^2}\lesssim |k_n^e|^{-6}\|\pa_y^6v_{0,\rm in}^1 \|_{L^2},\\
    |b_n|&=\lf|\frac{\langle \sin (k_n^oy),\pa_y^6v_{0,\rm in}^1(y)\rangle}{(k_n^o)^6\|\sin(k_n^oy)\|_{L^2}^2}\rg|\lesssim\frac{\|\sin (k_n^oy)\|_{L^2}\|\pa_y^6v_{0,\rm in}^1(y)\|_{L^2}}{(k_n^e)^6\|\sin(k_n^oy)\|_{L^2}^2}\lesssim |k_n^o|^{-6}\|\pa_y^6v_{0,\rm in}^1 \|_{L^2}.
\end{align*}
Moreover, $\tan k_n^e=(\al k_n^e)^{-1}$, $\tan k_n^o=-\al k_n^o$ show that $k_n^o\sim n\pi-\pi/2$, $k_n^e\sim n\pi$. Then we can conclude that
\begin{align*}
    |\wt{U}|&\leq \sum_{n=1}|a_n|e^{-\nu(k_n^e)^2t}|\cos(k_n^ey)|+\sum_{n=1}|b_n|e^{-\nu(k_n^o)^2t}|\sin(k_n^oy)|\lesssim \|\pa_y^6v_{0,\rm in}^1 \|_{L^2}.
\end{align*}
Similarly, \begin{align*}
 |\pa_y\wt U |+|\pa_y^2\wt U |+|\pa_y^3\wt{U}|\lesssim \|\pa_y^6v_{0,\rm in}^1 \|_{L^2}.
\end{align*}
Since $\|v_{\rm in} \|_{H^6}\le \ep\nu^{\f13}$, we finish the proof.
\end{proof}

\begin{lemma}[Lemma 9.3 in \cite{ChenLiWeiZhang18}]\label{basic-interpo} If $(\pa_y^2-k^2)\phi=w$, $\phi(\pm 1)=0$, $|k|\geq 1$, then
\begin{align*}
&    \|\pa_y\phi\|_{L^2}^2+|k|^2\|\phi\|_{L^2}^2=\langle w,-\phi\rangle\lesssim |k|^{-1}\|w\|_{L^1}^2,\\
  &  \|\pa_y\phi\|_{L^{\infty}}+|k|\|\phi\|_{L^{\infty}}\lesssim\|w\|_{L^1},\qquad \|\pa_y\phi\|_{L^{\infty}}+|k|\|\phi\|_{L^{\infty}}\lesssim |k|^{-\f12}\|w\|_{L^2}.
\end{align*}
\end{lemma}
\begin{lemma} For $k\ne0$, and $ u _k=(u^1_k,u^2_k)=(-\pa_y\p_k,ik\p_k)$, $\p_k(\pm1)=0$, it holds
    \begin{equation}\label{eq-weig}
        \|(1-|y|)^{-1}u^2_k\|_{L^2L^{\infty}} \lesssim |k|\|\pa_y\p_k\|_{L^2L^{\infty}}\lesssim|k|\|\pa_y\p_k\|_{L^2L^2}^{\f12}\|\pa_y^2\p_k\|_{L^2L^2}^{\f12}.
    \end{equation}
\end{lemma}
\begin{proof}
            For $k\ne 0$, direct calculations show
        \begin{equation*}
    \begin{aligned}
        \|(1-|y|)^{-1}u^2_k\|_{L^2L^{\infty}}^2
       &\leq\int_0^{+\infty}\max\Big\{\sup_{y\in[-1,0]}\f{|\int_{-1}^y\pa_zu_k^2(t,z)dz|^2}{(1+y)^2},\sup_{y\in[0,1]}\f{|\int_y^1\pa_zu_k^2(t,z)dz|^2}{(1-y)^2}\Big\}dt\\
       &\lesssim \int_0^{+\infty}\|\pa_yu_k^2\|_{L^{\infty}}^2dt \lesssim |k|^2\|\pa_y\p_k\|_{L^2L^{\infty}}^2.
    \end{aligned}
    \end{equation*}
    Then after applying the Gagliardo-Nirenberg inequality, we finish the proof.
\end{proof}

\section{Properties of the Singular Integral Operators}\label{sec-operator}
In this section, we study the properties of the singular integral operators defined in Definition \ref{defn:SIO}. The following lemma is established as a general tool to obtain the boundedness of the singular integral operators associated with kernels satisfying specific properties. 
\begin{lemma}\label{Lem-S}
Let $\mathcal{U}(t,y)$ satisfies $\mathcal{U}\in L_t^\infty C^3_{y}$, and $\inf_{t,y}|\mathcal{U}'(t,y)|\geq c_1>0.$ 
    Consider the singular integral operator $\mathfrak{S}_k$, associated with the function $\mathcal{G}_k$, defined as follows
    \begin{align}\label{general-SIO}
        \mathfrak{S}_k[f](y):=kP.V.\int_{-1}^1\frac{\mathcal{G}_k(y,y')}{2i\big(\mathcal{U}(t,y)-\mathcal{U}(t,y')\big)}f(y')dy'.
    \end{align}
  Further assume that the function $\mathcal{G}_k\in L^{\infty}_{y}L_{y'}^1\cap L^{\infty}_{y'}L_{y}^1\cap W^{1,\infty}_{y,y'}$ and $\mathcal{G}_k(y,y')=\mathcal{G}_k(y',y)$. Then, the following estimate holds, for any $\epsilon_s\in(0,1/2)$, 
    \begin{align*}
        \|\mathfrak{S}_k\|_{L^2\to L^2}\leq C\big( \ep_s ^{-1}|k| \lf \|\mathcal{G}_k(y,y')\rg\|_{L_y^{\infty}L_{y'}^1\cap L_{y'}^{\infty}L_{y}^{1}}+\ep_s |k|\|\na_{y,y'}\mathcal{G}_k\|_{L_{y,y'}^{\infty}}+ \|k\mathcal{G}_k(y,y)\|_{L^{\infty}}\big).
    \end{align*}
    Here, $C$ is a constant depending on $\|\mathcal{U}\|_{L_t^\infty C^1_y}$ and $c_1$, but independent of $|k|$.
\end{lemma}
\begin{proof}
    First, we decompose the expression as follows
       \begin{align*}
       & \mathfrak{S}_k[f_k](y)
        = k \int_{[-1,1]\backslash(y-\ep_s ,y+\ep_s )}\frac{\mathcal{G}_k(y,y')}{2i\big(\mathcal{U}(t,y)-\mathcal{U}(t,y')\big)}f_k(y')dy'\\
         &\quad
         +kP.V. \int_{y-\ep_s }^{y+\ep_s }\frac{\mathcal{G}_k(y,y')-\mathcal{G}_k(y,y)}{2i\big(\mathcal{U}(t,y)-\mathcal{U}(t,y')\big)}f_k(y')dy' + k \mathcal{G}_k(y,y)P.V.\int_{y-\ep_s }^{y+\ep_s }\frac{f_k(y')}{2i\big(\mathcal{U}(t,y)-\mathcal{U}(t,y')\big)}dy'\\
         &=:T_1+T_2+T_3.
        \end{align*}
        {\bf Bound of $T_1$:} We use a Schur-type estimate and the rapid decay of $\mathcal{G}_k(y,y')$ to gain integrability. Since
        \begin{align*}
            T_1
            &=k \int_{-1}^1\frac{\mathcal{G}_k(y,y'){\bf 1}_{\{|y-y'|\ge\ep_s \}}}{2i\big(\mathcal{U}(t,y)-\mathcal{U}(t,y')\big)}f_k(y')dy'=: \int_{-1}^1K_k(y,y')f_k(y')dy'.
        \end{align*}
  Moreover,
        \begin{align*}
            \|K_k\|_{L_{y'}^{\infty}L_y^1}+\|K_k\|_{L_{y}^{\infty}L_{y'}^{1}}
            &\lesssim  \Big\|\frac{k\mathcal{G}_k(y,y'){\bf 1}_{\{|y-y'|\ge\ep_s \}}}{|y-y'|}\Big\|_{L_{y'}^{\infty}L_y^1}+\Big\|\frac{k\mathcal{G}_k(y,y'){\bf 1}_{\{|y-y'|\ge\ep_s \}}}{|y-y'|}\Big\|_{L_{y}^{\infty}L_{y'}^1}\\
            &\lesssim \ep_s ^{-1}|k|\big(\|\mathcal{G}_k\|_{L_{y'}^{\infty}L_{y}^1}+\|\mathcal{G}_k\|_{L_{y}^{\infty}L_{y'}^1}\big).
        \end{align*}
        Then, by using Schur's test, it follows
      \begin{equation}\label{est-T1-G}
        \|T_1\|_{L^2}\leq C\big(\|K_k\|_{L_y^{\infty}L_{y'}^1}+\|K_k\|_{L_{y'}^{\infty}L_y^1}\big)\|f_k\|_{L^2}\le C\ep_s ^{-1}|k|\big(\|\mathcal{G}_k\|_{L_{y'}^{\infty}L_{y}^1}+\|\mathcal{G}_k\|_{L_{y}^{\infty}L_{y'}^1}\big)\|f_k\|_{L^2}.
        \end{equation}
        {\bf Bound of $T_2$:} In this step, Schur's test gives
        \begin{equation}\label{est-T2-G}
        \begin{aligned}
          \|T_2\|_{L^2}&\le C\Big(\lf\|k \frac{\mathcal{G}_k(y,y')-\mathcal{G}_k(y,y)}{\min_{z\in[y',y]}\{|\pa_z\mathcal{U}(t,z)|\}|y-y'|}{\bf 1}_{\{|y-y'|\le\ep_s \}}(y')\rg\|_{L_y^{\infty}L_{y'}^1}\\
          &\quad \qquad+\lf\| k\frac{\mathcal{G}_k(y,y')-\mathcal{G}_k(y,y)}{\min_{z\in[y',y]}\{|\pa_z\mathcal{U}(t,z)|\}|y-y'|}{\bf 1}_{\{|y-y'|\le\ep_s \}}(y')\rg\|_{L_{y'}^{\infty}L_{y}^{1}}\Big)\|f_k\|_{L^2}\\
          &\le C|k|\ep_s  \|\na_{y,y'}\mathcal{G}_k\|_{L_{y,y'}^{\infty}}\|f_k\|_{L^2}.
        \end{aligned}
        \end{equation}
        {\bf Bound of $T_3$:} Using the change of variables $\mathcal{U}(t,y)=s,$ $\mathcal{U}(t,y')=s'$, and denoting $s_+=\mathcal{U}(t,y+\ep_s)$, $s_-=\mathcal{U}(t,y-\ep_s)$, we get
        \begin{align*}
            P.V.\int_{y-\ep_s }^{y+\ep_s }\frac{f_k(y')}{\mathcal{U}(t,y)-\mathcal{U}(t,y')}dy'=P.V.\int_{s_- }^{s_+ }\frac{f_k(\mathcal{U}^{-1}(s'))}{s-s'}\pa_{s'}\big(\mathcal{U}^{-1}(t,s')\big)ds'.
        \end{align*}
Use the change of variables $s-s'=z$, and the properties of the maximal Hilbert transform and the Hilbert transform, then
        \begin{equation}\label{est-T3-G}
        \begin{aligned}
            &\|T_3\|_{L^2}\le \|k\mathcal{G}_k(y,y)\|_{L^{\infty}}\Big\|P.V.\int_{y-\ep_s }^{y+\ep_s }\frac{f_k(y')}{\mathcal{U}(t,y)-\mathcal{U}(t,y')}dy'\Big\|_{L^2_y} \\
            &=\|k\mathcal{G}_k(y,y)\|_{L^{\infty}}\Big\|P.V.\int_{s_- }^{s_+}\frac{f_k(\mathcal{U}^{-1}(s'))\pa_{s'}\big(\mathcal{U}^{-1}(t,s')\big)}{s-s'}ds'\big|\pa_s\big(\mathcal{U}^{-1}(t,s)\big)\big|^{\f12}\Big\|_{L^2_s}\\
            &\lesssim \|k\mathcal{G}_k(y,y)\|_{L^{\infty}}\Big(\Big\|\sup_{0<\d}\int_{|z|>\d}\frac{(f_k\circ\mathcal{U}^{-1})(s-z)\pa_z\big(\mathcal{U}^{-1}(t,s-z)\big)}{z}dz\Big\|_{L^2}\\
            &\qquad\qquad\qquad\qquad\quad+\|\mathcal{H}\big((f_k\circ\mathcal{U}^{-1})(\mathcal{U}^{-1})'\big)\|_{L^2}\Big)\lesssim \|k\mathcal{G}_k(y,y)\|_{L^{\infty}}\|f_k\|_{L^2_y},
            \end{aligned}
        \end{equation}
        where
      $
            \mathcal{H}(f):=P.V.\int_{\mathbb R}\frac{f(y-z)}{z}dz
      $
        is the Hilbert transform.

Combining the bounds \eqref{est-T1-G}, \eqref{est-T2-G}, \eqref{est-T3-G}, we finish the proof.
\end{proof}
Next, we apply Lemma \ref{Lem-S} to get the estimates of $\mathfrak{J}_k$ and $\w{\mathfrak{J}_k}{\rm e}$.
\begin{corol}\label{corol-bound}
    The singular integral operators $\mathfrak{J}_k$ and $\mathfrak{J}^{\rm (e)}_k$ extend to bounded linear operators on $L^2\to L^2$ and
    \begin{align*}
    \|\mathfrak{J}_k\|_{L^2\to L^2}\lesssim1,\qquad
        \|\mathfrak{J}^{\rm (e)}_k\|_{L^2\to L^2}\lesssim 1.
    \end{align*}
\end{corol}
\begin{proof}
    By taking $\mathcal{G}_k=G_k(y,y')$ and $\nu^{\f13}|k|^{\f23}\w G{\rm e}_k(y,y')$ in Lemma \ref{Lem-S}, respectively, after properly choosing the parameter $\ep_s $, we can obtain the bounds of $\mathfrak{J}_k$ and $\mathfrak{J}_k^{\rm (e)}$.

By the definitions of $G_k$ and $\w G{\rm e}_k$, we immediately get
\begin{align*}
  &  \|G_k\|_{L_y^{\infty}L_{y'}^1\cap L_{y'}^{\infty}L_y^1}
    \lesssim \sup_y\frac{|\sinh \big(k(1+y)\big)|}{|k||\sinh(2k)|}\int_y^1|\sinh\big(k(1-y')\big)|dy'\\
    &\qquad+\sup_{y'}\frac{|\sinh \big(k(1-y')\big)|}{|k||\sinh(2k)|}\int_{-1}^{y'}|\sinh\big(k(1+y)\big)|dy\\
    &\quad\lesssim \frac{1}{|k|}\sup_y\big(\int_y^1\f{e^{|k|(1+y)}e^{|k|(1-y')}}{e^{2|k|}}dy'\big)+\sup_{y'}\big(\int_{-1}^{y'}\f{e^{|k|(1+y)}e^{|k|(1-y')}}{e^{2|k|}}dy\big)\lesssim |k|^{-2},\\
&\|\pa_yG_k(y,y')\|_{L_{y,y'}^{\infty}}+\|\pa_{y'}G_{k}(y,y')\|_{L_{y,y'}^{\infty}}\lesssim 1,\qquad\quad   \|kG_k(y,y)\|_{L^{\infty}}\lesssim1.
\end{align*}
    By choosing $\ep_s =|k|^{-1}$ in Lemma \ref{Lem-S}, there holds
   $        \|\mathfrak{J}_k\|_{L^2\to L^2}\lesssim 1.
    $ 
Similarly, for $\nu^{\f13}|k|^{\f23}G^{\rm (e)}_k$, we can get the following estimates
\begin{align*}
   & \|\nu^{\f13}|k|^{\f23}G_k^{\rm (e)}\|_{L_y^{\infty}L_{y'}^1\cap L_{y'}^{\infty}L_y^1}\lesssim \sup_y\frac{|\sinh\big(\nu^{-\f13}|k|^{-\f23}|k|(y+1)\big)|}{|k||\sinh (2|k|)|}\int_y^1|\sinh\big(\nu^{-\f13}|k|^{-\f23}|k|(y'-1)\big)|dy'\\
    &\qquad\quad +\sup_{y'}\frac{|\sinh\big(\nu^{-\f13}|k|^{-\f23}|k|(y'-1)\big)|}{|k||\sinh (2|k|)|}\int_{-1}^{y'}|\sinh\big(\nu^{-\f13}|k|^{-\f23}|k|(y+1)\big)|dy\lesssim \nu^{\f13}|k|^{-\f43},\\
&    \|\nu^{\f13}|k|^{\f23}\pa_yG_k^{\rm (e)}\|_{L_{y,y'}^{\infty}}+\|\nu^{\f13}|k|^{\f23}\pa_{y'}G_k^{\rm (e)}\|_{L_{y,y'}^{\infty}}\lesssim \nu^{-\f13}|k|^{-\f23},\qquad     \|\nu^{\f13}|k|^{\f23}kG_k^{\rm (e)}(y,y)\|_{L^{\infty}}\lesssim 1.
\end{align*}
Finally, by choosing $\ep_s =\nu^{\f13}|k|^{-\f13}$ in Lemma \ref{Lem-S}, we obtain that
\begin{equation*}
    \begin{aligned}
        \|\mathfrak{J}_k^{\rm (e)}\|_{L^2\to L^2}
        \lesssim \ep_s ^{-1}|k| \nu^{\f13}|k|^{\f23}|k|^{-2}+\ep_s  |k| \nu^{-\f13}|k|^{-\f23}+1\lesssim1.
    \end{aligned}
\end{equation*}
This concludes the proof of the corollary. 
\end{proof}
The next lemma captures the commutators of $\pa_y$ with $\mathfrak{J}_k$ and $\mathfrak{J}_k^{\rm (e)}$.

\begin{lemma}\label{comm-bound}
      Assume that $f\in H^1(-1,1)$, then $\mathfrak{J}_k[f],\ \mathfrak{J}^{\rm (e)}_k[f]\in H^1(-1,1)$. The commutators
  $  [\pa_y,\mathfrak{J}_k]=\mathfrak{H}_k,\,   [\pa_y,\mathfrak{J}_k^{\rm (e)}]=\mathfrak{H}_k^{\rm (e)}$ have the following integral form
\begin{align*}
\mathfrak{H}_k[f]&=kP.V.\int_{-1}^1\frac{H_k(y,y')}{2i\big(\mathcal{U}(t,y)-\mathcal{U}(t,y')\big)}f(y')dy',\\
    \mathfrak{H}_k^{\rm (e)}[f]&=\nu^{\f13}|k|^{\f23}kP.V.\int_{-1}^1\frac{H^{\rm (e)}_k(y,y')}{2i\big(\mathcal{U}(t,y)-\mathcal{U}(t,y')\big)}f(y')dy',
\end{align*}
where $H_k$, $H_k^{\rm (e)}$ are continuous functions given by
\begin{equation}\label{comm-H-He}
\begin{aligned}
    H_k(y,y')&:=\frac{\sinh\big(k(y+y')\big)}{\sinh (2k)}+\frac{\pa_{y'}\mathcal{U}(t,y')-\pa_y\mathcal{U}(t,y)}{\mathcal{U}(t,y)-\mathcal{U}(t,y')}G_k(y,y'),\\
     H_k^{\rm (e)}(y,y')&:=\frac{\sinh\big(\nu^{-\f13}|k|^{-\f23}|k|(y+y')\big)}{\nu^{\f13}|k|^{\f23}\sinh(2\nu^{-\f13}|k|^{-\f23}|k|)}+\nu^{\f13}|k|^{\f23}\frac{\pa_{y'}\mathcal{U}(t,y')-\pa_y\mathcal{U}(t,y)}{\mathcal{U}(t,y)-\mathcal{U}(t,y')}G^{\rm (e)}_k(y,y').
\end{aligned}
\end{equation}
Moreover, the operators are bounded, i.e., 
$
    \|\mathfrak{H}_k\|_{L^2\to L^2}\lesssim |k|,\,  \|\mathfrak{H}^{\rm (e)}_k\|_{L^2\to L^2}\lesssim |k|.
$
\end{lemma}
\begin{proof}
For the general case $\mathfrak{S}_k$ \eqref{general-SIO}, integrating by parts and using the boundary conditions, we get
\begin{align*}
   &[\pa_y,\mathfrak{S}_k][f](y)=\pa_y\big(\mathfrak{S}_k[f]\big)(y)-\mathfrak{S}_k[\pa_yf](y)\\
&=kP.V.\int_{-1}^1\frac{(\pa_y+\pa_{y'})\mathcal{G}_k(y,y')}{2i\big(\mathcal{U}(t,y)-\mathcal{U}(t,y')\big)}f(y')dy'+kP.V.\int_{-1}^1\frac{\mathcal{G}_k(y,y')\big(\pa_{y'}\mathcal{U}(t,y')-\pa_y\mathcal{U}(t,y)\big)}{2i\big(\mathcal{U}(t,y)-\mathcal{U}(t,y')\big)^2}f(y')dy'\\
    &=kP.V.\int_{-1}^1\frac{(\pa_y+\pa_{y'})\mathcal{G}_k(y,y')+\frac{\pa_{y'}\mathcal{U}(t,y')-\pa_y\mathcal{U}(t,y)}{\mathcal{U}(t,y)-\mathcal{U}(t,y')}\mathcal{G}_k(y,y')}{2i\big(\mathcal{U}(t,y)-\mathcal{U}(t,y')\big)}f(y')dy'.
\end{align*}
Taking $\mathcal{G}_k(y,y')=G_k(y,y')$ and $\nu^{\f13}|k|^{\f23}G_k^{\rm (e)}(y,y')$, respectively, we find that the kernels of $\mathfrak{H}_k[f]$ and $\mathfrak{H}_k^{\rm(e)}[f]$ are those given in \eqref{comm-H-He}. We now proceed to estimate the bounds of the commutators. By the definitions of $G_k(y,y')$, $G_k^{\rm (e)}(y,y')$, it holds
\begin{align*}
  &  \|H_k(y,y')\|_{L_y^{\infty}L_{y'}^1\cap L_{y'}^{\infty}L_y^1}\lesssim \sup_y\int_{-1}^1\frac{\sinh\big(k(y+y')\big)}{\sinh (2k)}dy'+\sup_{y'}\int_{-1}^1\frac{\sinh\big(k(y+y')\big)}{\sinh (2k)}dy\\
    &\qquad\qquad\qquad\qquad\quad\qquad+\nu^{\f13}\|G_k\|_{L_y^{\infty}L_{y'}^1\cap L_{y'}^{\infty}L_y^1}\lesssim|k|^{-1},\\
    &\|\pa_yH_k(y,y')\|_{L_{y,y'}^{\infty}}+\|\pa_{y'}H_k(y,y')\|_{L_{y,y'}^{\infty}}\lesssim \Big\|\frac{\pa_y(\sinh\big(k(y+y')\big))}{\sinh (2k)}\Big\|_{L_{y,y'}^{\infty}}+\Big\|\frac{\pa_{y'}(\sinh\big(k(y+y')\big))}{\sinh (2k)}\Big\|_{L_{y,y'}^{\infty}}\\
    &\quad+\Big\|\frac{\pa_{y'}\mathcal{U}(t,y')-\pa_y\mathcal{U}(t,y)}{\mathcal{U}(t,y)-\mathcal{U}(t,y')}\pa_yG_k(y,y')\Big\|_{L_{y,y'}^{\infty}}+\Big\|\frac{\pa_{y'}\mathcal{U}(t,y')-\pa_y\mathcal{U}(t,y)}{\mathcal{U}(t,y)-\mathcal{U}(t,y')}\pa_{y'}G_k(y,y')\Big\|_{L_{y,y'}^{\infty}}\\
    &\quad +\Big\|\frac{-\pa_y^2\mathcal{U}(t,y)\big(\mathcal{U}(t,y)-\mathcal{U}(t,y')\big)+\pa_y\mathcal{U}(t,y)\big(\pa_y\mathcal{U}(t,y)-\pa_{y'}\mathcal{U}(t,y')\big)}{\big(\mathcal{U}(t,y)-\mathcal{U}(t,y')\big)^2}G_k(y,y')\Big\|_{L_{y,y'}^{\infty}}\\
    &\quad +\Big\|\frac{\pa_{y'}^2\mathcal{U}(t,y')\big(\mathcal{U}(t,y)-\mathcal{U}(t,y')\big)+\pa_{y'}\mathcal{U}(t,y')\big(\pa_{y'}\mathcal{U}(t,y')-\pa_{y}\mathcal{U}(t,y)\big)}{\big(\mathcal{U}(t,y)-\mathcal{U}(t,y')\big)^2}G_k(y,y')\Big\|_{L_{y,y'}^{\infty}}\lesssim |k|,\\
&    \|kH_k(y,y)\|_{L^{\infty}}\lesssim |k|.
\end{align*}
Then Lemma \ref{Lem-S} gives
\begin{equation*}
\begin{aligned}
    \|\mathfrak{H}_k\|_{L^2\to L^2}&\lesssim \ep_s ^{-1}|k| \|H_k(y,y')\|_{L_y^{\infty}L_{y'}^1\cap L_{y'}^{\infty}L_y^1}+\ep_s |k|\|\na_{y,y'}H_k(y,y')\|_{L_{y,y'}^{\infty}}+\|kH_k(y,y)\|_{L^{\infty}}\\
    &\lesssim \ep_s ^{-1}|k| |k|^{-1}+\ep_s |k||k|+|k|\lesssim|k|,
    \end{aligned}
\end{equation*}
where we take $\ep_s=|k|^{-1}$.

Similarly, according to the definition of $\w G{\rm e}_k$, we can get
\begin{align*}
 & \|H_k^{\rm (e)}(y,y')\|_{L_y^{\infty}L_{y'}^1\cap L_{y'}^{\infty}L_y^1} \lesssim|k|^{-1},\\
    &\|\pa_yH_k^{\rm (e)}(y,y')\|_{L_{y,y'}^{\infty}}+\|\pa_{y'}H^{\rm (e)}_k(y,y')\|_{L_{y,y'}^{\infty}}\lesssim \nu^{-\f23}|k|^{-\f13},\qquad    \|kH_k^{\rm (e)}(y,y)\|_{L^{\infty}}\lesssim \nu^{-\f13}|k|^{\f13}.
\end{align*}
Then Lemma \ref{Lem-S} gives
\begin{equation*}
\begin{aligned}
   \nu^{-\f13}|k|^{-\f23} \|\mathfrak{H}^{\rm (e)}_k\|_{L^2\to L^2}\lesssim \ep_s ^{-1}|k| |k|^{-1}+\ep_s |k|\nu^{-\f23}|k|^{-\f13}+\nu^{-\f13}|k|^{\f13}\lesssim \nu^{-\f13}|k|^{\f13},
    \end{aligned}
\end{equation*}
where we take $\ep_s=\nu^{\f13}|k|^{-\f13}$.
\end{proof}
The following lemma captures the commutators of $\pa_t$ with $\mathfrak{J}_k$ and $\mathfrak{J}_k^{\rm (e)}$.
\begin{lemma}\label{lem-s-t}
The commutators $[\pa_t,\mathfrak{J}_k]$ and $[\pa_t,\mathfrak{J}_k^{(\e)}]$ have the following form
\begin{align*}
&[\pa_t,\mathfrak{J}_k][f]=kP.V.\int_{-1}^1\frac{\wt{H}_k(y,y')}{2i\big(\mathcal{U}(t,y)-\mathcal{U}(t,y')\big)}f(y')dy' , \\
&[\pa_t,\mathfrak{J}_k^{(\e)}][f]=kP.V.\int_{-1}^1\frac{\wt{H}^{(\e)}_k(y,y')}{2i\big(\mathcal{U}(t,y)-\mathcal{U}(t,y')\big)}f(y')dy',
\end{align*}
where
\begin{align}\label{wt_Hk_Hke}\begin{split}
    &\wt{H}_k(y,y'):=  \frac{-\nu G_k(y,y')\big(\pa_y^2\mathcal{U}(t,y)-\pa_{y'}^2\mathcal{U}(t,y')\big)}{\mathcal{U}(t,y)-\mathcal{U}(t,y')},\\
    &\wt{H}_k^{(\e)}(y,y'):=  \frac{-\nu^{\f43}|k|^{\f23}G _k(y,y')\big(\pa_y^2\mathcal{U}(t,y)-\pa_{y'}^2\mathcal{U}(t,y')\big)}{\mathcal{U}(t,y)-\mathcal{U}(t,y')}.
    \end{split}
\end{align}
And there holds
\begin{align*}
    \|[\pa_t,\mathfrak{J}_k]\|_{L^2\to L^2}\le C \nu^{\f43},\qquad \|[\pa_t,\mathfrak{J}_k^{(\e)}]\|_{L^2\to L^2}\le C \nu^{\f43},
\end{align*}
where $C>0$ is a constant independent of $\nu$ and $k$.
\end{lemma}
\begin{proof}
    For $\mathfrak{S}_k$ \eqref{general-SIO}, we consider the commutator $[\pa_t,\mathfrak{S}_k]$. A direct calculation and the relation $\pa_t\mathcal{U}-\nu\pa_y^2\mathcal{U}=0$ lead to
\begin{align*}
    [\pa_t,\mathfrak{S}_k][f](t,y)&= kP.V.\int_{-1}^1\frac{-\mathcal{G}_k(y,y')\big(\pa_t\mathcal{U}(t,y)-\pa_t\mathcal{U}(t,y')\big)}{2i\big(\mathcal{U}(t,y)-\mathcal{U}(t,y')\big)^2}f(t,y')dy'\\
    &=  kP.V.\int_{-1}^1\frac{-\nu\mathcal{G}_k(y,y')\big(\pa_y^2\mathcal{U}(t,y)-\pa_{y'}^2\mathcal{U}(t,y')\big)}{2i\big(\mathcal{U}(t,y)-\mathcal{U}(t,y')\big)^2}f(t,y')dy',
\end{align*}
which is a singular integral operator generated by 
$$\wt{\mathfrak{H}}_k=\frac{-\nu\mathcal{G}_k(y,y')\big(\pa_y^2\mathcal{U}(t,y)-\pa_{y'}^2\mathcal{U}(t,y')\big)}{\mathcal{U}(t,y)-\mathcal{U}(t,y')}.$$
Taking $\mathcal{G}_k(y,y')=G_k(y,y')$ and $\nu^{\f13}|k|^{\f23}G_k^{\rm (e)}(y,y')$, we observe that the kernels correspond to $[\pa_t,\mathfrak{J}_k][f]$ and $[\pa_t,\mathfrak{J}_k^{(\e)}][f]$ are $\wt{H}_k, \wt{H}_k^{(\e)}$ \eqref{wt_Hk_Hke}. 
According to the definition of $G_k(y,y')$ and $\mathcal{U}(t,y)$, we can obtain
\begin{align*}
    &\|\wt{H}_k(y,y')\|_{L_{y}^{\infty}L_{y'}^1\cap L_{y'}^{\infty}L_y^1}\lesssim \nu^{\f43}|k|^{-2},\quad \|\na_{y,y'}\wt{H}_k(y,y')\|_{L_{y,y'}^{\infty}}\lesssim \nu^{\f43},\quad
        \|k\wt{H}_k(y,y)\|_{L^{\infty}}\lesssim \nu^{\f43},\\
        &\|\wt{H}_k^{(\e)}(y,y')\|_{L_y^{\infty}L_{y'}^1\cap L_{y'}^{\infty}L_y^1}\lesssim \nu^{\f53}|k|^{-\f43},\quad
     \|\na_{y,y'}\wt{H}_k^{(\e)}(y,y')\|_{L_{y,y'}^{\infty}}\lesssim \nu|k|^{-\f23},\quad
     \|k\wt{H}_k^{(\e)}(y,y)\|_{L^{\infty}}\lesssim \nu^{\f43}.
\end{align*}
Therefore, applying Lemma \ref{Lem-S} with $\ep_s=|k|^{-1}$ and $\nu^{\f13}|k|^{-\f13}$, respectively, yields that
\begin{align*}
    \|[\pa_t,\mathfrak{J}_k]\|_{L^2\to L^2}\lesssim \nu^{\f43},\qquad \|[\pa_t,\mathfrak{J}_k^{(\e)}]\|_{L^2\to L^2}\lesssim \nu^{\f43}.
\end{align*}
Hence, we finish the proof.
\end{proof}
Next, we display the following symmetry properties.
\begin{lemma}
\label{lem-symm}
    For any $f,g\in L^2$, $\mathcal{G}_k(y,y')=\overline{\mathcal{G}_k(y',y)}$, there holds
    \begin{align*}
        \overline{\mathfrak{S}_k[f]}=-\mathfrak{S}_k[\overline{f}],\qquad
        \langle \mathfrak{S}_k[f],g\rangle=\langle f,\mathfrak{S}_k[g]\rangle,
    \end{align*}
    which implies $\mathfrak{S}_k=\mathfrak{S}_k^*$.
\end{lemma}
\begin{proof}
    The first identity follows directly by the definition \eqref{general-SIO}. For the second one, using the first identity and the Fubini's theorem, we have
    \begin{align*}
    \langle \mathfrak{S}_k[f],g\rangle&=\int_{-1}^1 \mathfrak{S}_k[f]\overline{g}dy=\int_{-1}^1kP.V.\int_{-1}^1\frac{\mathcal{G}_k(y,y')}{2i\big(\mathcal{U}(t,y)-\mathcal{U}(t,y')\big)}f(y')dy'\overline{g(y)}dy\\
    &=k P.V. \int_{-1}^1\int_{-1}^1f(y')\frac{-\mathcal{G}_k(y,y')}{2i\big(\mathcal{U}(t,y')-\mathcal{U}(t,y)\big)}\overline{g(y)}dy dy'\\
    &= \int_{-1}^1f(y')\overline{ k P.V.\int_{-1}^1\frac{\mathcal{G}_k(y',y)}{2i\big(\mathcal{U}(t,y')-\mathcal{U}(t,y)\big)}g(y)dy} dy'=\langle f,\mathfrak{S}_k[g]\rangle,
    \end{align*}
    which finishes the proof.
\end{proof}
\begin{remark}
    Taking $\mathcal{G}_k(y,y')=G_k(y,y')$, $\nu^{\f13}|k|^{\f23}G_k^{(\e)}(y,y')$, and $e^{-ikt(y-y')}G_k(y,y')$ $(\mathcal{U}(t,y)=y)$ in $\mathfrak{S}_k$, respectively. For any $f,g\in L^2$, Lemma \ref{lem-symm} gives
    \begin{equation}\label{eq-symm}
         \langle \mathfrak{J}_k[f],g\rangle=\langle f,\mathfrak{J}_k[g]\rangle,\qquad   \langle \mathfrak{J}^{(\e)}_k[f],g\rangle=\langle f,\mathfrak{J}^{(\e)}_k[g]\rangle,\qquad  \langle \wt{\mathfrak{J}}_0[f],g\rangle=\langle f,\wt{\mathfrak{J}}_0[g]\rangle.
    \end{equation}
\end{remark}
The following lemma establishes a useful identity that will be used in the energy estimates.
\begin{lemma}\label{lem-id-Re}
    For any $f$ in $L^2$, $\mathcal{G}_k(y,y')=\overline{\mathcal{G}_k(y',y)}$ it holds
\begin{align*}
  2{\rm Re}  \langle f,\mathfrak{S}_k[-ik\mathcal{U}f]\rangle= \f{k^2}{2}{\rm Re}\big\langle \int_{-1}^1\mathcal{G}_k(y,y')f(y')dy',f(y)\big\rangle.
\end{align*}
\end{lemma}
\begin{proof}
    By the definition of $\mathfrak{S}_k$ \eqref{general-SIO}, we get
    \begin{equation}\label{eq-inner}
    \begin{aligned}
      {\rm Re}\big  \langle f,\mathfrak{S}_k[-ik\mathcal{U}f]\big\rangle={\rm Re}\int_{-1}^1k P.V.\int_{-1}^1\f{\mathcal{G}_k(y,y')\big(-ik\mathcal{U}(t,y')f(y')\big)}{2i\big(\mathcal{U}(t,y)-\mathcal{U}(t,y')\big)}dy'\overline{f(y)}dy.
    \end{aligned}
    \end{equation}
    Additionally, by the symmetry property of $\mathfrak{S}_k$ in Lemma \ref{lem-symm}, there holds
    \begin{align*}
    &{\rm Re} \langle f,\mathfrak{S}_k[-ik\mathcal{U}f]\rangle=    \Re \langle \mathfrak{S}_k[f],-ik\mathcal{U}f\rangle\\
    &=\Re\int_{-1}^1kP.V.\int_{-1}^1\f{\mathcal{G}_k(y,y')f(y')}{2i\big(\mathcal{U}(t,y)-\mathcal{U}(t,y')\big)}dy'\big(ik\mathcal{U}(t,y)\big)\overline{f(y)}dy.
    \end{align*}
    Then, combined with \eqref{eq-inner}, we obtain
    \begin{align*}
      & 2 {\rm Re} \langle f,\mathfrak{S}_k[-ik\mathcal{U}f]\rangle={\rm Re}\int_{-1}^1k P.V.\int_{-1}^1\f{\mathcal{G}_k(y,y')\big(-ik\mathcal{U}(t,y')f(y')\big)}{2i\big(\mathcal{U}(t,y)-\mathcal{U}(t,y')\big)}dy'\overline{f(y)}dy\\
       &\qquad +\Re\int_{-1}^1kP.V.\int_{-1}^1\f{\mathcal{G}_k(y,y')f(y')}{2i\big(\mathcal{U}(t,y)-\mathcal{U}(t,y')\big)}dy'\big(ik\mathcal{U}(t,y)\big)\overline{f(y)}dy\\
       &=\Re\int_{-1}^1\int_{-1}^1\f{k^2}{2}\mathcal{G}_k(y,y')f(y')dy'\overline{f(y)}dy=\f{k^2}{2}{\rm Re}\big\langle \int_{-1}^1\mathcal{G}_k(y,y')f(y')dy',f(y)\big\rangle.
    \end{align*}
\end{proof}
\begin{lemma}\label{Lem-SIO}
For any $f_k\in L^2(-1,1)$, $f_k(\pm1)=0$, it holds
\begin{align}\label{SIO_ID}
    &2{\rm Re}  \big\langle f_k,\mathfrak{J}_k[-ik\mathcal{U}f]\big\rangle=-\f{k^2}{2}\|\na_k(\Delta_k)^{-1}f_k\|_{L^2}^2,\\
    &2{\rm Re} \big \langle f,\mathfrak{J}^{(\e)}_k[-ik\mathcal{U}f]\big\rangle\le -\nu^{\f13}|k|^{\f23}\|f_k\|_{L^2}^2+C\nu \|\na_kf_k\|_{L^2}^2.\label{SIO_ED}
\end{align}
For any $f_0\in L^2(-1,1)$, it holds that
\begin{align}\label{SIO_Cascade}
    2\Re\big\langle [\pa_t,\wt{\mathfrak{J}}_0][f_0],f_0\big\rangle=-\sum_{k\in\mathbb Z\backslash\{0\}}\f{1}{|k|}\|\na_k^L(\D_k^L)^{-1}f_0\|_{L^2}^2,
\end{align}
where $\na_k^L=(ik,\pa_y-ikt)$ and $\D_k^Lg=(-k^2+(\pa_y-ikt)^2)g$.
\end{lemma}
\begin{remark}
    Here, the estimate \eqref{SIO_ID} provides space-time type inviscid damping, and the bound \eqref{SIO_ED} provides enhanced dissipation. Finally, the $\wt {\mathfrak J}_0$ describes the cascade-type growth. 
\end{remark}
\begin{proof}
     Taking $\mathcal{G}_k(y,y')=G_k(y,y')$ and $\nu^{\f13}|k|^{\f23}G_k^{(\e)}(y,y')$ in $\mathfrak{S}_k$, respectively. Since $f_k(\pm1)=0$, through integration by parts, applying Lemma \ref{lem-id-Re}, together with the properties of $G_k$ and $\w G{\e}_k$, we get
    \begin{equation*}
    \begin{aligned}
        & 2{\rm Re}  \big\langle f_k,\mathfrak{J}_k[-ik\mathcal{U}f]\big\rangle= \f{k^2}{2}{\rm Re}\Big\langle \int_{-1}^1G_k(y,y')f_k(y')dy',f_k(y)\Big\rangle\\
        &\quad =\f{k^2}{2}{\rm Re}\big\langle (\D_k)^{-1}f_k,f_k(y)\big\rangle=-\f{k^2}{2}\|\na_k(\Delta_k)^{-1}f_k\|_{L^2}^2,
        \end{aligned}
        \end{equation*}
        \begin{equation}\label{eq-inner-J-e}
        \begin{aligned}
         &2{\rm Re} \big \langle f_k,\mathfrak{J}^{(\e)}_k[-ik\mathcal{U}f]\big\rangle= \f{\nu^{\f13}|k|^{\f83}}{2}{\rm Re}\Big\langle \int_{-1}^1G^{(\e)}_k(y,y')f_k(y')dy',f_k(y)\Big\rangle\\
         &\quad  = -\f{\nu|k|^4}{2}\big\|\pa_yh_k\big\|_{L^2}^2-\f{\nu^{\f13}|k|^{\f83}}{2}\big\||k|h_k\big\|_{L^2}^2.
         \end{aligned}
    \end{equation}
  Here, $(\nu^{\f23}|k|^{\f43}\pa_y^2-k^2)h_k=f_k$.  Additionally, since $f_k(\pm1)=0$, through integration by parts, it yields
    \begin{align*}
      &  \nu^{\f13}|k|^{\f23}\|f_k\|_{L^2}^2=\nu^{\f13}|k|^{\f23}\big\langle (\nu^{\f23}|k|^{\f43}\pa_y^2-k^2)h_k,f_k\big\rangle\\
      &\le \nu k^2\|\pa_yh_k\|_{L^2}\|\pa_yf_k\|_{L^2}+\nu^{\f13}|k|^{\f83}\|h_k\|_{L^2}\|f_k\|_{L^2}\\
      &\le \nu k^4\|\pa_yh_k\|_{L^2}^2+\nu\|\pa_yf_k\|_{L^2}^2+\nu^{\f13}|k|^{\f83}\||k|h_k\|_{L^2}^2+\f{1}{2}\nu^{\f13}|k|^{\f23}\|f_k\|_{L^2}^2.
    \end{align*}
    Hence, combined with \eqref{eq-inner-J-e}, it follows
    \begin{align*}
        2{\rm Re} \big \langle f_k,\mathfrak{J}^{(\e)}_k[-ik\mathcal{U}f]\big\rangle+\nu^{\f13}|k|^{\f23}\|f_k\|_{L^2}^2\le C\nu\|\pa_yf_k\|_{L^2}^2\le C\nu\|\na_kf_k\|_{L^2}^2.
    \end{align*}
    Finally, for the $\wt{\mathfrak{J}}_0$-estimate, a direct calculation gives
    \begin{align*}
         &  [\pa_t,\wt{\mathfrak{J}}_0][f_0](t,y)=\pa_t\big(\wt{\mathfrak{J}}_0[f_0]\big)-\wt{\mathfrak{J}}_0[\pa_tf_0]\\
         &=\f12\sum_{k\in\mathbb Z\backslash\{0\}}\f{1}{|k|}\int_{-1}^1e^{ikt(y-y')}G_k(y,y')f_0(y')dy'=\f12\sum_{k\in\mathbb Z\backslash\{0\}}|k|^{-1}(\D_k^L)^{-1}f_0.
    \end{align*}
    Then, since $G_k(\pm1,y')=0$ for any $y'\in[-1,1]$, through integration by parts, we get
    \begin{align*}
         2\Re \big\langle [\pa_t,\wt{\mathfrak{J}}_0][f_0],f_0\big\rangle= \Re \big\langle \sum_{k\in\mathbb Z\backslash\{0\}}|k|^{-1}(\D_k^L)^{-1}f_0,\D_k^L(\D_k^L)^{-1}f_0\big\rangle=-\sum_{k\in\mathbb Z\backslash\{0\}}|k|^{-1}\|\na_k^L(\D_k^L)^{-1}f_0\|_{L^2}^2.
    \end{align*}
    At this end, we finish the proof.
\end{proof}

\noindent \textbf{Acknowledgements:} S. H. acknowledges support from NSF DMS-2406293.

\noindent \textbf{Conflict of Interest:} We confirm that we do not have any conflict of interest.  

\noindent \textbf{Data Availability:} This article has no associated datasets.

\noindent \textbf{Author Contributions Statement: }  S. H., B. N., and W. Z. have the same contribution to the research.   
\bibliographystyle{abbrv}
\bibliography{ref.bib}
\end{document}